\documentclass{amsart}
\usepackage{amscd,amssymb,amsmath,amsfonts}

\newcommand\supp{\operatorname{supp}}
\newcommand\rank{\operatorname{rank}}

\newcommand\gsifp{{G'_{F,\Si}}}
\newcommand\gsif{{G_{F,\Si}}}
\newcommand\gsip{{G'_\Si}}
\newcommand\gsi{{G_\Si}}
\newcommand\gsin{{G^\N_\Si}}
\newcommand\gsinz{{G^{0,\N}_\Si}}
\newcommand\gsinf{{G^{\N}_{\Si,F}}}

\newcommand\pa{{\textsc{pa}}}
\newcommand\CH{{\mathcal{C}\mathcal{H}}}
\newcommand\bs{{\beta_{\Si}}}
\newcommand\bsp{{\beta'_{\Si}}}
\newcommand\bns{{\beta_{\N\times\Si}}}
\newcommand\bnsy{{Y^\N_\Si}}
\newcommand\bnsyn{{Y^\N_{n,\Si}}}
\newcommand\bnsynm{{Y^\N_{n-1,\Si}}}
\newcommand\bnsyp{{{Y'_\Si}^\N}}
\newcommand\bnsz{{\beta^0_{\N\times\Si}}}
\renewcommand\lq{\leqslant}
\newcommand\gq{\geqslant}
\newcommand\E{\mathcal{E}}
\newcommand\II{\mathcal{I}}
\newcommand\X{\mathcal{X}}
\newcommand\XS{{\mathcal{X}_\Si}}
\newcommand\CC{\mathcal{C}}
\renewcommand\SS{\mathcal{S}}
\newcommand\A{\mathcal{A}}
\newcommand\Ap{{\mathcal{A}^\oplus}}

\newcommand\sta{{\mathcal{P}\!\mathcal{A}}}
\newcommand\ac{{\mathcal{A}_{C_0(\Si)}}}
\newcommand\aci{{\mathcal{A}^\infty_{C_0(\Si)}}}
\newcommand\bci{{\mathcal{B}^\infty_{C_0(\Si)}}}
\newcommand\intdt{\stackrel{\scriptscriptstyle{o}}{\Delta}} 

\newcommand\diag{\operatorname{diag}}
\newcommand\Ind{\operatorname{I}}

\newcommand\I{\operatorname{I}_F^\Ga}
\newcommand\ds{\displaystyle}

\newcommand\C{\mathbb C}
\newcommand\K{\mathcal{K}}
\newcommand\F{\mathcal{F}}
\newcommand\Si{\Sigma}
\renewcommand\l{\ell}
\newcommand\JR{\mathcal{J}_{\Ga}^{red}}

\newcommand\JJ{\mathcal{J}}
\renewcommand\H{\mathcal{H}}
\newcommand\de{\delta}
\newcommand\N{\mathbb N}

\newcommand\Id{\mathcal{I}d}

\newcommand\B{\mathcal B}
\newcommand\R{\mathbb R}
\newcommand\st{\text{ such that }}
\newcommand\Z{\mathbb Z}
\newcommand\T{\mathcal{T}}
\newcommand\TT{\mathcal{T}}
\renewcommand\O{\mathcal{O}}
\newcommand\MM{\mathcal{M}}
\renewcommand\L{\mathcal{L}}
\newcommand\D{\mathcal D}
\newcommand\DD{\mathcal D}
\newcommand\G{\mathcal{G}}
\newcommand\ts{{\otimes}}
\newcommand\rt{{\rtimes}}
\newcommand\rtr{{\rtimes_{red}}}

\newcommand\si{\sigma}
\newcommand\la{\lambda}
\newcommand{\aeq}{\stackrel{(\lambda,h)}{\sim}}

\newcommand\w{\omega}
\newcommand\Ga{{\Gamma}}

\newcommand\ga{{\gamma}}
\newcommand\lto{{\longrightarrow}}
\newcommand\defi{{\stackrel{\text{def}}{=\!=}}}

\newcommand\AG{{A\rtimes_{red}\Ga}}
\newcommand\M{{M}}
\newcommand\U{\operatorname{U}}
\renewcommand\P{\operatorname{P}}
\newcommand\ka{\kappa}
\newcommand\ue{\operatorname{U}_n^{\varepsilon,r}}
\newcommand\pe{\operatorname{P}_n^{\varepsilon,r}}
\newcommand\eps{\varepsilon}
\newcommand\erp{$\eps$-$r$-projection }

\theoremstyle{plain}
\newtheorem{theorem}{Theorem}[section]
\newtheorem{proposition}[theorem]{Proposition}
\newtheorem{corollary}[theorem]{Corollary}

\newtheorem{lemma}[theorem]{Lemma}
\newtheorem{definition}[theorem]{Definition}
\theoremstyle{definition}

\newtheorem{example}[theorem]{Example}
\newtheorem{notation}[theorem]{Notation}
\begin{document}
\title[Persistence approximation property and controlled operator $K$-th.]{Persistence approximation property and controlled operator $K$-theory}

 \author[H. Oyono-Oyono]{Herv\'e Oyono-Oyono}
 \address{Universit\'e de Lorraine, Metz , France}
 \email{herve.oyono@math.cnrs.fr}
\author[G. Yu]{Guoliang Yu }
 \address{Texas A\&M University, USA}
 \email{guoliangyu@math.tamu.edu}
\thanks{Oyono-Oyono is partially supported by the ANR ``Kind'' and Yu is partially supported by  a grant from the
U.S. National Science Foundation.}


\begin{abstract}
In this paper, we introduce and study the persistent approximation property for quantitative $K$-theory of filtered $C^*$-algebras. In the case of crossed product $C^*$-algebras, the persistent approximation property follows from the Baum-Connes conjecture with coefficients. We also discuss some applications  of the quantitative $K$-theory to the Novikov conjecture. 
\end{abstract}
\maketitle

\begin{flushleft}{\it Keywords: Baum-Connes Conjecture, Coarse Geometry, 
  Novikov Conjecture, Operator Algebra $K$-theory, Roe Algebras}

\medskip

{\it 2000 Mathematics Subject Classification: 19K35,46L80,58J22}
\end{flushleft}

\tableofcontents
\section{Introduction}
The idea of quantitative operator $K$-theory was first introduced in \cite{y2}  to study the Novikov conjecture for groups with finite asymptotic dimension.
In \cite{oy2}, the authors introduced a  general quantitative $K$-theory for filtered $C^*$-algebras. Examples of filtered  $C^*$-algebras include group $C^*$-algebras, crossed product $C^*$-algebras, Roe algebras, foliation $C^*$-algebras and finitely generated $C^*$-algebras.  
For a $C^*$-algebra $A$ with a filtration,  the $K$-theory of A, $K_*(A)$ is the limit of the quantitative $K$-theory groups $K_*^{\epsilon, r}(A)$ when $r$ goes to infinity. The crucial point is that  quantitative $K$-theory is often more computable using certain controlled exact sequences (e.g. see \cite{y2}  and \cite{oy2}).
The study of  $K$-theory for the Roe algebra  can be reduced to  that of quantitative $K$-theory for the Roe algebra associated to finite metric spaces, which in essence is a finite dimensional linear algebra problem.

\medskip

The main purpose of this paper is to introduce and study the persistent approximation property for quantitative $K$-theory of filtered $C^*$-algebras. Roughly  speaking, the persistent approximation property means that the convergence of $K_*^{\epsilon, r}(A)$ to $K_*(A)$  is uniform. More precisely, we say that the filtered $C^*$-algebra $A$ has persistent approximation property if for each $\eps$ in $(0,1/4)$  and  $r>0$, there exists $r'\gq r$ and $\eps'$ in $[\eps,1.4)$ such that an element from $K_*^{\epsilon, r}(A)$ is zero in $K_* (A)$ if and only if  it is zero in  $K_*^{\epsilon', r'}(A)$. The main motivation to study the persistent approximation property is that it provides an effective way of approximating $K$-theory with quantitative $K$-theory. In the case of crossed  product $C^*$-algebras, the Baum-Connes conjecture with coefficients provides many examples that satisfy the persistent approximation property.     It turns out that this property provides geometrical obstruction for the Baum-Connes conjecture. In order to study this obstruction in full generality, we consider  the persistence approximation property for  filtered $C^*$-algebra $A\ts \K(\ell^2(\Si))$, where $A$ is a $C^*$-algebra and $\Si$ is a discrete metric space with bounded geometry. For this purpose, we introduce a bunch of  quantitative local assembly maps valued in the quantitative $K$-theory for $A\ts \K(\ell^2(\Si))$ and we set quantitative statements, analogue in this geometric setting to the quantitative statements  of \cite[Section 6.2]{oy2} for the quantitative Baum-Connes assembly maps. We also show that  if  these statements hold uniformly for the family of  finite subsets of a discrete metric space $\Si$ with bounded geometry,  the coarse Baum-Connes conjecture for $\Si$ is satisfies. In particular, in the case of a finitely generated group $\Ga$ provided with the metric arising from any word length, then these  uniform statements for finite metric subsets of $\Ga$ implies the Novikov conjecture for $\Ga$ on homotopy invariance of higher signatures.  We point out  that in this case, these statements  reduce to finite dimension problems in linear algebra and analysis. 

\medskip

The paper is organized  as follows. In section \ref{sec-survey}, we review the main results of \cite{oy2} concerning quantitative $K$-theory. In section \ref{section-persistence}, we introduce the persistence approximation property. We prove that if $\Ga$ is a finitely generated group that satisfies the Baum-Connes conjecture with coefficients and  which admits a cocompact universal example for proper actions, then for any $\Ga$-$C^*$-algebra $A$, the reduced crossed product $A\rtimes_r\Ga$ satisfies the persistence approximation property. In the special case of the action of the group $\Ga$ on $C_0(\Ga)$ by translation, we get a canonical   identification between $C_0(\Ga)\rtimes\Ga$ and $\K(\ell^2(\Si))$ that preserves the filtration structure. Hence, the persistence approximation property can be stated in a completely geometrical way. This leads us to consider this property for the  algebra $A\ts \K(\ell^2(\Si))$, where $A$ is a $C^*$-algebra and $\Si$ is a proper discrete metric space, with filtration structure induced  by the metric of $\Si$.  In section \ref{section-coarse-geometry}, following the idea  of the Baum-Connes conjecture, we construct in order  to compute the quantitative $K$-theory groups for
$A\ts \K(\ell^2(\Si))$ a bunch of quantitative assembly maps $\nu_{\Si,A,*}^{\eps,r,d}$. If view of the  proof of the persistence approximation property in the crossed product algebras  case,   we introduce  a geometrical assembly map $\nu_{\Si,A,*}^{\infty}$ (which plays   the role of the Baum-Connes assembly map with relevant coefficients). Following the route of \cite{sty}, we show that the target of these geometric assembly maps is indeed the $K$-theory of the crossed product algebra of an appropriate $C^*$-algebra $\ac$  by the groupoid $G_\Si$ associated in \cite{sty} to the coarse structure of $\Si$. In section \ref{section-BC-for-GSI}, we study the Baum-Connes assembly map for the pair $(G_\Si,\ac)$ and we show that  the bijectivity of the  geometric assembly maps $\nu_{\Si,A,*}^{\infty}$ is equivalent to the Baum-Connes conjecture for  $(G_\Si,\ac)$. We set in the geometric setting the analogue of the quantitative statements of \cite[Section 6.2]{oy2} for the quantitative Baum-Connes assembly maps and we prove that these statements holds when $\Si$ coarsely embeds into a Hilbert space. We then apply this results to the  persistent approximation property for $A\ts \K(\ell^2(\Si))$. In particular, we prove it when $\Si$ coarsely embeds into a Hilbert space, under an assumption of coarse uniform contractibility. This condition is the analogue in the geometric setting of the existence of a cocompact universal example for proper actions and is satisfied for instance  for Gromov hyperbolic discrete metric spaces.
In section \ref{section-Novikov}, we show that for a discrete metric space with bounded geometry, if the quantitative statements of section \ref{section-BC-for-GSI} for $\nu_{F,A,*}^{\infty}$ holds uniformly when $F$ runs through finite subsets of $\Si$, then $\Si$ satisfies the coarse Baum-Connes conjecture.

\section{Survey on quantitative $K$-theory}\label{sec-survey}%

We gather this section with  the main results of \cite{oy2} concerning quantitative $K$-theory and that we shall use throughout  this paper. Quantitative $K$-theory was introduced to describe  propagation phenomena in higher  index theory for non-compact spaces. 
More generally,  we use the framework of filtered $C^*$-algebras to model the concept of propagation. 
\begin{definition}
A filtered $C^*$-algebra $A$ is a $C^*$-algebra equipped with a family
$(A_r)_{r>0}$ of  closed linear subspaces   indexed by positive numbers such that:
\begin{itemize}
\item $A_r\subset A_{r'}$ if $r\lq r'$;
\item $A_r$ is stable by involution;
\item $A_r\cdot A_{r'}\subset A_{r+r'}$;
\item the subalgebra $\ds\bigcup_{r>0}A_r$ is dense in $A$.
\end{itemize}
If $A$ is unital, we also require that the identity  $1$ is an element of $ A_r$ for every positive
number $r$. The elements of $A_r$ are said to have {\bf propagation $r$}.
\end{definition}
 Let $A$ and $A'$ be respectively  $C^*$-algebras filtered by
$(A_r)_{r>0}$ and  $(A'_r)_{r>0}$. A homomorphism of $C^*$
-algebras $\phi:A\lto A'$
is a {\bf filtered homomorphism} (or a {\bf homomorphism of  filtered $C^*$-algebras}) if  $\phi(A_r)\subset A'_{r}$ for any
positive number $r$.

\medskip

If $A$ is not unital, let us denote by $A^+$ its unitarization, i.e 
$$A^+=\{(x,\lambda);\,x\in A\,,\lambda\in \C\}$$  with the product $$(x,\lambda)(x',\lambda')=(xx'+\lambda x'+\lambda' x)$$ for all $(x,\lambda)$ and $(x',\lambda')$ in $A^+$. Then ${A^+}$ is filtered with
$${A^+_r}=\{(x,\lambda);\,x\in A^+_{r}\,,\lambda\in \C\}.$$ 
We also define $\rho_A:A^+\to\C;\, (x,\lambda)\mapsto \lambda$.

\subsection{Definition of quantitative $K$-theory}

Let $A$ be a unital filtered $C^*$-algebra. For any  positive
numbers $r$ and $\eps$, we call
\begin{itemize}
\item an element $u$ in $A$  a $\eps$-$r$-unitary if $u$
  belongs to $A_r$,  $\|u^*\cdot
  u-1\|<\eps$
and  $\|u\cdot u^*-1\|<\eps$. The set of $\eps$-$r$-unitaries on $A$ will be denoted by $\operatorname{U}^{\varepsilon,r}(A)$.
\item an element $p$ in $A$   a $\eps$-$r$-projection    if $p$
  belongs to $A_r$,
  $p=p^*$ and  $\|p^2-p\|<\eps$. The set of $\eps$-$r$-projections on $A$ will be denoted by $\operatorname{P}^{\varepsilon,r}(A)$.
\end{itemize} 
Notice that a $\eps$-$r$-unitary is invertible, and that if $p$ is an \erp in $A$, then it has a spectral gap around $1/2$ and then gives rise by functional calculus to a  projection $\ka_{0}(p)$  in  $A$ such that
 $\|p-\ka_{0}(p)\|< 2\eps$.
 
 
 For any  $n$ integer, we set  $\ue(A)=\operatorname{U}^{\varepsilon,r}(M_n(A))$ and
$\pe(A)=\operatorname{P}^{\varepsilon,r}(M_n(A))$.
For any  unital filtered $C^*$-algebra $A$, any
 positive  numbers $\eps$ and $r$ and  any positive integer $n$, we consider inclusions
$$\P_n^{\eps,r}(A)\hookrightarrow \P_{n+1}^{\eps,r}(A);\,p\mapsto
\begin{pmatrix}p&0\\0&0\end{pmatrix}$$ and
$$\U_n^{\eps,r}(A)\hookrightarrow \U_{n+1}^{\eps,r}(A);\,u\mapsto
\begin{pmatrix}u&0\\0&1\end{pmatrix}.$$ This allows us to  define
 $$\U_{\infty}^{\eps,r}(A)=\bigcup_{n\in\N}\ue(A)$$ and
$$\P_{\infty}^{\eps,r}(A)=\bigcup_{n\in\N}\pe(A).$$

For a unital filtered $C^*$-algebra $A$, we define the
following
equivalence relations on $\P_\infty^{\eps,r}(A)\times\N$ and on  $\U_\infty^{\eps,r}(A)$:
\begin{itemize}
\item if $p$ and $q$ are elements of $\P_\infty^{\eps,r}(A)$, $l$ and
  $l'$ are positive integers, $(p,l)\sim(q,l')$ if there exists a
  positive integer $k$ and an element $h$ of
  $\P_\infty^{\eps,r}(A[0,1])$ such that $h(0)=\diag(p,I_{k+l'})$
and $h(1)=\diag(q,I_{k+l})$.
\item if $u$ and $v$ are elements of $\U_\infty^{\eps,r}(A)$, $u\sim v$ if
  there exists an element $h$ of
  $\U_\infty^{3\eps,2r}(A[0,1])$ such that $h(0)=u$
and $h(1)=v$. Notice that we have changed slightly the definition of \cite{oy2}, in order to make 
$K_1{\eps,r}(A)$ into group (see \cite[Remark 1.15]{oy2}).
\end{itemize}

If $p$ is an  element of $\P_\infty^{\eps,r}(A)$ and  $l$ is an integer, we
denote by $[p,l]_{\eps,r}$ the equivalence class of $(p,l)$ modulo  $\sim$
and if $u$ is an element of $\U_\infty^{\eps,r}(A)$ we denote by
$[u]_{\eps,r}$ its  equivalence class  modulo  $\sim$.
\begin{definition} Let $r$ and $\eps$ be positive numbers with
  $\eps<1/4$.
We define:
\begin{enumerate}
\item $K_0^{\eps,r}(A)=\P_\infty^{\eps,r}(A)\times\N/\sim$ for $A$ unital and
$$K_0^{\eps,r}(A)=\{[p,l]_{\eps,r}\in \P^{\eps,r}({A^+})\times\N/\sim \st
\rank \kappa_0(\rho_{A}(p))=l\}$$ for $A$ non unital ($\kappa_0(\rho_{A}(p))$ being the spectral projection associated to $\rho_A(p)$);
\item $K_1^{\eps,r}(A)=\U_\infty^{\eps,r}({A^+})/\sim$, with $A=A^+$ if $A$ is already unital.
\end{enumerate}
\end{definition}

 Then $K_0^{\eps,r}(A)$ turns to be an abelian group  \cite[Lemma 1.15]{oy2} where 
 $$[p,l]_{\eps,r}+[p',l']_{\eps,r}=[\diag(p,p'),l+l']_{\eps,r}$$  for any  $[p,l]_{\eps,r}$ and $[p',l']_{\eps,r}$ in $K_0^{\eps,r}(A)$. According to \cite[Remark 1.15]{oy2},  $K_1^{\eps,r}(A)$ is 
 equipped with a structure of abelian group such that 
$$[u]_{\eps,r}+[u']_{\eps,r}=[\diag(u,v)]_{\eps,r},$$ for 
any  $[u]_{\eps,r}$ and $[u']_{\eps,r}$ in $K_1^{\eps,r}(A)$.
  
  \medskip
  
  Recall from \cite[corollaries 1.20 and 1.21]{oy2} that 
for any positive numbers  $r$ and $\eps$ with $\eps<1/4$, then
$$K_0^{\eps,r}(\C)\to\Z;\,[p,l]_{\eps,r}\mapsto \rank\kappa_0(p)-l$$
is an isomorphism and 
$K_1^{\eps,r}(\C)=\{0\}$.

 \medskip  
  
 We have for any filtered $C^*$-algebra $A$ and any  positive numbers
$r$, $r'$, $\eps$ and $\eps'$  with
  $\eps\lq\eps'<1/4$ and $r\lq r'$  natural group homomorphisms
\begin{itemize}
\item $\iota_0^{\eps,r}:K_0^{\eps,r}(A)\lto K_0(A);\,
[p,l]_{\eps,r}\mapsto [\kappa_0(p)]-[I_l]$ (where  $\kappa_0(p)$ is the spectral projection associated to $p$);
\item $\iota_1^{\eps,r}:K_1^{\eps,r}(A)\lto K_1(A);\,
  [u]_{\eps,r}\mapsto [u]$  ;
\item $\iota_*^{\eps,r}=\iota_0^{\eps,r}\oplus \iota_1^{\eps,r}$;
\item $\iota_0^{\eps,\eps',r,r'}:K_0^{\eps,r}(A)\lto K_0^{\eps',r'}(A);\,
[p,l]_{\eps,r}\mapsto [p,l]_{\eps',r'};$
\item $\iota_1^{\eps,\eps',r,r'}:K_1^{\eps,r}(A)\lto K_1^{\eps',r'}(A);\,
  [u]_{\eps,r}\mapsto [u]_{\eps',r'}$.
\item $\iota_*^{\eps,\eps',r,r'}=\iota_0^{\eps,\eps',r,r'}\oplus\iota_1^{\eps,\eps',r,r'}$
\end{itemize}
If some of the indices $r,r'$ or $\eps,\eps'$ are equal, we shall not
repeat it in $\iota_*^{\eps,\eps',r,r'}$.
The following result is a consequence of \cite[Remark 1.4]{oy2}.
\begin{proposition}\label{proposition-approximation}
Let $A=(A_r)_{r>0}$ be a filtered $C^*$-algebra.
\begin{enumerate}
\item For any $\eps$ in  $(0,1/4)$ and any $y$ in $K_*(A)$, there exist a positive number $r$ and an element $x$ in
$K_*^{\eps,r}(A)$ such that $\iota_{*}^{\eps,r}(x)=y$;
\item There exists a positive number $\lambda>1$ independent on $A$ such  that the following is satisfies:

\medskip

let $\eps$ be in $(0,1/4)$, let $r$ be a positive number and let $x$ and $x'$ be elements in $K_*^{\eps,r}(A)$ such that
$\iota_*^{\eps,r}(x)=\iota_*^{\eps,r}(x')$ in $K_*(A)$. Then there
exists a positive number  $r'$ with $r'>r$ such that
$\iota_*^{\eps,\lambda\eps,r,r'}(x)=\iota_*^{\eps,\lambda\eps,r,r'}(x')$ in
$K_*^{\lambda\eps,r'}(A)$.
\end{enumerate}
\end{proposition}

 If  $\phi:A\to B$ is a  homomorphism  filtered $C^*$-algebras, then since $\phi$ preserve $\eps$-$r$-projections
 and $\eps$-$r$-unitaries, it obviously induces  for any positive number $r$ and any
$\eps\in(0,1/4)$ a
group homomorphism $$\phi_*^{\eps,r}:K_*^{\eps,r}(A)\longrightarrow
K_*^{\eps,r}(B).$$  Moreover quantitative $K$-theory is homotopy invariant with respect to homotopies that preserves propagation 
\cite[Lemma 1.27]{oy2}.
There is also a quantitative version of Morita equivalence \cite[Proposition 1.29]{oy2}.
\begin{proposition}\label{prop-morita}
If $A$ is a filtered algebra and $\H$ is a separable Hilbert space, then the homomorphism
$$A\to \K(\H)\otimes A;\,a\mapsto \begin{pmatrix}a&&\\&0&\\&&\ddots\end{pmatrix}$$
induces a ($\Z_2$-graded) group isomorphism (the Morita equivalence)
$$\MM_A^{\eps,r}:K_*^{\eps,r}(A)\to K_*^{\eps,r}(\K(\H)\otimes A)$$
for any positive number $r$ and any
$\eps\in(0,1/4)$.
\end{proposition}

\subsection{Quantitative objects}
In order to study the functorial properties of quantitative $K$-theory, we introduce the concept of quantitative object.   
\begin{definition} A control pair  is a pair $(\lambda,h)$, where
\begin{itemize}
 \item $\lambda >1$;
\item  $h:(0,\frac{1}{4\lambda})\to (1,+\infty);\, \eps\mapsto h_\eps$  is a map such that there exists a non-increasing map
$g:(0,\frac{1}{4\lambda})\to (0,+\infty)$, with $h\lq g$.
\end{itemize}\end{definition}
 The set of control pairs is equipped with a partial order:
 $(\lambda,h)\lq (\lambda',h')$ if $\lambda\lq\lambda'$ and $h_\eps\lq h'_\eps$
for all $\eps\in (0,\frac{1}{4\lambda'})$.

\begin{definition}
 A quantitative object is a family   $\O=(O^{\eps,r})_{0<\eps<1/4,r>0}$ of abelian groups,
together with a family of 
group homomorphisms  $$\iota_{\O}^{\eps,\eps',r,r'}:O^{\eps,r}\to O^{\eps',r'}$$ for $0<\eps\lq \eps'<1/4$ and $0<r\leq r'$ such that
\begin{itemize}
 \item $\iota_{\O}^{\eps,\eps,r,r}=Id_{O^{\eps,r}}$ for any  $0<\eps<1/4$ and $r>0$;
 \item $\iota_{\O}^{\eps',\eps'',r',r''}\circ\iota_{\O}^{\eps,\eps',r,r'}=\iota_{\O}^{\eps,\eps'',r,r''}$ for any $0<\eps\lq \eps'\lq \eps''<1/4$ and $0<r\lq r'\lq r''$;
\item there exists a control pair $(\alpha,k)$ such that the following holds: 
for any $0<\eps<\frac{1}{4\alpha}$ and $r>0$ and any $x$ in 
$O^{\eps,r}$, there exists $x'$ in $O^{\alpha\eps,k_\eps r}$ satisfying $\iota_{\O}^{\eps,\alpha\eps,r,k_\eps r}(x)+x'=0$.
\end{itemize}
\end{definition}
\begin{example}\
\begin{enumerate}
 \item Our prominent example will be of course quantitative $K$-theory $\K_*(A)=(K^{\eps,r}_*(A))_{0<\eps<1/4,r>0}$ of a filtered $C^*$-algebras $A=(A_r)_{r>0}$  with structure maps $\iota^{\eps,\eps',r,r'}:K^{\eps,r}_*(A)\lto K^{\eps',r'}_*(A)$ and $\iota^{\eps,\,r}:K^{\eps,r}_*(A)\lto K_*(A)$ such that $\iota_*^{\eps,\eps',r,r'}
=\iota^{\eps',r'}_*\circ \iota_*^{\eps,\eps',r,r'}$ for $0<\eps\lq \eps'<1/4$ and $0<r\lq r'$;
\item If  $(\O_i)_{i\in\N}$ is a family of quantitative object with $\O_i=(O_i^{\eps,r})_{0<\eps<1/4,r>0}$ for any integer $i$. 
Define $\prod_{i\in\N} \O_i=(\prod_{i\in\N} O_i^{\eps,r})_{0<\eps<1/4,r>0}$.
Then $\prod_{i\in\N} \O_i$ is also a quantitative object. In the case of a constant family $(\O_i)_{i\in\N}$ with $\O_i=\O$ a 
quantitative object, then we set  $\O^\N$ for $\prod_{i\in\N} \O_i$.

\end{enumerate}
\end{example}

\subsection{Controlled morphisms}

Obviously, the definition of controlled morphism \cite[Section 2]{oy2} can be then extended to quantitative objects.

\begin{definition}
 Let $(\lambda,h)$ be a control pair and let  $\O=(O^{\eps,r})_{0<\eps<1/4,r>0}$ and $\O'=(O'^{\eps,r})_{0<\eps<1/4,r>0}$ be 
quantitative objects. A $(\lambda,h)$-controlled morphism
$$\F:\O\to\O'$$ is a family $\F=(F^{\eps,r})_{0<\eps< \frac{1}{4\lambda},r>0}$ of semigroups homomorphisms
 $$F^{\eps,r}:\O^{\eps,r} \to \O'^{\lambda\eps,h_\eps r}$$ such that for any positive
numbers $\eps,\,\eps',\,r$ and $r'$ with
$0<\eps\lq\eps'< \frac{1}{4\lambda}$ and $h_\eps r\lq h_{\eps'}r'$, we have
$$F^{\eps',r'}\circ \iota_\O^{\eps,\eps',r,r'}=\iota_{\O'}^{\lambda\eps,\lambda\eps', h_\eps r,h_{\eps'}r'}\circ F^{\eps,r}.$$
\end{definition}
When  it is not necessary  to specify the control pair, we will just say that $\F$ is a controlled morphism. 
If $\O=(O^{\eps,r})_{0<\eps<1/4,r>0}$ is a quantitative object, let us define the identity $(1,1)$-controlled morphism
$\Id_\O=(Id_{O^{\eps,r}})_{0<\eps< 1/4,r>0}:\O\to\O$. Recall that if $A$ and $B$ are filtered
$C^*$-algebra and if $\F:\K_*(A)\to\K_*(B)$ is a $(\lambda,h)$-controlled morphism, then $\F$ induces
 a  morphism $F:K_*(A)\to K_*(B)$ unically defined
by $\iota^{\eps,r}_*\circ F^{\eps,r}=F\circ\iota^{\eps,r}_*$.

If $(\lambda,h)$ and $(\lambda',h')$ are two control pairs, define
$$h*h': (0,\frac{1}{4\lambda\lambda'})\to (0,+\infty);\, \eps\mapsto h_{\lambda'\eps}h'_\eps.$$
Then $(\lambda\lambda',h*h')$  is again  a control pair. 
 Let   $\O=(O^{\eps,r})_{0<\eps<1/4,r}$, $\O'=(O'^{\eps,r})_{0<\eps<1/4,r}$ and $\O''=(O''^{\eps,r})_{0<\eps<1/4,r}$ 
be quantitative objects, let $$\F=(F^{\eps,r})_{0<\eps<\frac{1}{4\alpha_\F},r>0}:\O\to\O'$$ 
 be
a $(\alpha_\F,k_\F)$-controlled morphism, and let
 $$\G=(G^{\eps,r})_{0<\eps<\frac{1}{4\alpha_\G},r>0}:\O'\to\O''$$ 
be a $(\alpha_\G,k_\G)$-controlled morphism. Then $\G\circ\F:\O\to \O''$ is the $(\alpha_\G\alpha_F,k_\G*k_\F)$-controlled
morphism defined by the family 
$$(G^{\alpha_\F\eps,k_{\F,\eps}r}\circ F^{\eps,r}:O^{\eps,r}\to 
{O''}^{\alpha_\G\alpha_\F\eps,k_{\F,\eps}k_{\G,\alpha_\F,\eps}r})_{0<\eps<\frac{1}{4\alpha_\F\alpha_\G},r>0:}.$$

\begin{notation}
  Let   $\O=(O^{\eps,r})_{0<\eps<1/4,r>0}$ and $\O'=(O'^{\eps,r>0})_{0<\eps<1/4,r>0}$be  quantitative objects and
let $\F=(F^{\eps,r})_{0<\eps<1/4,r>0}:\O\to\O'$ (resp. $\G=(G^{\eps,r})_{0<\eps<1/4,r>0}:\O\to\O'$) be a
$(\alpha_\F,k_\F)$-controlled morphism (resp. a $(\alpha_\G,k_\G)$-controlled morphism). Then we write
$\F\aeq\G$ if
\begin{itemize}
 \item $(\alpha_\F,k_\F)\lq (\lambda,h)$ and $(\alpha_\G,k_\G)\lq (\lambda,h)$.
\item for every $\eps$ in $(0,\frac{1}{4\lambda})$ and $r>0$, then
$$\iota^{\alpha_\F\eps,\lambda\eps,k_{\F,\eps}r,h_\eps r}_j\circ F^{\eps,r}=\iota^{\alpha_\G\eps,\lambda\eps,k_{\G,\eps}r,h_\eps r}_j\circ G^{\eps,r}.$$
\end{itemize}
\end{notation}


 \begin{definition}
  Let $(\lambda,h)$ be a control pair, and let $\F:\O\to\O'$  be a $(\alpha_\F,k_\F)$-controlled morphism with
$(\alpha_\F,k_\F)\lq (\lambda,h)$.
 $ \F$ is called  $(\lambda,h)$-invertible or a $(\lambda,h)$-isomorphism if there exists a controlled morphism 
$\G:\O'\to\O$ such that
$\G\circ\F\aeq \Id_{\O}$ and 
 $\F\circ\G\aeq \Id_{\O'}$. The controlled morphism $\G$ is called a $(\lambda,h)$-inverse for $\G$.
 \end{definition}

In particular, if $A$ and $B$ are filtered
$C^*$-algebras and if $\G:\K_*(A)\to\K_*(B)$ is a $(\lambda,h)$-isomorphism, then the induced  morphism $G:K_*(A)\to K_*(B)$  is an isomorphism
and its inverse is induced by  a controlled morphism (indeed induced by any $(\lambda,h)$-inverse for $\F$).

If $\A=(A_i)_{i\in\N}$ is  any  family of filtered $C^*$-algebras and if $\H$  a separable
Hilbert space. Set $\A^\infty_{c,r}=\prod_{i\in\N}\K(\H)\ts A_{i,r}$ for any $r>0$ and define  the $C^*$-algebra
$\A_c^\infty$ as the closure of $\bigcup_{r>0}A^\infty_{c,r}$ in $\prod_{i\in\N}\K(\H)\ts A_i$.

\begin{lemma}\label{lem-prod-filtered}
Let   $\A=(A_i)_{i\in\N}$ be a family of filtered $C^*$-algebras
and let $$\F_{\A,*}=(F_{\A,\eps,r})_{0<\eps, 1/4,r>0}:\K_*(\A_c^\infty)\longrightarrow \prod \K_*(A_i),$$
where $$F^{\eps,r}_{\A,*}: K_*^{\eps,r}(\A_c^\infty)\longrightarrow \prod_{i\in\N} K_*^{\eps,r}(A_i)$$ is  the map induced on
the $j$ th factor and up to the Morita  equivalence by the restriction to  $\A_c^{\infty}$  of the  evaluation 
$\prod_{i\in\N} \K(\H)\ts A_{i}\to \K(\H)\ts A_{j}$  at $j\in\N$. Then, $\F_{\A,*}$ is a
 $(\alpha,h)$-controlled isomorphism for a control pair   
$(\alpha,h)$ independent on the family  $\A$.
\end{lemma}
We postpone the  proof of this lemma until  the end the next subsection.

\subsection{Control exact sequences}\label{subsection-control-exact-sequence}\begin{definition}
Let $(\lambda,h)$ be a control pair,
\begin{itemize}
 \item
Let $\O=(O_{\eps,r})_{0<\eps< \frac{1}{4},r>0},\, \O'=(O'_{\eps,r})_{0<\eps< \frac{1}{4},r>0}$  and
$\O''=(O''_{\eps,r})_{0<\eps< \frac{1}{4},r>0}$ be quantitative objects and let 
$$\F=(F^{\eps,r})_{0<\eps< \frac{1}{4\lambda},r>0}:\O\to\O'$$ be a  $(\alpha_\F,k_\F)$-controlled morphism and 
  let   $$\G=(G^{\eps,r})_{0<\eps<\frac{1}{4\alpha_\G},r>0}:\O'\to \O''$$
be a $(\alpha_\G,k_\G)$-controlled morphism.
Then the composition
$$\O\stackrel{\F}{\to}O'\stackrel{\G}{\to}\O''$$ is said to be $(\lambda,h)$-exact at $\O'$ if
 $\G\circ\F=0$ and if
 for any $0<\eps<\frac{1}{4\max\{\lambda\alpha_\F,\alpha_\G\}}$, any $r>0$  and  any $y$ in
$O'^{\eps,r}$  such that  $G^{\eps,r}(y)=0$ in $O''_{\eps,r}$, there exists an
  element $x$ in $O^{\lambda\eps,h_\eps r}$
such that $$F^{\lambda\eps,h_{\lambda\eps}r}(x)=\iota_{\O'}^{\eps,\alpha_\F\lambda\eps,r,k_{\F,\lambda\eps}h_\eps r}(y)$$ in
$O'^{\alpha_\F\lambda\eps,k_{\F,\lambda\eps}h_\eps r}$.
\item  A sequence of controlled morphisms
$$\cdots\O_{k-1}\stackrel{\F_{k-1}}{\lto}\O_{{k}}\stackrel{\F_{k}}{\lto}
\O_{{k+1}}\stackrel{\F_{k+1}}{\lto}\O_{{k+2}}\cdots$$ is called $(\lambda,h)$-exact if for every $k$,
the composition   $$\O_{{k-1}}\stackrel{\F_{k-1}}{\lto}\O_{{k}} \stackrel{\F_{k}}{\lto}
\O_{{k+1}}$$ is $(\lambda,h)$-exact at $\O_{{k}}$.
\end{itemize}
\end{definition}

\begin{definition}
Let $A$ be a $C^*$-algebra filtered by $(A_r)_{r>0}$ and let $J$ be an
ideal of $A$. The  extension of $C^*$-algebras
$$0\to J\to A\to A/J\to 0$$ is said to be filtered and semi-split (or a semi-split extension of filtered $C^*$-algebras) if there
exists a completely positive cross-section  $$s:A/J\to A$$ such that
$$s((A/ J)_r))\subset A_r$$  for any
number $r>0$. Such a cross-section is said to be semi-split and filtered.
\end{definition}
Notice that  in this case, the ideal $J$ is then filtered by $(A_r\cap J)_{r>0}$.
For any extension of $C^*$-algebras $$0\to J\to A\to A/J\to 0$$  we denote by $\partial_{J,A}:K_{*}(A/J)\to K_{*}(J)$ the associated
(odd degree) boundary map.
\begin{proposition}\label{prop-bound}
There exists a control pair $(\alpha_\DD,k_\DD)$  such that for any    semi-split  extension of filtered
$C^*$-algebras
 $$0\longrightarrow J \longrightarrow A \stackrel{q}{\longrightarrow}
A/J\longrightarrow 0,$$ there exists a $(\alpha_\DD,k_\DD)$-controlled  morphism of odd degree
 $$\DD_{J,A}=(\partial_ {J,A}^{\eps,r})_{0<\eps<\frac{1}{4\alpha_\DD},r>0}: \K_*(A/J)\to
  \K_*(J)$$
 which induces in $K$-theory   $\partial_{J,A}:K_{*}(A/J)\to K_{*}(J)$.
\end{proposition}
Moreover  the controlled boundary map enjoys the usual naturally properties with respect to extensions.
If the extension $$0\to J\to A\to A/J\to 0$$ is split by a filtered homomorphism, i.e there exists a homomorphism of filtered $C^*$-algebras $s:A/J\to A$ such that $q\circ s=Id_{A/J}$, then we have
$\DD_{J,A}=0$.
\begin{theorem}\label{thm-six term}  There exists a control  pair $(\lambda,h)$  such that for any  semi-split
extension of filtered    $C^*$-algebras $$0  \longrightarrow J
\stackrel{\jmath}{\longrightarrow} A \stackrel{q}{\longrightarrow}
A/J\longrightarrow 0,$$ then the following six-term sequence is $(\lambda,h)$-exact
$$\begin{CD}
\K_0(J) @>\jmath_*>> \K_0(A)  @>q_*>>\K_0(A/J)\\
    @A\DD_{J,A} AA @.     @V\DD_{J,A} VV\\
\K_1(A/J) @<q_*<< \K_1(A)@<\jmath_*<< \K_1(J)
\end{CD}
$$
\end{theorem}
If $A$ is  a filtered $C^*$-algebra, let us denote its suspension and its cone respectively by $SA$ and $CA$, i.e $SA=C_0((0,1))$ and $CA=C_0((0,1])$. We endow $SA$ and $CA$ with the obvious structure of filtered $C^*$-algebras arising from $A$. The algebra $CA$ being contractible as a filtered $C^*$ algebra, the we have $K^{\eps,r}_*(CA)=\{0\}$ for every positive number $\eps$ and $r$ such that $\eps<1/4$ \cite[Lemma 1.27]{oy2}. If we consider the Bott extension
$$0\lto SA\to CA\stackrel{ev_1}{\lto}A\lto 0,$$ where $ev_1:CA\to A$ is the evaluation at $1$ with  corresponding  controlled boundary morphisms
$\DD_A=\DD_{SA,CA}$. Then $$\DD_A=(\partial_A^{\eps,r})_{0<\eps\frac{1}{4\alpha_\DD},r>0}:\K_0(A)\to\K_1(SA)$$ and 
$$\DD_A=(\partial_A^{\eps,r})_{0<\eps\frac{1}{4\alpha_\DD},r>0}:\K_1(A)\to\K_0(SA)$$ are controlled isomorphisms that induce the Bott isomorphisms
$\partial_A:K_0(A)\to K_1(SA)$ and $\partial_A:K_1(A)\to K_0(SA)$.

\medskip

In the particular case of   a filtered extension of $C^*$-algebras  $$0  \to J
\stackrel{\jmath}{\to} A \stackrel{q}{\to}
A/J\to 0$$  that splits  by a filtered morphism, then the following sequence is $(\lambda,h)$-exact
$$0\lto \K_0(J) \stackrel{\jmath}{\lto}  \K_0(A)  \stackrel{q}{\lto}\K_0(A/J)\lto 0.$$

{\it {Proof of lemma \ref{lem-prod-filtered}}}. Assume  first that all the $A_i$ are unital. Then the result is a consequence of \cite[Proposition 3.1]{oy2}. If $A_i$ is not unital for some $i$, then 
for every integer $i$, let us provide   $$\widetilde{A_i}=\{(x,\lambda);\,x\in A_i\,,\lambda\in \C\}$$ with the product $$(x,\lambda)(x',\lambda')=(xx'+\lambda x'+\lambda' x)$$ for all $(x,\lambda)$ and $(x',\lambda')$ in $A_i$. Then $\widetilde{A_i}$ is filtered with
$$\widetilde{A}_{i,r}=\{(x,\lambda);\,x\in A_{i,r}\,,\lambda\in \C\}.$$ Set then 
$\widetilde{\A}=(\widetilde{A}_i)_{i\in\N}$.
Let us denote by  $\mathcal{C}$ the constant family of the $C^*$-algebra $\C$. Then $$0\lto A_c^\infty\lto \widetilde{A}_c^\infty\lto \mathcal{C}_c\lto 0$$ is  a split extension of filtered $C^*$-algebra. Then   we have a commutative diagram

 $$\begin{CD}
0@>>> \K_*({A}_c^\infty)@>>> \K_*(\widetilde{A}_c^\infty)@>>>\K_*(\mathcal{C}_c)@>>>0 \\
   @.     @V \F_{{\A},*} VV
         @VV\F_{\widetilde{\A},*}V     @VV\F_{\mathcal{C},*}V \\
0@>>> \prod_{i\in\N} \K_*({A_i})@>>> \prod_{i\in\N} \K_*({\widetilde{A}_i})@>>> \K^\N_*(\C)@>>>0\end{CD},$$
with  $(\lambda,h)$-exact rows  for the control pair $(\lambda,h)$ of theorem \ref{thm-six term}. The result is now a consequence of a five lemma type argument.
\qed

 \subsection{$KK$-theory and controlled morphisms}
 Let $A$ be a $C^*$-algebra and let $B$ be a filtered $C^*$-algebra filtered by $(B_r)_{r>0}$.  Let us define $A\ts B_r$ as the closure in the spatial tensor
product $A\ts B$ of the algebraic tensor product of $A$ and $B_r$. Then the $C^*$-algebra $A\ts B$ is filtered by $(A\ts B_r)_{r>0}$. 
 If $f:A_1\to A_2$ is a homomorphism of $C^*$-algebras, let us set 
 $$f_B:A_1\ts B\to A_2\ts B;\, a\ts b\mapsto f(a)\ts b.$$ Recall from \cite{kas} that for $C^*$-algebras  $A_1$, $A_2$ and $B$, G. Kasparov defined a tensorization map
$$\tau_B:KK_*(A_1,A_2)\to KK_*(A_1\ts B,A_2\ts B).$$ If $B$ is a filtered $C^*$-algebra, then for any $z$ in $KK_*(A_1,A_2)$  the morphism
$$K_*(A_1\ts B)\lto K_*(A_2\ts B);\, x\mapsto x\ts_{A_1\ts B}\tau_B(z)$$  is induced by a control morphism  which enjoys compatibity properties with Kasparov product \cite[Theorem 4.4]{oy2}.
  \begin{theorem}\label{thm-tensor}
  There exists a control pair $(\alpha_\TT,k_\TT)$ such that 
  \begin{itemize}
  \item for any filtered   $C^*$-algebra $B$; 
  \item for any $C^*$-algebras $A_1$ and $A_2$;
  \item for any element $z$ in $KK_*(A_1,A_2)$,
  \end{itemize}
  There exists a $(\alpha_\TT,k_\TT)$-controlled morphism  $\TT_B(z):\K_*(A_1\ts B)\to \K_*(A_2\ts B)$  with same degree as $z$ that satisfies the following:
 \begin{enumerate}
 \item $\TT_B(z):\K_*(A_1\ts B)\to \K_*(A_2\ts B)$ induces in $K$-theory the right multiplication by $\tau_B(z)$;
\item  For any  elements $z$ and $z'$ in $KK_*(A_1,A_2)$ then
$$\TT_B(z+z')=\TT_B(z)+\TT_B(z').$$
\item Let $A'_1$ be a filtered $C^*$-algebras and  let $f:A_1\to A'_1$ be a homomorphism of  $C^*$-algebras, then
$\T_B(f^*(z))=\T_B(z)\circ f_{B,*}$  for all $z$ in $KK_* (A'_1,A_2)$.
\item Let $
A'_2$ be a $C^*$-algebra and  let $g:A'_2\to A_2$ be a homomorphism of $C^*$-algebras then
$\TT_B(g_{*}(z))=g_{B,*}\circ \TT_B(z)$
for any $z$ in $KK_*(A_1,A'_2)$.
\item $\TT_B([Id_{A_1}])\stackrel{(\alpha_\TT,k_\TT)}{\sim} \Id _{\K_*({A_1}\ts B)}$.
\item For any $C^*$-algebra $D$ and any element $z$ in $KK_*(A_1,A_2)$, we have  $\TT_B(\tau_D(z))=\TT_{B\ts D}(z)$.
\item For any  semi-split extension of  $C^*$-algebras $0\to J\to A\to A/J\to 0$    with corresponding   element   $[\partial_{J,A}]$  of
$KK_1(A/J,J)$
 that implements the boundary map, then we have
$$\TT_B([\partial_{J,A}])=\DD_{J\ts B,A\ts B}.$$\end{enumerate}
\end{theorem}
Moreover, $\TT_B$ is compatible with Kasparov products.
\begin{theorem}\label{thm-product-tensor} There exists a control pair $(\lambda,h)$ such that the following holds :

  let  $A_1,\,A_2$ and $A_3$  be separable $C^*$-algebras and let $B$ be a  filtered $C^*$-algebra. Then for any
$z$ in $KK_*(A_1,A_2)$ and  any  $z'$ in
$KK_*(A_2,A_3)$, we have
$$\TT_B(z\ts_{A_2} z')\aeq  \TT_B(z') \circ\TT_B(z).$$
\end{theorem}
We also have in the case of finitely generated group a controlled version of the Kasparov transformation.
Let $\Ga$ be a finitely generated group. Recall that a length on $\Gamma$ is a map $\ell:\Gamma\to\R^+$ such
that
\begin{itemize}
\item $\ell(\gamma)=0$ if and only if $\gamma$ is the identity element  $e$  of $\Gamma$;
\item  $\ell(\gamma\gamma')\lq\ell(\gamma)+\ell(\gamma')$ for all
  element $\gamma$ and $\gamma'$ of $\Gamma$.
\item $\ell(\gamma)=\ell(\gamma^{-1})$.
\end{itemize}
In what follows, we will assume that $\ell$ is a word length arising from a finite generating symmetric  set $S$, i.e
$\ell(\gamma)=\inf\{d\text{ such that }\ga=\ga_1\cdots\ga_d\text{ with } \ga_1,\ldots,\ga_d\text{ in }S\}$.
Let us denote by $B(e,r)$ the
ball centered at the neutral  element of $\Ga$ with  radius $r$, i.e
$B(e,r)=\{\gamma\in\Gamma\text{ such that }\ell(\gamma)\lq r\}$. Let $A$ be a separable $\Ga$-$C^*$-algebra, i.e a separable $C^*$-algebra provided with an action of $\Ga$ by automorphisms.  For any positive number $r$, we set
$$(A\rtr \Gamma)_r\defi\{f\in C_c(\Ga,A)\text{ with support in
}B(e,r)\}.$$ Then the $C^*$-algebra $A\rtr \Gamma$ is filtered by
$((A\rtr \Gamma)_r)_{r>0}$.
Moreover if $f:A\to B$ is a $\Ga$-equivariant morphism of $C^*$-algebras, then the induced homomorphism
$f_\Ga:A\rtr\Ga\to B\rtr\Ga$ is a filtered homomorphism. In \cite{kas} was constructed
for any $\Ga$-$C^*$-algebras $A$ and $B$ a natural transformation
$J_\Ga:KK_*^\Ga(A, B)\to KK_*(A\rtr\Ga,B\rtr\Ga)$ that preserves Kasparov products.
\begin{theorem}\label{thm-kas}
There exists a control pair $(\alpha_\JJ,k_\JJ)$ such that
\begin{itemize}
\item for any  separable  $\Gamma$-$C^*$-algebras  $A$ and $B$;
\item For any $z$ in $KK^\Ga_*(A,B)$,
\end{itemize}
there exists a $(\alpha_\JJ,k_\JJ)$-controlled morphism
$$\JR(z): \K_*(\AG)\to
\K_*(B\rtr\Gamma)$$
 of same degree as $z$ that satisfies the following:
\begin{enumerate}
\item For any element  $z$  of
$KK^\Ga_*(A,B)$, then
$\JR(z): \K_*(\AG)\to
\K_*(B\rtr\Gamma)$ induces in $K$-theory right multiplication by
$J_\Ga^{red}(z)$.
\item  For any $z$ and
  $z'$ in $KK_*^\Gamma(A,B)$, then
  $$\JR(z+z')=\JR(z)+\JR(z').$$

\item For any $\Gamma$-$C^*$-algebra $A'$, any homomorphism  $f:A\to A'$  of $\Gamma$-$C^*$-algebras and  any $z$ in $KK_*^\Gamma(A',B)$, then
$\JR(f^*(z))=
\JR(z)\circ f_{\Ga,*}$.
\item For any $\Gamma$-$C^*$-algebra $B'$, any homomorphism $g:B\to B'$ of
  $\Gamma$-$C^*$-algebras and  any
  $z$ in $KK_*^\Gamma(A,B)$, then $\JR(g_*(z))=g_{\Ga,*}\circ
\JR(z)$.
\item If $$0\to J\to A\to A/J\to 0$$ is a semi-split exact sequence of
  $\Ga$-$C^*$-algebras, let $[\partial_{J,A}]$ be the element of
  $KK^\Ga_1(A/J,J)$ that implements the boundary map
  $\partial_{J,A}$. Then we have
$$\JR([\partial_{J,A}])=\D_{J\rtr\Ga,A\rtr\Ga}.$$\end{enumerate}
\end{theorem}
The controlled Kasparov transformation is compatible with Kasparov products.
\begin{theorem}\label{thm-product} There exists a control pair $(\lambda,h)$ such that the following holds:
for every  separable $\Ga$-$C^*$-algebras $A,\,B$ and $D$, any elements $z$ in
$KK_*^\Gamma(A,B)$ and   $z'$ in
$KK_*^\Gamma(B,D)$, then
$$\JR(z\otimes_B z')\aeq \JR(z')\circ \JR(z).$$
\end{theorem}
We have similar result for maximal crossed products.
\subsection{Quantitative  assembly maps}\label{subsection-quantitative-assembly-map}
Let $\Ga$ be a finitely generated group and let $B$ be a $\Ga$-$C^*$-algebra $B$. We equip $\Ga$ with any word metric. Recall that if $d$ is  a positive number, then the Rips complex of degree $d$ is the set $P_d(\Ga)$ of probability measure with support of diameter less than $d$. Then $P_d(\Ga)$ is a locally finite simplicial complex and  provided  with the simplicial topology, $P_d(\Ga)$ is endowed  with a proper and cocompact action of $\Ga$ by left translation.
In \cite{oy2} was constructed for any $\Ga$-$C^*$-algebra $B$ a bunch of quantitative assembly maps
$$\mu_{\Gamma,B,*}^{\eps,r,d}: KK_*^\Gamma(C_0(P_d(\Ga)),B)\to
K_*^{\eps,r}(B\rtr\Ga),$$ with $d>0,\,\eps\in(0,1/4)$ and $r\gq r_{d,\eps}$,  where 
$$r: [0,+\infty)\times (0,1/4)\to (0,+\infty):\, (d,\eps)\mapsto r_{d,\eps}$$ is a function  independent on 
$B$, non decreasing in $d$ and 
non increasing  in $\eps$.  Moreover, the maps $\mu_{\Gamma,B,*}^{\eps,r,d}$ induced the usual assembly maps 
$$\mu_{\Gamma,B,*}^{d}: KK_*^\Gamma(C_0(P_s(\Ga)),B)\to
K_*(B\rtr\Ga),$$ i.e  $\mu_{\Gamma,B,*}^{d}=\iota_*^{\eps,r}\circ\mu_{\Gamma,B,*}^{\eps,r,d}$.
Let us recall now the definition of the quantitative assembly maps.
Observe first that  any $x$ in $P_d(\Ga)$ can be written down  in a unique way as a finite convex combination
$$x=\sum_{\ga\in\Ga}\lambda_\ga(x)\de_\ga,$$ where $\de_\ga$ is the Dirac probability measure at $\ga$ in $\Ga$.
The functions $$\lambda_\ga:P_d(\Ga)\to [0,1]$$ are continuous and 
$\ga(\lambda_{\ga'})=\lambda_{\ga\ga'}$ for all $\ga$ and $\ga'$ in $\Ga$.
 The
function
$$p_{\Ga,d}:\Gamma\to C_0(P_d(\Ga));\, \gamma\mapsto\sum_{\gamma\in\Gamma}\lambda_e^{1/2}\lambda_\ga^{1/2}$$
 is a projection of
 $C_0(P_d(\Ga))\rtr\Gamma$ with  propagation  less than $d$. Let us set then $r_{d,\eps}=k_{\JJ,\eps/\alpha_J}d$.
Recall that $k_\JJ$ can be chosen non increasing and in this case,  $r_{d,\eps}$ is non decreasing in
 $d$ and non increasing in $\eps$.
\begin{definition}\label{def-quantitative-assembly-map}
For any $\Gamma$-$C^*$-algebra $A$ and any positive numbers $\eps$, $r$ and
$d$ with $\eps<1/4$ and
$r\gq r_{d,\eps}$, we define the quantitative assembly map
\begin{eqnarray*}
\mu_{\Gamma,A,*}^{\eps,r,d}: KK_*^\Gamma(C_0(P_d(\Ga)),A)&\to&
K_*^{\eps,r}(\AG)\\
z&\mapsto&\big(J_\Gamma^{red,\frac{\eps}{\alpha_J},\frac{r}{k_{J,{\eps}/{\alpha_J}}}}(z)\big)\left([p_{\Ga,d},0]_{\frac{\eps}{\alpha_J},\frac{r}{k_{J,{\eps}/{\alpha_J}}}}\right).
\end{eqnarray*}
\end{definition}
Then according to  theorem \ref{thm-kas}, the map  $\mu_{\Gamma,A}^{\eps,r,d}$
is a  homomorphism of   groups  (resp. groups) in even (resp. odd) degree. For any positive numbers
 $d$ and $d'$ such that $d\lq d'$, we denote by
$q_{d,d'}:C_0(P_{d'}(\Ga))\to C_0(P_d(\Ga))$ the homomorphism induced by the restriction
  from $P_{d'}(\Ga)$ to $P_d(\Ga)$. It is straightforward to check that if
  $d$, $d'$ and  $r$  are positive numbers such that $d\lq d'$ and
  $r\gq r_{d',\eps}$, then
$\mu_{\Gamma,A}^{\eps,r,d}=\mu_{\Gamma,A}^{\eps,r,d'}\circ
q_{d,d',*}$. Moreover, for every positive numbers
$\eps,\,\eps',\,d,\,r$ and $r'$ such that $\eps\lq\eps'\lq
1/4$, $r_{d,\eps}\lq r$, $r_{d,\eps'}\lq r'$, and
$r<r'$, we get by definition of a controlled morphism that
\begin{equation*}
\iota^{\eps,\eps',r,r'}_*\circ \mu_{\Ga,A,*}^{\eps,r,d}=\mu_{\Ga,A,*}^{\eps',r',d}.
\end{equation*}
%
%

\section{Persistence approximation property}\label{section-persistence}

In this section, we introduce the persistence approximation property for filtered $C^*$-algebras. In the case of  a crossed product $C^*$-algebra by  a finitely generated group, we prove that  the persistence approximation property follows  from the Baum-Connes conjecture with  coefficients.

\smallskip

Let $B$ be a filtered $C^*$-algebra. As a consequence of proposition \ref{proposition-approximation}, we see that there exists for every  $\eps\in (0,1/4]$ a surjective map $$\lim_{r>0} K_*^{\eps,r}(B)\to K_*(B)$$ induced by
$(\iota_{*}^{\eps,r})_{r>0}$. Moreover, although this morphism is not a priori ono-to-one,  there exist for every   $\eps\in (0,1/4]$
and $r>0$, positive numbers    $\eps'$ in $[\eps,1/4)$
 (indeed independent on $x$ and $B$) and $r'>r$ such that
for any $x$ in  $K_*^{\eps,r}(B)$, then $\iota_{*}^{\eps,r}(x)=0$ implies that $\iota_{*}^{\eps,\eps',r,r'}(x)=0$ in $K_*^{\eps',r'}(B)$.
It is of revelance to ask whether this $r'$ depends or not on $x$, in other word whether 
 the family $(K_*^{\eps,r}(B))_{\eps\in(0,1/4),r>0}$ provides a persistent approximation for $K_*(B)$ in the following sense:
for any $\eps$ in $(0,1/4)$ and $r>0$, there exist $\eps'$ in $(\eps,1/4)$ and $r'\gq r$ such that for any $x$ in $K_*^{\eps,r}(B)$, then
$\iota_{*}^{\eps,\eps',r,r'}(x)\neq 0$ in $K_*^{\eps',r'}(B)$ implies that $\iota_{*}^{\eps,r}(x)\neq 0$ in $K_*(B)$.

Let us consider for a filtered $C^*$-algebra $B$ and positive numbers  $\eps,\,\eps',\,$ and $r'$ such that $0<\eps\lq\eps'<1/4$ and $0<r\lq r'$ 
 the following statement:

\smallskip

$\sta_*(B,\eps,\eps',r,r')$ : for any $x\in K_*^{\eps,r}(B)$, then $\iota^{\eps,r}_*(x)=0$ in $K_*(B)$ implies 
that $\iota^{\eps,\eps',r,r'}_*(x)=0$ in  $K_*^{\eps',r'}(B)$.

\smallskip
Notice that $\sta_*(B,\eps,\eps',r,r')$  can be rephrased as follows: 

\smallskip

the restriction of $\iota_*^{\eps',r'}:K_*^{\eps',r'}(B)\to K_*(B)$  to $\iota_*^{\eps,\eps',r,r'}(K_*^{\eps,r}(B))$ is one-to-one.

\smallskip

We  investigate in this section the following persistence approximation property: given $\eps$ small enought and $r$ positive numbers, 
is there exist positive numbers $\eps'$ and $r'$ with $0<\eps\lq\eps'<1/4$ and $r<r'$ such that $\sta_*(B,\eps,\eps',r,r')$  holds?

%
%

\subsection{The case of crossed products}

\begin{theorem}\label{thm-persistence-crossed-products}
Let $\Ga$ be a finitely generated group. Assume  that 
\begin{itemize}
 \item $\Ga$ satisfies the Baum-Connes conjecture with coefficients.
\item  $\Ga$ admits a cocompact universal example for proper actions.
\end{itemize}
Then for some universal constant $\lambda_{\pa}\gq 1$,  any $\eps$ in $(0,\frac{1}{4\lambda_\pa})$, 
any $r>0$, and any  $\Ga$-$C^*$-algebra $A$ there exists  $r'\gq r$  such that
$\sta(A\rtr\Ga,\eps,\lambda_\pa\eps,r,r')$ holds.
%
%
%

\end{theorem}
\begin{proof}
Notice first that since $\Ga$ satisfies the Baum-Connes conjecture with  coefficients and admits  a cocompact universal example for proper action, there exist positive numbers $d$ end $d'$ with $d\lq d'$ such that
for any $\Ga$-$C^*$-algebra $B$, the following is satisfies:
\begin{itemize}
\item for any $z$ in $K_*(B\rtr\Ga)$, there exists $x$ in $KK_*^\Ga(C_0(P_d(\Ga)),B)$ such that $\mu_{\Ga,B,*}^d(x)=z$;
\item for any $x$ in $KK_*^\Ga(C_0(P_d(\Ga)),B)$ such that $\mu_{\Ga,B,*}^d(x)=0$, then $q_{d,d'}^*(x)=0$ in   $KK_*^\Ga(C_0(P_d(\Ga)),B)$, where 
$q_{d,d'}^*: KK_*^\Ga(C_0(P_d(\Ga)),B)\to KK_*^\Ga(C_0(P_{d'}(\Ga)),B)$ is induced by the inclusion $P_d(\Ga)\hookrightarrow P_{d'}(\Ga)$.
\end{itemize}
Let us fix such $d$ and $d'$, let  $\lambda$ be as in proposition \ref{proposition-approximation},  pick  $(\alpha,h)$ as in lemma \ref{lem-prod-filtered}
and set $\lambda_{\pa}=\alpha\lambda$. Assume that this statement does not hold. 
Then there exists:
\begin{itemize}
 \item $\eps$ in $(0,\frac{1}{4\lambda_{\pa}})$ and $r>0$;
\item an unbounded increasing sequence $(r_i)_{i\in\N}$ bounded below by $r$;
\item a sequence  of $\Ga$-$C^*$-algebras $(A_i)_{i\in\N}$;
\item a sequence of elements $(x_i)_{i\in\N}$ with $x_i$ in $K^{\eps,r}_*(A_i\rtr\Ga)$
\end{itemize}
such that 
$\iota^{\eps,r}_*(x_i)=0$ in $K_*(A_i\rtr\Ga)$ and $\iota^{\eps,\lambda_\pa\eps,r,r_i}_*(x_i)\neq 0$  
in  $K_*^{\lambda_\pa\eps,r_i}(A_i\rtr\Ga)$ for every integer $i$.
We can assume without loss of generality that $r\gq r_{d',\eps}$.

According to lemma   \ref{lem-prod-filtered}, there exists an element $x$ in 
 $K^{\alpha\eps,h_\eps r}_*\big(\big(\prod_{j\in\N}\K(H)\ts A_j\big)\rtr\Ga\big)$
that maps to $\iota_*^{\eps,\alpha\eps,r,h_\eps r}(x_i)$ for all integer $i$  under the composition
\begin{equation*}
 K^{\alpha\eps,h_\eps r}_*\big(\big(\prod_{j\in\N}\K(H)\ts A_j\big)\rtr\Ga\big)\longrightarrow
K^{\alpha\eps,h_\eps r}_*(\K(\H)\otimes A_i\rtr\Ga)\stackrel{\MM_{A_i}^{\alpha\eps,h_\eps r}}{\longrightarrow}
K^{\alpha\eps,h_\eps r}_*(\ A_i\rtr\Ga), \end{equation*} where the first map is induced by the
{\it i th} projection  $\prod_{j\in\N}\K(H)\ts A_j{\longrightarrow}
\K(\H)\otimes A_i$ and the map $\MM_{A_i}^{\alpha\eps,h_\eps r}$ is the Morita equivalence of proposition\ref{prop-morita}.
Let $z$ be an element in
$KK_*^\Ga\big(C_0(P_d(\Ga)),\prod_{j\in\N}\K(H)\ts A_j\big)$ such that 
$$\mu_{\Ga,\prod_{j\in\N}\K(H)\ts A_j,*}^d(z)=\iota_*^{\alpha\eps,h_\eps r}(x)$$ in $K_*\big(\big(\prod_{j\in\N}\K(H)\ts A_j\big)\rtr\Ga\big)$.
Recall from  \cite[Proposition 3.4]{oy1},
that we have an isomorphism
\begin{equation}\label{equ-morita-equiv-khom} 
KK^\Ga_*\big(C_0(P_{d}(\Ga)),\prod_{j\in\N}\K(H)\ts A_j\big)\stackrel{\cong}{\longrightarrow}
\prod_{j\in\N}KK^\Ga_*(C_0(P_{d}(\Ga)), A_j)
\end{equation}
induced on the $i$ {\it th} factor  and up to the Morita equivalence
$$KK^\Ga_*(C_0(P_d(\Ga)),A_j)\cong KK^\Ga_*(C_0(P_d(\Ga)),\K(\H)\otimes
A_j)$$ by the $i$ the projection  $\prod_{j\in\N}\K(H)\ts A_j\to\K(\H)\otimes A_i$.
Let  $(z_j)_{j\in\N}$  be the element of $\prod_{j\in\N}KK^\Ga_*(C_0(P_{d}(\Ga)), A_j)$
corresponding to $z$ under this identification. The quantitative Baum-Connes 
assembly maps  being compatible with  the usual one, we get that 
$$\mu_{\Ga,\prod_{j\in\N}\K(H)\ts A_j,*}^d(z)=
\iota_*^{\alpha\eps,h_\eps r}\circ\mu^{d,\alpha\eps,h_\eps r}_{\Ga,\prod_{j\in\N}\K(H)\ts A_j,*}(z).$$ But then, according to item (ii) of proposition \ref{proposition-approximation}, there exists $R\gq  h_\eps r$ such that 
\begin{eqnarray*}
 \iota^{\alpha\eps,\lambda_\pa\eps,h_\eps r,R}_*(x)&=
&\iota^{\alpha\eps,\lambda_\pa\eps,h_\eps r,R}_*\circ \iota_*^{\alpha\eps,h_\eps r}
\circ\mu^{d,\alpha\eps,h_\eps r}_{\Ga,\prod_{j\in\N}\K(H)\ts A_j,*}(z)\\
&=&\mu^{d,\lambda_\pa\eps,R}_{\Ga,\prod_{j\in\N}\K(H)\ts A_j,*}(z)
\end{eqnarray*}
Using once again the compatibility of  the quantitative assembly maps 
with  the usual ones, we obtain by naturality that 
$\mu^{d}_{\Ga, A_i,*,red}(z_i)=0$ for every integer $i$ and hence $q_{d,d',*}(z_i)=0$ in $KK^\Ga_*(C_0(P_{d'}(\Ga)),A_i)$. Using once more,  equation (\ref{equ-morita-equiv-khom}) we deduce that $q_{d,d',*}(z)=0$ in $KK^\Ga_*\big(C_0(P_{d'}(\Ga)),\prod_{j\in\N}\K(H)\ts A_j\big)$ and since $$\mu^{d,\lambda_\pa\eps,R}_{\Ga,\prod_{j\in\N}\K(H)\ts A_j,*}(z)=\mu^{d',\lambda_\pa\eps,R}_{\Ga,\prod_{j\in\N}\K(H)\ts A_j,*}\circ q_{d,d',*}(z)$$  that $\iota^{\alpha\eps,\lambda_\pa\eps,h_\eps r,R}_*(x)=0$ in 
$K^{\lambda_\pa \eps,R}_*\big(\big(\prod_{j\in\N}\K(H)\ts A_j\big)\rtr\Ga\big)$. By naturality, we see that
$\iota^{\eps,\lambda_\pa\eps,r,R}_*(x_i)=0$ in 
$K^{\lambda_\pa\eps,R}_*(A_i\rtr\Ga)$ for every integer $i$. Pick then an integer $i$ such that $r_i\gq R$, we have 
\begin{eqnarray*}
 \iota^{\eps,\lambda_\pa\eps,r,r_i}_*(x_i)&= &\iota^{\lambda_\pa\eps,R,r_i}\circ \iota^{\eps,\lambda_\pa\eps,r,R}_* (x_i)\\
&=&0\end{eqnarray*}
which contradicts our assumption.
\end{proof}
If we specifies the coefficients in the previous  proof, we get indeed
\begin{proposition}
Let $\Ga$ be a finitely generated  group and let $A$ be a $\Ga$-$C^*$-algebra. Assume  that 
\begin{itemize}
\item  $\Ga$ admits a cocompact universal example for proper actions;
\item the Baum-Connes assembly map for $\Ga$ with coefficients in $\ell^\infty(\N,\K(\H)\otimes A)$ is onto;
\item the Baum-Connes assembly map for $\Ga$ with coefficients in $A$ is one to one.
\end{itemize}
Then for some universal constant $\lambda_\pa\gq 1$,  any $\eps$ in $(0,\frac{1}{4\lambda_\pa})$
and any $r>0$ there exists  $r'\gq r$  such that $\sta(A\rtr\Ga,\eps,\lambda_\pa\eps,r,r')$ is satisfied.
\end{proposition}
Since for any $C^*$-algebra $B$, the Baum-Connes assembly map for $\Ga$ with coefficient in 
$C_0(\Ga,B)$ ($B$ being provided with the trivial action) is an isomorphism and since 
$C_0(\Ga,B)\rtr\Ga\cong B\otimes \K(\ell^2(\Ga))$,
 previous proposition leads to the following corollary
\begin{corollary}\label{cor-quant-emb}
Let $\Ga$ be a finitely generated group and let $B$ be a $C^*$-algebra. Assume  that 
\begin{itemize}
\item  $\Ga$ admits a cocompact universal example for proper actions;
\item the Baum-Connes assembly map for $\Ga$ with coefficients in $\ell^\infty(\N,C_0(\Ga,\K(\H)\otimes B))$ is onto;
\end{itemize}
Then for some universal constant $\lambda_\pa\gq 1$,  any $\eps$ in $(0,\frac{1}{4\lambda_\pa})$
and any $r>0$ there exists  $r'\gq r$  such that $\sta(B\otimes \K(\ell^2(\Ga)),\eps,\lambda_\pa\eps,r,r')$ is satisfied.
Moreover, if $\Ga$ satisfies the Baum-Connes conjecture with coefficients, then $r'$ does not depend on $B$.
\end{corollary}
If we take $B=\C$ in the previous corollary, we obtain the following linear algebra statement:
\begin{proposition}\label{prop-quant-emb}
Let $\Ga$ be a finitely generated  and let $H$ be a separable Hilbert space. Assume  that 
\begin{itemize}
\item  $\Ga$ admits a cocompact universal example for proper actions;
\item the Baum-Connes assembly map for $\Ga$ with coefficients in $\ell^\infty(\N,C_0(\Ga,\K(\H)))$ is onto;
\end{itemize}
Then for some universal constant $\lambda\gq 1$,  any $\eps$ in $(0,\frac{1}{4\lambda})$
and any $r>0$ there exists  $R\gq r$  such that
\begin{itemize}
 \item If $u$ is an $\eps$-$r$-unitary of $\K(\ell^2(\Ga)\ts H)+\C Id_{\ell^2(\Ga)\ts H}$, then $u$ is connected
to $Id_{\ell^2(\Ga)\ts H}$ by a homotopy of $\lambda\eps$-$R$-unitaries.
\item If $q_0$ and $q_1$ are  $\eps$-$r$-projections of $\K(\ell^2(\Ga)\ts H)$ such that 
$\rank \kappa_0(q_0)=\rank \kappa_0(q_1)$. Then $q_0$ and $q_1$ are connected
 by a homotopy of $\lambda\eps$-$R$-projections.
\end{itemize}
\end{proposition}

\subsection{Induction and geometric setting}
The conclusions of corollary \ref{cor-quant-emb} and proposition
 \ref{prop-quant-emb} concern only the metric properties of $\Ga$ (indeed as we shall see latter up to  quasi-isometries). For the purpose 
of having 
statements  analogous to  corollary \ref{cor-quant-emb}  in a metric setting, we need to have a completly geometric
description of the
quantitative assembly maps 
\begin{equation*}
 \begin{split}
\mu^{d,\eps,r}
_{\Ga,\prod_{i\in\N}C_0(\Ga,\K(\H)\ts A_i),*}:KK^\Ga_*(C_0(P_d(\Ga)),
&\prod_{i\in\N}C_0(\Ga,\K(\H)\ts A_i))\\ &\lto
K_*((\prod_{i\in\N}C_0(\Ga,\K(\H)\ts A_i)\rtr\Ga) 
\end{split}
\end{equation*}(see the proof of theorem  \ref{thm-persistence-crossed-products}).
Namely, we study  in this subsection a slight generalisation of these maps to the case of induced algebras from the action of a finite subgroup
 of $\Ga$.

\smallskip

Let $\Ga$ be a discrete group equipped with a proper lenght $\ell$. Let $F$ be a finite subgroup of $\Ga$. 
For any $F$-$C^*$-algebra $A$, let us consider the induced $\Ga$-algebra
$$\Ind_F^\Ga(A)=\{f\in C_0(\Ga,A)\text{ such that } f(\ga )=kf(\ga k)\text{ for every }k\text{ in }F\}.$$
Then left translation on  $C_0(\Ga,A)$ provides a $\Ga$-$C^*$-algebra structure on $\Ind_F^\Ga(A)$. Moreover, there is a covariant
representation of $(\Ind_F^\Ga(A),\Ga)$ on the algebra of  adjointable operators of the right Hilbert $A$-module $A\ts \l^2(\Ga)$, 
where
\begin{itemize}
 \item if $f$ is in $\Ind_F^\Ga(A)$, then $f$ acts on $A\ts \l^2(\Ga)$  by pointwise multiplication by $\ga\mapsto\ga^{-1}(f(\ga))$;
\item $\Ga$ acts by left translations.
\end{itemize}
The induced representation then provides an identification between $\Ind_F^\Ga(A)\rtr \Ga$ and the 
algebra of $F$-invariant element of $A\ts \K(\l^2(\Ga))$ for the diagonal action of $F$, the action on  $\K(\l^2(\Ga)$ being by 
right translation. Let us denote by $A_{F,\Ga}$ the algebra of $F$-invariant element  of $A\ts \K(\l^2(\Ga))$  and by 
$$\Phi_{A,F,\Ga}:\Ind_F^\Ga(A)\rtr \Ga \to A_{F,\Ga},$$ the isomorphism induced by the above covariant representation.

The lenght $\l$ gives rise to a
filtration structure $(\Ind_F^\Ga(A)\rtr \Ga_r)_{r>0}$ on $\Ind_F^\Ga(A)\rtr \Ga$ (recall that 
$(\Ind_F^\Ga(A)\rtr \Ga_r)$ is the the set of functions of 
$C_c(\Ga,\Ind_F^\Ga(A)$ with support in the ball of radius $r$ centered at the neutral element). The right invariant metric associated to $\l$ also provides
a filtration structure on  $\K(\l^2(\Ga))$ and hence on $A\ts \K(\l^2(\Ga))$. This filtration is invariant under the action of $F$ and 
moreover the isomorphism $\Phi_{A,F,\Ga}:\Ind_F^\Ga(A)\rtr \Ga \to A_{F,\Ga}$ preserves the filtrations.
By using the induced algebra in the proof of  corollary  \ref{cor-quant-emb}, we get
\begin{proposition}\label{prop-Q-BC}
 Let $F$ be a finite subgroup of a finitely generated  group $\Ga$ and let $A$ be a $F$-$C^*$-algebra.
Assume  that 
\begin{itemize}
\item  $\Ga$ admits a cocompact universal example for proper actions;
\item the Baum-Connes assembly map for $\Ga$ with coefficient in $\ell^\infty(\N,C_0(\Ga,\K(\H)\otimes \Ind_F^\Ga(A)))$ 
is onto;
\end{itemize}Then for some universal constant $\lambda_\pa\gq 1$,  any $\eps$ in $(0,\frac{1}{4\lambda_\pa})$
and any $r>0$ there exists  $r'\gq r$  such that $\sta(A_{F,\Ga},\eps,\lambda_\pa\eps,r,r')$ is satisfied.
Moreover, if $\Ga$ satisfies the Baum-Connes conjecture with coefficients, then $r'$ does not depend on $F$ and $A$.
\end{proposition}
In  \cite{oyono} was stated   for any $F$-$C^*$-algebra an isomorphism
\begin{equation}\label{equ-induction}
\Ind_F^{\Ga}(P_s(\Ga))_*:\lim_X KK_*^F(C(X),A)\stackrel{\cong}{\longrightarrow}KK_*^\Ga(C_0(P_s(\Ga)),\I (A)),
\end{equation}
where
$X$ runs through  $F$-invariant compact subsets of $P_s(\Ga)$.  In order to describe this isomorphism, let us first recall the definition
of induction for equivariant $KK$-theory. Let $A$ and $B$ be $F$-$C^*$-algebras and 
let $(\E,\rho,T)$ be a $K$-cycle for $KK_*^F(A,B)$ where,
\begin{itemize}
 \item $\E$ is a right $B$-Hilbert module provided with an equivariant action of $F$;
\item $\rho:A\to\L_B(\E)$ is an $F$-equivariant representation of $A$ into the algebra $\L_B(\E)$ of adjointable operators
of $\E$;
\item $T$ is a $F$-equivariant operator of $\L_B(\E)$ satisfying the $K$-cycle relations.
\end{itemize}
Let us define 
$$\Ind_F^\Ga(\E)=\{f\in C_0(\Ga,\E)\text{ such that } f(\ga)=kf(\ga k)\text{ for every }k\text{ in }F\}.$$
Then $\Ind_F^\Ga(\E)$ is a right $\Ind_F^\Ga(B)$-Hilbert module for the pointwise scalar product and 
multiplication and the representation
$\rho: A\to \L_B(\E)$  gives rise in the same way to a representation 
$$\Ind_F^\Ga\rho:\Ind_F^\Ga(A)\to \L_{\Ind_F^\Ga(B)}(\Ind_F^\Ga(\E)).$$
 Let $\I T$ be the operator of $\L_{\I(A)}(\I(\E))$ given 
by the pointwise multiplication by $T$, it is then plain to check that $(\I(\E),\I \rho,\I\,T)$ is a $K$-cycle for 
$KK_*^\Ga(\I A,\I B)$ and that moreover, $(\E,\rho,T)\to(\I(\E),\I \rho,\I\,T)$ gives rise to a well defined
morphism  $\I:KK_*^F(A,B){\longrightarrow}KK_*^\Ga(\I(A),\I(B))$.

\smallskip

Back to the definition of the isomorphism of equation (\ref{equ-induction}), let $F$ be a finite subgroup of a discrete group $\Ga$ 
and let $X$ be a $F$-invariant compact subset of $P_s(\Ga)$ for $s>0$.
 If we equipped $\Ga\times X$ with the diagonal 
action of $F$, where the action on 
$\Ga$ is by right
multiplication, then there is a natural identification between $\I(C(X))$ and $C_0((\Ga\times X)/{F})$. 
The map $$(\Ga\times X)/{F}\to P_s(\Ga);\, [(\ga,x)]\mapsto \ga x$$ then gives rises to a
$\Ga$-equivariant homomorphism  $$\Upsilon_{F,X}^\Ga:C_0(P_s(\Ga))\to \I(C(X)).$$ 
Then for any $F$-$C^*$-algebra $A$, the morphism 
 $$KK_*^F(C(X),A){\longrightarrow}KK_*^\Ga(C_0(P_s(\Ga)),\I(A));\, x\mapsto \Upsilon_{F,X}^{\Ga,*}(\I(x))$$
 is compatible
with the inductive limit over $F$-invariant compact subsets of $P_s(\Ga)$ and hence we eventually  obtain a natural homomorphism
$$\Ind_F^{\Ga}(P_s(\Ga))_*:\lim_X KK_*^F(C(X),A){\longrightarrow}KK_*^\Ga(C_0(P_s(\Ga),\I(A))$$
which turns out to be an isomorphism.
\medskip
Let us consider now the composition
\begin{equation}
 \label{equ-geom-setting}\Phi_{A,F,\Ga,*}\circ \mu_{\Gamma,\I A,*}^{\eps,r,s}\circ \Ind_F^{\Ga}(P_s(\Ga))_*:\lim_X KK_*^F(C(X),A){\longrightarrow}K_*^{\eps,r}(A_{F,\Ga}),\end{equation} 
where $X$ runs through $F$-invariant compact subsets 
of $P_s(\Ga)$. The two sides of these maps depend only on the metric structure of $\Ga$ (indeed only on the coarse structure), 
and our aim in next section is to provide a geometric definition for these bunch of assembly maps.
\section{Coarse geometry}\label{section-coarse-geometry}
Let $\Si$ be a proper metric space equipped with a free  action of a finite group $F$ by isometries and
 let $A$ be a $F$-$C^*$-algebra.
Define then $A_{F,\Si}$ as the set of invariant elements  of $A\ts\K(\l^2(\Si))$ for the diagonal action of $F$. For $F$ trivial, we set
 $A_{\{e\},\Si}=A_{\Si}$.
The filtration
  $(A\ts \K(\ell^2(\Si))_r)_{r>0}$ on $A\ts \K(\ell^2(\Si))$ is preserved by the action of the group $F$. Hence, if
$A_{F,\Si,r}$ stands for the set of $F$-invariant elements of $A\ts \K(\ell^2(\Si))_r$, then $(A_{F,\Si,r})_{r>0}$ provides
 $A_{F,\Si}$ with a structure of filtered $C^*$-algebra. Our aim in this section is to investigate the permanence approximation property for
$A_{F,\Si}$. Let us set $\sta_{F,\Si}(\eps,\eps',r,r')$ for the property  $\sta(A_{F,\Si},\eps,\eps',r,r')$, i.e the restriction of 
$$\iota_*^{\eps',r'}:K_*^{\eps',r'}(A_{F,\Si})
\longrightarrow K_*(A_{F,\Si})$$ to 
$\iota_*^{\eps,\eps',r,r'}(K_*^{\eps,r}(A_{F,\Si}))$ is one-to-one.

Following the route of the proof of theorem \ref{thm-persistence-crossed-products}, and in view of equation (\ref{equ-geom-setting}),
 let us set 
$$K_*^F(P_s(\Si),A)=\lim_X KK_*^F(C(X),A),$$where in the inductive limit,
$X$ runs through  $F$-invariant compact subsets of $P_s(\Si)$ for $s>0$.
Our purpose is to define a bunch of local quantitative coarse  assembly maps
$$\nu_{F,\Si,A,*}^{\eps,r,s}:K_*^F(P_s(\Si),A){\longrightarrow}K_*^{\eps,r}(A_{F,\Si}),$$ for 
 $s>0,\,\eps\in(0,1/4)$, $r\gq r_{s,\eps}$  and 
 $$[0,+\infty)\times (0,1/4)\to (0,+\infty):\, (s,\eps)\mapsto r_{s,\eps}$$  a function independant on $A$, non decreasing in $s$ and
non increasing  in $\eps$
such that, if $F$ is a subgroup of a discrete group $\Ga$ equipped with right invariant metric arising from  a proper lenght, then
$\nu_{F,\Ga,A,*}^{\eps,r,s}$ coincides with the composition of equation (\ref{equ-geom-setting}).
\subsection{A local coarse assembly map}\label{sec-indexmap}
Let $\Si$ be a proper discrete metric space, with bounded geometry and equipped with a free  action of a finite group
 $F$ by isometries and let $A$ be a $F$-algebra.  Recall that $A_{F,\Si}$ is defined as the set of invariant elements  
of $A\ts\K(\l^2(\Si))$ for the diagonal action of $F$. Notice that since the action of $F$ on $\Si$ is free, the choice of an equivariant
identification between $\Si/F\times F$ and $\Si$ (i.e the choice of a fundamental domain) gives rise to a 
Morita   equivalence between  $A_{F,\Si}$ and $A\rtimes F$. 
Let us set for $s$ a positive number $$K_*^F(P_s(\Si),A)=\lim_X KK_*^F(C(X),A),$$where  $X$ runs through  
$F$-invariant compact subsets of the Rips complex $P_s(\Si)$ of degree $s$.

The  aim of this section is to construct for $s>0$ a bunch of local coarse assembly 
   maps
 $$\nu_{F,\Si,A,*}^{s}K_*^F(P_d(\Si),A){\longrightarrow}K_*(A_{F,\Si}).$$ \smallskip

Le us define first for any $F$-algebras $A$ and $B$  a map 
$$\tau_{F,\Si}:KK^F_*(A,B)\to KK_*(A_{F,\Si},B_{F,\Si})$$ analogous to the  Kasparov transformation.

Let $z$ be an element in $KK^F_*(A,B)$. Then $z$ can be represented by an equivariant $K$-cycle
$(\pi,T,\H\otimes\ell^2(F)\otimes B)$ where
\begin{itemize}
\item $\H$ is a separable Hilbert
space;
\item $F$ acts diagonally on $\H\otimes\ell^2(F)\otimes B$, trivially on $\H$ and by the right regular 
representation on $\ell^2(F)$.
\item  $\pi$ is a $F$-equivariant
representation of $A$ in the algebra $\L_B(\H\otimes\ell^2(F)\otimes B)$ of adjointable operators of
$\H\otimes B$;
\item $T$ is a $F$-equivariant self-adjoint operator of $\L_B(\H\otimes\ell^2(F)\otimes B)$ 
satisfying  the $K$-cycle conditions, i.e. $[T,\pi(a)]$ and 
  $\pi(a)(T^2-\Id_{\H\otimes B})$  belongs to
  $\K(\H\otimes \ell^2(F))\otimes B$, for every $a$ in $A$.
\end{itemize}
Let $\H_{B,F,\Si}$ be the set of invariant elements in $\H\otimes\ell^2(F)\otimes B\otimes\K(\ell^2(\Si))$. 
Then $\H_{B,F,\Si}$ is obviously a 
right $B_{F,\Si}$-Hilbert module, and $\pi$ induces a representation $\pi_{F,\Si}$ of $A_{F,\Si}$ on the 
algebra $\L_{B_{F,\Si}}(\H_{B,F,\Si})$ of adjointable operators of
$\H_{B,F,\Si}$ and $T$ gives rise also a self-adjoint element  $T_{B,F,\Si}$ of $\L_{B_{F,\Si}}(\H_{B,F,\Si})$. Moreover, 
by choosing an equivariant identification
between $\Si/F\times F$ and $\Si$, we can check that the algebra of $F$-equivariant compact operators on
 $\H\otimes\ell^2(F)\otimes \ell^2(\Si)\otimes B$ coincides with the algebra of compact operators on
the right $B_{F,\Si}$-Hilbert module $\H_{B,F,\Si}$. Hence, $(\pi_{F,\Si},T_{B,F,\Si},\H_{B,F,\Si})$ is a 
 $K$-cycle for $KK_*(A_{F,\Si},B_{F,\Si})$. Furthermore, its class in $KK_*(A_{F,\Si},B_{F,\Si})$ only depends on $z$ and thus we end up with 
a morphism \begin{equation}\label{equ-tau-F-Sigma}\tau_{F,\Si}:KK^F_*(A,B)\to KK_*(A_{F,\Si},B_{F,\Si}).\end{equation}  It also quite easy to see that
 $\tau_{F,\Si}$ is functorial in both variables. Namely, for any $F$-equivariant homomorphism $f:A\to B$ of $F$-algebras, let us 
set $f_{F,\Si}:A_{F,\Si}\to B_{F,\Si}$ for the induced homomorphism. Then for any $F$-algebras $A_1$, $A_2$, $B_1$ and $B_2$ 
and any homomorphism of $F$-algebra $f:A_1\to A_2$ and $g:B_1\to B_2$, we have
$$\tau_{F,\Si}(f^*(z))=f^*_{F,\Si}(\tau_{F,\Si}(z))$$ and $$\tau_{F,\Si}(g_*(z))=g_{F,\Si,*}(\tau_{F,\Si}(z))$$ 
for any $z$ in $KK^F_*(A_2,B_1)$.

We are now in position to define the index map.
Observe that  any $x$ in $P_s(\Si)$ can be written as a finite convex combination
$$x=\sum_{\si\in\Si}\lambda_\si(x)\de_\si$$ where
\begin{itemize}
 \item $\delta_\si$ is the Dirac probability measure  at $\si$ in $\Si$.
\item for every $\si$ in $\Si$, the coordinate function $\lambda_\si: P_s(\Si)\to[0,1]$ is continuous with support in the ball centered at $\si$ 
and with 
radius $1$ for the simplicial distance.
\end{itemize}Moreover, for any $\si$ in $\Si$ and $k$ in $F$, then we have $\lambda_{k\si}(kx)=\lambda_\si(x)$.
Let $X$ be a compact $F$-invariant subset of $P_d(\Si)$.
Let us define $$P_X:C(X)\ts\ell^2(\Si)\to C(X)\ts\ell^2(\Si)$$  by
\begin{equation}\label{equ-def-PX}(P_X\cdot h)(x,\si)=\lambda^{1/2}_\si(x)\sum_{\si'\in\Si}h(x,\si')\lambda^{1/2}_{\si'}(x),\end{equation} for any $h$ in $C(X)\ts\ell^2(\Si)$.
Since $\sum_{\si\in\Si}\lambda_\si=1$, it is straightforward to check that $P_X$ is a $F$-equivariant projection 
in $C(X)\ts \K(\ell^2(\Si))$ with  propagation less than $s$ 
and hence gives rise in particular to a class $[P_X]$ in  $K_0(C(X)_{F,\Si})$.
For any $F$-$C^*$-algebra $A$, the map
$$KK_*^F(C(X),A)\longrightarrow K_*(A_{F,\Si});\, 
x\mapsto [P_X]\otimes_{C(X)_{F,\Si}}\tau_{F,\Si}(x)$$ is compatible with inductive limit over $F$-invariant 
compact subset of $P_s(\Si)$ and hence gives rise to a local coarse assembly  map
$$\nu_{F,\Si,A,*}^{s} :K_*^F (P_s(\Si),A){\longrightarrow}K_*(A_{F,\Si}).$$ 
This local coarse assembly map  is  natural in the $F$-algebra. 
Furthermore, let us denote for   any positive number $s$ and $s'$ such that $s\lq s'$  by
$$q_{s,s',*}:K_* (P_s(\Si),A){\longrightarrow}K_* (P_{s'}(\Si),A)$$  the homomorphism induced by the inclusion
  $P_{s}(\Si)\hookrightarrow P_{s'}(\Si)$, then it is straightforward to check that
$$\nu_{F,\Si,A,*}^{s}=\nu_{F,\Si,A}^{s'}\circ  q_{s,s',*}.$$


\subsection{Quantitative local coarse assembly maps}
With notation of  section \ref{sec-indexmap}, if $\Si$ is proper discrete metric space equipped with an 
 action  of  a finite group $F$ by isometries, then since the action of $F$ preserves the filtration of $A\ts\K(\ell^2(\Si))$, then 
$A_{F,\Si}$ inherits from  $A\ts\K(\ell^2(\Si))$ a  structure of filtered $C^*$-algebra. Our aim is to define a quantitative version of 
the geometrical $\mu_{F,\Si,A,*}^{s}$.
The argument of the proof of theorem \ref{thm-tensor}, can be  easily adapted to prove
\begin{theorem}\label{thm-tensor-F}
There exists a control pair $(\alpha_\TT,k_\TT)$ such that 
\begin{itemize}
 \item 
for any  proper discrete metric space $\Si$ equipped with a free  action  of  a finite group $F$ by isometries;
 \item for any $F$-$C^*$-algebras $A$ and $B$;
\item  any $z$ in $KK^F_*(A,B)$, \end{itemize}
there 
exists a $(\alpha_\TT,k_\TT)$-controlled morphism
 $$\TT_{F,\Si}(z)=(\tau_{F,\Si}^{\eps,r})_{0<\eps<\frac{1}{4\alpha_\TT}}:\K_*(A_{F,\Si})\to \K_*(B_{F,\Si})$$ that satisfies the following:
\begin{enumerate}
\item$\TT_{F,\Si}(z):\K_*(A_{F,\Si})\to \K_*(B_{F,\Si})$  induces  in 
$K$-theory the right  multiplication by  the element $\tau_{F,\Si}(z)\in KK_*(A_{F,\Si},B_{F,\Si})$ defined by equation (\ref{equ-tau-F-Sigma});
\item  For any  elements $z$ and $z'$ in $KK_*^F(A,B)$ then
$$\TT_{F,\Si}(z+z')=\TT_{F,\Si}(z)+\TT_{F,\Si}(z').$$
\item Let $A'$ be a $F$-$C^*$-algebras and  let $f:A\to A'$ be a $F$-equivariant homomorphism of  $C^*$-algebras, then
$\T_{F,\Si}(f^*(z))=\T_{F,\Si}(z)\circ f_{F,\Si,*}$  for all $z$ in $KK_* (A',B)$.
\item Let $
B'$ be a  $F$-$C^*$-algebra and  let $g:B'\to B$ be a homomorphism of $C^*$-algebras then
$\TT_{F,\Si}(g_{*}(z))=g_{F,\Si,*}\circ \TT_{F,\Si}(z)$
for any $z$ in $KK_*^F(A,B')$.
\item $\TT_{F,\Si}([Id_{A}])\stackrel{(\alpha_\TT,k_\TT)}{\sim} \Id _{\K_*(A_{F,\Si})}$.
\item For any  semi-split extension of  $F$-$C^*$-algebras $0\to J\to A\to A/J\to 0$    with corresponding   element   $[\partial_{J,A}]$  of
$KK_1(A/J,J)$
 that implements the boundary map, then we have
$$\TT_{F,\Si}([\partial_{J,A}])=\DD_{J_{F,\Si},A_{F,\Si}}.$$
\end{enumerate}
\end{theorem}
We can proceed as in the proof of  theorem \ref{thm-product-tensor}  to get the compatibility of  
$\TT_{F,\Si}$ with Kasparov product.
\begin{theorem}\label{thm-product-tensor-F} There exists a control pair $(\lambda,h)$ such that the following holds :

\smallskip

 let  $F$ be a  finite group  acting freely  by isometries on a discrete metric space $\Si$  and
  let  $A_1,\,A_2$ and $A_3$  be $F$-$C^*$-algebras. Then for any
$z$ in $KK_*(A_1,A_2)$ and  any  $z'$ in
$KK_*(A_2,A_3)$, we have
$$\TT_{F,\Si}(z\ts_{A_2} z')\aeq  \TT_{F,\Si}(z') \circ\TT_{F,\Si}(z).$$
\end{theorem}

Let us set $r_{s,\eps}=sk_{\TT,\eps/{\alpha_\TT}}$ for  any $\eps$ in $(0,1/4)$ and $s>0$.
Then for any $F$-$C^*$-algebra $A$ and any $r\gq r_{\eps,s}$, the map
\begin{eqnarray}\label{eq-quant-ind}
 KK_*^F(C(X),A)&\longrightarrow& K_*^{\eps,r}(A_{F,\Si})\\
\nonumber x&\mapsto &\left(\tau_{F,\Si}^{\eps/\alpha_\TT,r/k_{\TT,\eps/\alpha_\TT}}(x)\right)
([P_X,0]_{\eps/\alpha_\TT,r/k_{\TT,\eps/\alpha_\TT}})
\end{eqnarray}
is compatible with inductive limit over $F$-invariant 
compact subset of $P_s(\Si)$ and hence gives rise to a quantitative local coarse assembly   map
$$\nu_{F,\Si,A,*}^{\eps,r,s}: K_*^F(P_s(\Si),A){\longrightarrow}K_*^{\eps,r}(A_{F,\Si}).$$ 
The quantitative local coarse assembly  maps are natural in the $F$-algebras. 
It is straightforward  to check that
\begin{itemize}
 \item $\iota^*_{\eps,\eps',r,r'}\circ \nu_{F,\Si,A,*}^{\eps,r,s}=\nu_{F,\Si,A,*}^{\eps',r',s}$ for any
positive numbers $\eps,\,\eps',\,r,\,,r'$ and $s$ such that $\eps\lq\eps'<1/4,\,r_{s,\eps}\lq r,\,r_{s,\eps'}\lq r'$ and
$r\lq r'$;
\item $\nu_{F,\Si,A,*}^{\eps,r,s'}\circ q_{s,s',*}=\nu_{F,\Si,A,*}^{\eps,r,s}$ for any
positive numbers $\eps,\,r,\,s$ and $s'$ such that $\eps<1/4,\,s\lq s'$ and $r_{s'\eps}\lq r$;
\item $\nu_{F,\Si,A,*}^{s}=\iota_*^{\eps,r}\circ \nu_{F,\Si,A,*}^{\eps,r,s}$  for any
positive numbers $\eps,\,r$ and $s$ such that $\eps<1/4$ and $r_{s,\eps}\lq r$;
\end{itemize}
Let $F$ be a finite subgroup of a finitely generated  group $\Ga$ equipped with a right invariant metric. Let us show that
$$\nu_{F,\Ga,A,*}^{\eps,r,s}:K_*^F(P_s(\Ga),A){\longrightarrow}K_*^{\eps,r}(A_{F,\Ga})$$ coincides with the composition of equation
 (\ref{equ-geom-setting}).
 Using the naturality of the map 
$\Phi_{\bullet,F,\Ga}:I_F^\Ga(\bullet)\rtr\Ga\to\bullet_{F,\Ga}$ and by construction
of $\TT_{F,\Ga}$ and $\JJ_{\Ga}^{red}$ \cite[Section 5.2]{oy2}, we get the following:
\begin{lemma}\label{lem-comp-tau-J}
 Let $F$ be a finite subgroup of  a finitely generated discrete group $\Ga$. Then for any $F$-algebras $A$ and $B$ and any 
$x$ in $KK^F_*(A,B)$, we have
$$\Phi_{B,F,\Ga,*}\circ \JJ_{\Ga}^{red}(I_F^\Ga(x))=\TT_{F,\Ga}(x)\circ \Phi_{A,F,\Ga,*}.$$
\end{lemma}
\begin{proposition}
Let $\Ga$ be a finitely generated group, let $F$ be a finite subgroup of $\Ga$ and let $A$ be a $F$-$C^*$-algebra.
Then for any $\eps\in (0,1/4)$, any $s>0$ and any $r\gq r_{\eps,s}$ the following diagram is commutative
$$\begin{CD}
 KK_*^\Gamma(C_0(P_s(\Ga)),\I(A))@>\mu_{\Ga,\I(A),*}^{s,\eps,r}>> K_*^{\eps,r}( \I(A)\rtr\Ga)\\
         @A \Ind_F^{\Ga}(P_s(\Ga))_* AA
         @VV  \Phi^{\eps,r}_{A,F,\Ga,*}V\\
K_*^F(P_s(\Ga),A)@>\nu_{F,\Ga,A,*}^{\eps,r,s}>> K_*^{\eps,r}(A_{F,\Ga})
\end{CD},$$for  $s>0,\,\eps\in(0,1/4)$ and $r\gq r_{s,\eps}$.
\end{proposition}
\begin{proof}Let us set $(\alpha,k)=(\alpha_\JJ,k_\JJ)=(\alpha_\T,k_\T)$.
Let $X$ be a $F$-invariant compact subset of $P_s(\Ga)$ and let $x$ be an element of $KK_*^F(C(X),A)$.
The definition of the quantitative assembly maps was recalled in section \ref{subsection-quantitative-assembly-map}.   We have set $$p_{\Ga,s}:\Ga\to C_0( P_s(\Ga));\, \ga\mapsto \lambda_e^{1/2}\lambda_\ga^{1/2}.$$ 
 Then    $z_{\Ga,s}=[p_{\Ga,s},0]_{\frac{\eps}{\alpha},\frac{r}{k_{{\eps}/{\alpha}}}}$ defines an element in 
$K_0^{\frac{\eps}{\alpha},\frac{r}{k_{{\eps}/{\alpha}}}}(C_0( P_s(\Ga))\rtimes\Ga)$.  Moreover, we have the equalities
\begin{eqnarray}
\nonumber   \Phi_{A,F,\Ga,*}^{\eps,r}\circ \mu_{\Ga,\I(A),*}^{s,\eps,r}\circ\Ind_F^{\Ga}(P_s(\Ga))_*(x)&=&
\Phi_{A,F,\Ga,*}^{\eps,r}\circ\big(J_\Gamma^{red,\frac{\eps}{\alpha},\frac{r}{k_{{\eps}/{\alpha}}}}(\Upsilon^{\Ga,*}_{F,X}\I(x))\big)(z_{\Ga,s})\\
\label{equ-comp-coarse1}&=&
\Phi_{A,F,\Ga,*}^{\eps,r}\circ\big(J_\Gamma^{red,\frac{\eps}{\alpha},\frac{r}{k_{{\eps}/{\alpha}}}}(\I(x))\big)\circ
\Upsilon^{\Ga,\frac{\eps}{\alpha},\frac{r}{k_{{\eps}/{\alpha}}}}_{F,X,\Ga,*}(z_{\Ga,s})\\
\label{equ-comp-coarse2}&=&\TT_{F,\Ga}^{\frac{\eps}{\alpha},\frac{r}{k_{{\eps}/{\alpha}}}}(x)
\circ \Phi_{C(X),F,\Ga,*}^{\frac{\eps}{\alpha},\frac{r}{k_{{\eps}/{\alpha}}}}\circ\Upsilon^{\Ga,\frac{\eps}{\alpha},\frac{r}{k_{{\eps}/{\alpha}}}}_{F,X,\Ga,*}(z_{\Ga,s})
\end{eqnarray}
where
\begin{itemize}
\item $\Upsilon^{\Ga}_{F,X,\Ga}:C_0(P_s)\rt\Ga\to I_F^\Ga(C(X))\rt\Ga$ is the morphism induced by $\Upsilon^{\Ga}_{F,X}$;
 \item equation (\ref{equ-comp-coarse1}) is a consequence of naturality of $\JJ_{\Ga}^{red}$ (see section \ref{sec-survey});
\item equation (\ref{equ-comp-coarse2})  is a consequence of  lemma \ref{lem-comp-tau-J}.
\end{itemize}
Since $$\Phi_{C(X),F,\Ga,*}^{\frac{\eps}{\alpha},\frac{r}{k_{{\eps}/{\alpha}}}}\circ\Upsilon^{\Ga,\frac{\eps}{\alpha},\frac{r}{k_{{\eps}/{\alpha}}}}_{F,X,\Ga,*}(z_{\Ga,s})=
[\Phi_{C(X),F,\Ga}\circ\Upsilon^\Ga_{F,X,\Ga}(p_{\Ga,s}),0]_{\frac{\eps}{\alpha},\frac{r}{k_{{\eps}/{\alpha}}}},$$
the proposition is then a consequence of the  equality
$$\Phi_{C(X),F,\Ga}\circ \Upsilon^\Ga_{F,X,\Ga}(p_{\Ga,s})=P_X.$$
\end{proof}

\subsection{A geometric  assembly map}\label{subsection-another-assembly-map}

In order to generalize proposition \ref{prop-Q-BC} to the setting of  proper discrete metric spaces equipped with an isometric action of a finite 
group
$F$, we need
\begin{itemize}
 \item an analogue in this setting of the algebra $\ell^\infty(\N,\K(\H)\ts \Ind_F^\Ga(A))\rtr\Ga$ for 
an action on a  $C^*$-algebra $A$ of a finite subgroup $F$ of a finitely generated  $\Ga$;
\item an assembly map that computes its $K$-theory.
\end{itemize}
For a family  $\A=(A_i)_{\in\in\N}$  of $F$-$C^*$-algebras, let us define $\A_{F,\Si,r}=\prod_{i\in\N}  A_{i,F,\Si,r}$ and let
$\A_{F,\Si}$ be the closure of $\cup_{r>0}\A_{F,\Si,r}$ in $\prod_{i\in\N} A_{i,F,\Si}$. Then $\A_{F,\Si}$ is obviously a filtered $C^*$-algebra. 
We set for the trivial group 
$\A_{\{e\},\Si}=\A_{\Si}$ and thus,  if $\Si$ is acted upon by a finite group $F$ by isometries, $F$ acts on 
$\A_{\Si}$ and preserves the filtration. Clearly, $\A_{F,\Si}$ is the $F$-fixed points algebra of $\A_{\Si}$.
If  $\A=(A_i)_{\in\in\N}$  is a family of $F$-$C^*$-algebras, we set
$\A^\infty=(\K(\H)\ts A_i)_{\in\in\N}$, where $\K(\H)$ is equipped with the trivial action of $F$. We can then define  $\A^\infty_{F,\Si}$
from  $\A^\infty$ as above.
For a 
$F$-$C^*$-algebra $A$, we set $A^\N=(A_i)_{\in\in\N}$ for the constant family of $F$-$C^*$-algebras $A=A_i$ for all integer $i$ 
and define from this
$A^\N_{F,\Si}$ and $A^{\N,\infty}_{F,\Si}$ as above.
For any family  $\A=(A_i)_{\in\in\N}$  of $F$-$C^*$-algebras, let   us consider the following controlled morphism
$$\mathcal{G}_{F,\Si,\A,*}=(G^{\eps,r}_{F,\Si,\A})_{0<\eps<1/4,r>0}:\K_*(\A^{\infty}_{F,\Si})\to\prod_{i\in\N}\K_*(A_{i,F,\Si}),$$
where 
\begin{equation*}G^{\eps,r}_{F,\Si,\A,*} :  K^{\eps,r}_*(\A^{\infty}_{F,\Si}){\longrightarrow}\prod_{i\in\N} K^{\eps,r}_*(A_{i,F,\Si})
\end{equation*} is  the map induced on
the $j$ th factor and up to the Morita  equivalence by the restriction to  $\A^{\infty}_{F,\Si}$  of the  evaluation 
$\prod_{i\in\N} \K(\H)\ts A_{i,F,\Si}\to \K(\H)\ts A_{j,F,\Si}$  at $j\in\N$. 
 As a consequence of lemma \ref{lem-prod-filtered}, we have the following.
\begin{lemma}\label{lemma-prod}There exists a control pair $(\alpha,h)$ such that
\begin{itemize}
 \item for any  finite group $F$;
 \item for any proper discrete metric space $\Si$  provided with an action of $F$ by isometries;
\item for any families $\A=(A_i)_{i\in\N}$ of $F$-algebras,
\end{itemize}
then $\mathcal{G}_{F,\Si,\A,*}:\K_*(\A^{\infty}_{F,\Si})\to\prod_{i\in\N}\K_*(A_{i,F,\Si})$
is  a  $(\alpha,h)$-controlled isomorphism.
\end{lemma}
For any families of $F$-$C^*$-algebras $\A=(A_i)_{i\in\N}$ and $\B=(B_i)_{i\in\N}$ of $F$-$C^*$-algebras and any
family  $f=(f_i:A_i\to B_i)_{i\in\N}$  of 
$F$-equivariant homomorphisms, let us set $$f_{\Si,F}=\prod_{i\in\N}f_{i,\Si,F}: \A_{F,\Si}\longrightarrow \B_{F,\Si}$$
and $$f^\infty_{\Si,F}=\prod_{i\in\N} Id_{\K(\H)}\ts f_{i,\Si,F}: \A^\infty_{F,\Si}\longrightarrow \B^\infty_{F,\Si}.$$
Then together with   theorem \ref{thm-tensor-F}, lemma  \ref{lemma-prod} yields to
\begin{corollary}\label{cor-tensor-infty}
There exists a control pair $(\alpha,h)$ such that 
\begin{itemize}
 \item 
for any  proper discrete metric space $\Si$ equipped with a free  action  of  a finite group $F$ by isometries;
 \item for any families of $F$-$C^*$-algebras $\A=(A_i)_{i\in\N}$ and $\B=(B_i)_{i\in\N}$;
\item  for any $z=(z_i)_{i\in\N}$ in $\prod_{i\in\N}KK^F_*(A_i,B_i)$, \end{itemize}
there 
exists a $(\alpha,h)$-controlled morphism
 $$\TT^\infty_{F,\Si}(z)=(\tau_{F,\Si}^{\infty,\eps,r}(z))_{0<\eps<\frac{1}{4\alpha},r>0}:\K_*(\A^\infty_{F,\Si})\to 
\K_*(\B^\infty_{F,\Si})$$ that satisfies the following:
\begin{enumerate}
\item  For any  elements  $z=(z_i)_{i\in\N}$ and  $z'=(z'_i)_{i\in\N}$ in $\prod_{i\in\N}KK^F_*(A_i,B_i)$, then 
$$\TT^\infty_{F,\Si}(z+z')=\TT^\infty_{F,\Si}(z)+\TT^\infty_{F,\Si}(z')$$ for $z+z'=(z_i+z'_i)_{i\in\N}$
\item Let $\A'=(A'_i)_{i\in\N}$ be a family of $F$-$C^*$-algebras and  let $f=(f_i:A'_i\to A_i)_{i\in\N}$  be a family  of 
$F$-equivariant homomorphisms of  $C^*$-algebras. Then
$\T^\infty_{F,\Si}(f^*(z))=\T^\infty_{F,\Si}(z)\circ f^\infty_{F,\Si,*}$  for all  
$z=(z_i)_{i\in\N}$ in $\prod_{i\in\N}KK^F_*(A_i,B_i)$, where $f^*(z)=(f_i^ *(z_i))_{i\in\N}$.
\item Let $\B'=(B_i)_{i\in\N}$  be a  family of $F$-$C^*$-algebras  and  
let $g=(g_i:B_i\to B'_i)_{i\in\N}$  be a family of 
$F$-equivariant homomorphism of  $C^*$-algebras. 
Then
$\T^\infty_{F,\Si}(g_*(z))=g^\infty_{F,\Si,*}\circ\T^\infty_{F,\Si}(z)$  for all  
$z=(z_i)_{i\in\N}$ in $\prod_{i\in\N}KK^F_*(A_i,B_i)$, where $g_*(z)=(g_{i,*}(z_i))_{i\in\N}$ 
\item If we set $Id_\A=(Id_{A_i})_{i\in\N}$, then $\TT^\infty_{F,\Si}([Id_\A])\stackrel{(\alpha,k)}{\sim} \Id _{\K_*(\A^\infty_{F,\Si})}$.
\item For any  family of semi-split extensions of  $F$-$C^*$-algebras $$0\to J_i\to A_i\to A_i/J_i\to 0$$    with corresponding   element   $[\partial_{J_i,A_i}]$  of
$KK_1(A_i/J_i,J_i)$
 that implements the boundary maps,  let us  set $\mathcal{J}=(J_i)_{i\in\N},\,  \mathcal{A}=(A_i)_{i\in\N},\,\mathcal{A}/\mathcal{J}=(A_i/J_i)_{i\in\N}$   and $[\partial_{\mathcal{J},\A}]=([\partial_{J_i,A_i}])_{i\in\N}\in    \prod_{i\in\N}KK_1^\Ga(A_i/J_i,J_i)$. Then we we have
$$\TT_{F,\Si}^\infty([\partial_{\mathcal{J},\A}])=\DD_{\mathcal{J}^\infty_{F,\Si},\A^\infty_{F,\Si}}.$$\end{enumerate}
\end{corollary}
As a consequence    theorem \ref{thm-product-tensor-F} and of  lemma \ref{lemma-prod} we get
\begin{proposition}\label{prop-product-tensor-infty} There exists a control pair $(\lambda,h)$ such that the following holds :

\smallskip

 let  $F$ be a  finite group  acting freely  by isometries on a discrete metric space $\Si$  and
  let  $\A=(A_i)_{i\in\N},\,\B=(B_i)_{i\in\N}$ and $\B'=(B'_i)_{i\in\N}$  be families of $F$-$C^*$-algebras. Let us set  
$z\ts_{\B} z'=(z_i\ts_{B_i} z'_i)_{i\in\N}$ for any
$z=(z_i)_{i\in\N}$ in $\prod_{i\in\N}KK^F_*(A_i,B_i)$ and  any  $z'=(z'_i)_{i\in\N}$ in $\prod_{i\in\N}KK^F_*(B_i,B'_i)$. Then we have
$$\TT^\infty_{F,\Si}(z\ts_{\B} z')\aeq  \TT^\infty_{F,\Si}(z') \circ\TT^\infty_{F,\Si}(z).$$
\end{proposition}
 If $F$ is a finite group   and if  $\A=(A_i)_{i\in\N}$ 
is   a family of of $F$-$C^*$-algebras, let us consider the family $\A\ts \K(\l^2(F))=(A_i\ts\K(\l^2(F)))_{i\in\N}$, 
provided by the diagonal action of
$F$ where the action on  $\K(\l^2(F))$ is induced with the right  regular representation. If moreover $F$ acts on $\Si$ by 
isometries, $\A_\Si^\infty$ is indeed a $F$-$C^*$-algebra and we have a natural identification of filtered $C^*$-algebras
\begin{equation}\label{equ-id-crossprod}\A_\Si^\infty\rtimes F\cong (\A\ts \K(\l^2(F)))_{F,\Si}^\infty,\end{equation}
where  $\A_\Si^\infty\rtimes F$ is filtrered by 
$(C(F,\A^\infty_{\Si,r}))_{r>0}$ .
Applying corollary \ref{prop-product-tensor-infty} to the family
$M_{\A,F}=(M_{A_i,F})_{i\in\N}\in\prod_{i\in\N}KK^F_*(A_i,A_i\ts \K(\l^2(F)))$ of $F$-equivariant Morita equivalences, we get

\begin{lemma}\label{lem-morita-inf}
 There exists a control pair $(\alpha,h)$ such that for any finite group $F$, any  family  $\A=(A_i)_{i\in\N}$ of $F$-$C^*$-algebras,
and any discrete metric space $\Si$ equipped with a free action of $F$ by isometries, then, under the 
identification of equation (\ref{equ-id-crossprod}),
$$\MM^\infty_{\A,F}\defi\T^\infty_{F,\Si}(M_{\A}):\K_*(\A^\infty_{F,\Si})\longrightarrow \K_*(\A^\infty_\Si\rtimes F)$$ is a $(\alpha,h)$-controlled isomorphism.
\end{lemma}
 Recall that to any $F$-invariant compact subset $X$ of $P_s(\Si)$ is associated a projection $P_X$ of $C(X)_{F,\Si}$. Indeed for every $x$ in $X$,
then $P_X(x)$ is the matrix with almost all vanishing entries indexed by $\Si\times\Si$ defined by
$P_X(x)_{\si,\si'}=\lambda_\si(x)^{1/2}\lambda_{\si'}(x)^{1/2}$ (recall that $(\la_\si)_{\si\in\Si}$ is the set of coordinate functions on $P_r(\Si)$).
For any  family 
$\X=(X_i)_{i\in\N}$ of compact $F$-invariant subsets of $P_s(\Si)$, let us set $\mathcal{C}_{\X}=(C(X_i))_{i\in\N}$ and consider
the projection $P^\infty_{\X}=(P_{X_i}\ts e)_{i\in\N}$ of $\mathcal{C}_{\X,F,\Si}^\infty$, where $e$ is a fixed rank one 
projection of $\K(\H)$.
 The  propagation of $P^\infty_{\X}$ is  less than $s$. Hence for the control pair $(\alpha,h)$ of corollary \ref{cor-tensor-infty}, 
any  family $\A=(A_i)_{i\in\N}$ of $F$-$C^*$-algebras, any $\eps\in(0,1/4)$, any $s>$ and any $r\gq r_{s,\eps}$, then 
 the map  
$$\prod_{i\in\N}KK^F_*(C(X_i),A_i)\to K^{\eps,r}_*(\A^{\infty}_{F,\Si}); z\mapsto 
\tau_{F,\Si}^{\infty,\eps/\alpha,r/h_{\eps/\alpha}}(z)(P^\infty_{\X})$$ is 
compatible with  inductive limit of 
 families
$\X=(X_i)_{i\in\N}$ of compact $F$-invariant subset of $P_s(\Si)$. By composition with the controlled isomorphism
$$\T^\infty_{F,\Si}(M_{\A}):\K_*(\A^\infty_{F,\Si})\longrightarrow \K_*(\A^\infty_\Si\rtimes F),$$ we get for  a function $(0,1/4)\times (0,\infty)\to (0,\infty);\,(\eps,s)\mapsto r_{s,\eps}$  non-decreasing in $s$,
non increasing in $\eps$  and  independant on $F,\,\Si$ and $\A$ and 
 for any  
$\eps$ in  $(0,1/4)$, any positive numbers $s$ and $r$  such that  $r\gq  r_{s,\eps}$
a    quantitative geometric assembly map
$$\nu_{F,\Si,\A,*}^{\infty,\eps,r,s}: \prod_{i\in\N}K_*^F(P_s(\Si),A_i){\longrightarrow}
K_*^{\eps,r}(\A^{\infty}_{\Si}\rtimes F).$$
Therefore, for $s$ a fixed positive number, 
the bunch of
maps $(\nu_{F,\Si,\A,*}^{\infty,\eps,r,s})_{\eps>0,r\gq r_{s,\eps}}$ gives rise to a geometric  assembly map
$$\nu_{F,\Si,\A,*}^{\infty,s}: \prod_{i\in\N}K_*^F(P_s(\Si),A_i){\longrightarrow}
K_*(\A^{\infty}_{\Si}\rtimes F)$$   unically defined by
 $\nu_{F,\Si,\A,*}^{\infty,s}=\iota_*^{\eps,r}\circ \nu_{F,\Si,\A,*}^{\infty,\eps,r,s}$  for any
positive numbers $\eps,\,r$ and $s$ such that $\eps<1/4$ and $r\gq r_{s,\eps}$.

The quantitative  assembly maps $\nu_{F,\Si,\A,*}^{\infty,\eps,r,s}$ are compatible
with inclusions of Rips complexes: let
 \begin{equation}\label{equ-syst-induc}
  q_{s,s',*}^\infty: \prod_{i\in\N}K_*^F(P_s(\Si),A_i){\longrightarrow} \prod_{i\in\N}K_*^F(P_s(\Si),A_i)
 \end{equation}
be the map induced by the inclusion
 $P_s(\Si)\hookrightarrow  P_{s'}(\Si)$, then we have
 $$\nu_{F,\Si,\A,*}^{\infty,\eps,r,s'}\circ q_{s,s',*}^\infty=\nu_{F,\Si,\A,*}^{\infty,\eps,r,s}$$ for any
 positive numbers $\eps,\,s,\,s,'$ and $r$  such that 
 $\eps\in(0,1/4),\,s\lq s',\, r\gq r_{s',\eps}$, and thus
 $$\nu_{F,\Si,\A,*}^{\infty,s'}\circ q_{s,s',*}^\infty=\nu_{F,\Si,\A,*}^{\infty,s}$$ for any
 positive numbers $s$ and $s'$  such that 
$s\lq s'$. 

\smallskip

Eventually, we can take the inductive  limit over the degree of the Rips complex and  set
$$K_*^{{top},\infty}(F,\Si,\A)=\lim_{s>0,} \prod_{i\in\N}K_*^F(P_s(\Si),A_i)=\lim_{ s>0,(X^s_i)_{i\in\N}}\prod_{i\in\N} KK_*^F(C(X^s_i),A_i),$$ 
where in the inductive limit on the right hand   side,  $s$ runs through positive numbers and  $(X^s_i)_{i\in\N}$ runs through 
families of $F$-invariant compact subset of $P_s(\Si)$.
We get then  an assembly map
\begin{equation}\label{equ-Ainf}
 \nu_{F,\Si,\A,*}^{\infty}:K_*^{{top},\infty}(F,\Si,\A)\longrightarrow K_*(\A^{\infty}_{\Si}\rtimes F).
\end{equation}

\subsection{The groupoid approach}
In order to generalize the proof of   proposition \ref{prop-Q-BC} in the setting of dicrete metric space, our purpose in this
section is to follow  the route of \cite{sty} and to show that if  $\A=(A_i)_{i\in\N}$ is a family of $C^*$-algebras, then 
 $\A^{\infty}_{\Si}$ is the reduced crossed product of the algebra
 $\prod_{i\in\N}C_0(\Si,A_i\ts\K(H))$ by the diagonal action of the groupoid attached to the coarse structure 
of the discrete metric space $\Si$.

 In \cite{sty} was associated  to a discrete metric space $\Si$  with bounded geometry a groupoid
$G(\Si)$ with unit space the Stone-C\v ech compactification $\beta_\Si$ of $\Si$ and such that the Roe algebra of
$\Si$ is the reduced crossed product of $\ell^\infty(\Si,\K(H))$ by an action of $G(\Si)$. Let us describe the construction of 
this groupoid. If $(\Si,d)$ is a discrete metric space with bounded geometry. Then a subset $E$ of $\Si\times\Si$ 
is called an  entourage for $\Si$ if there exists $r>0$ such that $$E\subset\{(x,y)\in \Si\times\Si\text{ such that } d(x,y)<r\}.$$If 
$E$ is  an entourage for $\Si$, set $\bar{E}$ for its closure in  the Stone-C\v ech compactification $\beta_{\Si\times\Si}$ of 
$\Si\times\Si$.  Then there is a unique structure of groupoid on 
$G_\Si=\cup_{E\text{ entourage}}\bar{E}\subset \beta_{\Si\times\Si}$ with unit space the Stone-C\v ech compactification
$\beta_\Si$ of $\Si$ which extends the 
groupoid of pairs $\Si\times\Si$. 

 For 
a family $\A=(A_i)_{i\in\N}$ of $C^*$-algebras, let us set $\A_{C_0(\Si)}=\prod_{i\in\N} C_0(\Si,A_i)$. 
Then the diagonal action of  $\ell^\infty(\Si)$ by multiplication clearly provides $\ac$ with a   structure of $C(\beta_\Si)$-algebra.
Our aim is to show that $G_\Si$ acts diagonally on $\ac$ and that $\A_{C_0(\Si)}\rtr G_\Si$ is canonically isomorphic to $\A_{\Si}$.

\medskip
Let $C_0(G_\Si,\A)$ be the closure in $\prod_{i\in\N} C_0(\Si\times\Si,A_i)$ of $$\{(f_i)_{i\in \N};\,
\exists r>0;\, \forall i\in\N,\, \forall(\si,\si')\in\Si^2, d(\si,\si')>r\Rightarrow f_i(\si,\si')=0\}.$$
For an entourage $E$ and an element $f=(f_i)_{i\in \N}$ of $\ac$, let us define
$f_r^E=(f_{r,i}^E)_{i\in\N}$ and $f_s^E=(f_{s,i}^E)_{i\in\N}$ by 
$f_{r,i}^E(\si,\si')=\chi_E(\si,\si')f_i(\si)$ and $f_{s,i}^E(\si,\si')=\chi_E(\si,\si')f_i(\si')$
 for any integer $i$ and any $\si$ and $\si'$ in $\Si$.

\begin{lemma}
Let  $\A=(A_i)_{i\in\N}$ be a family of $C^*$-algebras. Then we have isomorphisms of $C(G_\Si)$-algebras
 $$\Psi_r:r^*\ac\to C_0(G_\Si,\A)$$ and  $$\Psi_s:s^*\ac\to C_0(G_\Si,\A)$$ 
only defined by $\Psi_r( \chi_E\ts_r f)= f_r^E$ and  $\Psi_s( \chi_E\ts_s f)= f_s^E$ for any $f$ in $\ac$ and any entourage $E$ for $\Si$.
\end{lemma}
\begin{proof}
 Is is clear that $\Psi_r$ and $\Psi_s$ are well and only defined by the formula above and are isometries.
Let us prove for instance that $\Psi_r$ is an isomorphism. Surjectivity of $\Psi_r$ amounts to prove that for any $(h_i)_{i\in\N}$ in 
$\prod_{i\in\N}C_0(\Si\times\Si,A_i)$ and any entourage $E$ then $h=(\chi_Eh_i)_{i\in\N}$ is in the range of $\Psi_r$.
According to  \cite[Lemma 2.7]{sty}, we can assume that the restrictions $s:E\to\Si$ and 
$r:E\to \Si$ are one-to-one. For any integer $i$, then define $f_i:\Si\to A_i$ by 
\begin{itemize}
 \item $f_i(\si)=h_i (\si,\si')$ if there exists $\si'$ such that $(\si,\si')$ is in $E$;
\item $f_i(\si)=0$ otherwise.
\end{itemize}
 Then $f_i$ is in $C_0(\Si,A_i)$ for every integer $i$ and if we set $f=(f_i)_{i\in\N}$, then $f_r^E=h$ and hence $h$ is in the range of $\Psi_r$.
\end{proof}

Let us define $V_\Si=\Psi_r\circ \Psi_s^{-1}$. Then $V_\Si:s^*\ac\to r^*\ac$ is an isomorphism of $C(G_\Si)$-algebras that can be describe on
elementary tensors as follows.
For an  entourage $E$ such that the restrictions  $s:E\to \Si$ and  $r:E\to \Si$ are one-to-one,
then for every $\si$ in $r(E)$ there exists a unique
$\si'$ in $s(E)$ such that $(\si,\si')$ is in $E$.  For  any $f=(f_i)_{i\in\N}$ in $\ac$, we define
 $E\circ f=(E\circ f_i)_{i\in\N}$ in $\ac$, where  for any integer $i$,
\begin{itemize}
 \item  $E\circ  f_i (\si)= f_i (\si')$ if $\si$ is in $r(E)$ and $(\si,\si')$ is in $E$;
\item  $E\circ  f_i(\si)=0$ otherwise.
\end{itemize}
Then under above assumptions, we have  $V_\Si( \chi_E\ts_r f)= \chi_E\ts_r E\circ f$.

 \begin{lemma} For every  family $\A=(A_i)_{i\in\N}$ of $C^*$-algebras, then 
$$V_\Si:s^*\ac\to r^*\ac$$  is an action of the  groupoid $G_\Si$ on $\ac$.
\end{lemma}
\begin{proof}
For an element $\ga$ in $G_\Si$, let $V_{\Si,\ga}:\ac_{s(\ga)}\to\ac_{r(\ga)}$ be the map induced by $V_\Si$ on the fiber of $\ac$ at 
$s(\ga)$.
Let   $\ga$ and $\ga'$ be elements
in $G_\Si^\infty$ such that $s(\ga)=r(\ga')$. Let $E$ and $E'$ be entourages such that the restrictions
of $s$ and $r$ to $E$ and $E'$ are one-to-one and such that $\ga\in\bar{E}$ and $\ga'\in\bar{E'}$. 
Let us set $$E\circ E'=\{ (\si,\si'')\in\Si\times\Si ;\exists\,\si'\in\Si ;
(\si,\si')\in E\text{ and } (\si',\si'')\in E'\}.$$Then
$\ga\cdot\ga'$ is in $\overline{E\circ E'}$ and the restrictions of $s$ and $r$ to $E\circ E'$ is one-to-one. Moreover,  we clearly have
$(E\circ E')\circ f= E\circ (E'\circ f)$  for  all $f$  in 
$\A_{C_0(\Si)}$. Hence, we get
\begin{eqnarray*}
 V_{\Si,\ga\cdot\ga'}(f_{s(\gamma')})&=&(E\circ E'\circ f)_{r(\gamma)}\\
&=&V_{\Si,\ga}((E'\circ f)_{s(\gamma)})\\
&=&V_{\Si,\ga}((E'\circ f)_{r(\gamma')})\\
&=&V_{\Si,\ga}\circ V_{\Si,\ga'}(f_{s(\gamma)})
\end{eqnarray*}
\end{proof}

\begin{proposition} \label{prop-iso-cross-product}
 Let $\Si$ be a discrete metric space with bounded geometry  and let  $\A=(A_i)_{i\in\N}$ be a family of
  $C^*$-algebras.  Then 
we have a natural isomorphism  
 $$\II_{\Si,\A}:\ac\rtr\gsi\stackrel{\cong}{\longrightarrow}\A_{\Si}.$$\end{proposition}
\begin{proof}
Following the proof of \cite{sty}, we obviously have that  
$J_{\Si,A}=\oplus_{i\in\N}C_0(\Si,A_i)$ is a $\gsi$-invariant ideal of $\A_{C_0(\Si)}$.
 For any $\si'$ in $\Si$, we have at any element of $\Si$ a canonical  identification of the fibre of $J_{\Si,A}$   with 
$\oplus_{i\in\N}A_i$ and under this identification, the action of $\Si\times\Si\subset \gsi$ on  $J_{\Si,A}$ is trivial.
According to \cite[lemma 4.3]{sty}, the reduced crossed product  $\A_{C_0(\Si)}\rtimes_r\gsi$ is faithfully represented
in the right $J_{\Si,A}$-Hilbert module 
$$L^2(\gsi,J_{\Si,A})\cong L^2(\gsi,\A_{C_0(\Si)})\ts_{\A_{C_0(\Si)}}J_{\Si,A}.$$ But we have
a natural identification of $J_{\Si,A}$-right Hilbert modules
 $$L^2(\gsi,J_{\Si,A})\cong C_0\left(\Si,\left(\oplus_{i\in\N} A_i \right)\otimes\ell^2(\Si)\right).$$
Under this identification, the representation of  $\A_{C_0(\Si)}\rtimes_r\gsi$  indeed arise from 
a pointwise action   on  $\left(\oplus_{i\in\N} A_i \right)\otimes\ell^2(\Si)$. As such, the underlying   representation  of 
$\A_{C_0(\Si)}\rtimes_r\gsi$ on $\left(\oplus_{i\in\N} A_i \right)\otimes\ell^2(\Si)$ is faithfull.
Let us describe this action.
\begin{itemize}
 \item an element $f=(f_i)_{i\in\N}$ in  $\A_{C_0(\Si)}\cong \prod_{i\in\N}A_i\ts C_0(\Si)$ acts on
 $\left(\oplus_{i\in\N} A_i \right)\otimes\ell^2(\Si)$ in the obvious way.
\item If $E$ is an entourage, then the action of $\chi_E$ on  $\left(\oplus_{i\in\N} A_i \right)\otimes\ell^2(\Si)$  is by pointwise multiplication
by $Id_{\oplus_{i\in\N} A_i }\ts T_E$, where   the operator $T_E$ is defined by 
$T_{E,\si,\si'}=\chi_E(\si,\si')$ for any $\si$ and $\si'$ in $\Si$.
\end{itemize}
The algebra $\A_\Si$ acts also faithfully on   $\left(\oplus_{i\in\N} A_i \right)\otimes\ell^2(\Si)$ by pointwise action at each integer $i$ of 
$A_i\ts\K( \ell^2(\Si))$ on  $A_i\ts\ell^2(\Si)$. It is then clear that if $f$ is in $\A_{C_0(\Si)}$ and $E$ is an 
entourage, then $fT_E$ is in $\A_\Si$. Conversely,  let us show any element in $A_\Si$  acts on $\left(\oplus_{i\in\N} A_i \right)\otimes\ell^2(\Si)$
as an element of $\A_{C_0(\Si)}\rtimes_r\gsi$. Let $(T_i)_{i\in\N}$ be an element of $A_{\Si,r}$. We can assume that
for every integer $i$, there exists a finite subset $X_i$ of $\Si$ such that, $T_i=(T_{i,\si,\si'})_{(\si,\si')\in\Si^2}$ lies indeed in
$A_i\ts \K(\ell^2(X_i))$. Applying  \cite[Lemma 2.7]{sty} to the union of the support of the $T_i$ when $i$ runs through integers,
we can actually assume 
without loss of generality that there exists an entourage $E$ such that
\begin{itemize}
 \item the restrictions
of $s$ and $r$ to $E$ are one-to-one;
\item for any integer $i$ and any $\si$ and $\si'$ in $\Si$, then $T_{i,\si,\si'}\neq 0$ implies that $(\si,\si')$ is in $E$.
\end{itemize}
Define then for any integer $i$
\begin{itemize}
 \item $f_i(\si)=T_{i,\si,\si'}$ if there  exists $\si'$ in $X_i$ such that $(\si,\si')$ is in $E\cap (X_i\times X_i)$.
 \item $f_i(\si)=0$ otherwise.
\end{itemize}
Then $f_i$ is in $C_0(\Si,A_i)$ for every integer $i$ and if we set $f=(f_i)_{i\in\N}$, then $fT_E$ acts on  $\left(\oplus_{i\in\N} A_i\right) \otimes\ell^2(\Si)$
as  $(T_i)_{i\in\N}$.
\end{proof}

If $\Si$ is equipped with an action of a finite group $F$ by isometries, then the diagonal action of $F$ on $\Si$ induces an 
action of $F$ on $\gsi$ by automorphisms of groupoids. Moreover, for any family 
$\A=(A_i)_{i\in\N}$ of $F$-$C^*$-algebra, the action of $\gsi$
 on  $\ac=\prod_{i\in\N} C_0(\Si,A_i)$ is covariant with respect the pointwise diagonal action of $F$.
Hence, we end up in this way with an action of $F$ on $\ac\rtimes_r\gsi $ by automorphisms. 
Namely, let us  consider the semi-direct product groupoid $$\gsif=\gsi\rtimes F=\{(\ga,x)\in \gsi \times F\}$$ provided
with the source map $$\gsif \to \beta_\Si;\,(\ga,x)\mapsto s(x^{-1}(\ga))$$ and range map
$$\gsif\to \beta_\Si;\,(\ga,x)\mapsto r(\ga)$$ and composition rule 
$(\ga,x)\cdot (\ga',x')=(\ga\cdot x(\ga'),xx')$ if $s(x^{-1}(\ga))=r(\ga')$. Then $\ac$ is actually a $\gsif$-$C^*$-algebra 
and we have a natural identification 
\begin{equation}\label{equ-identification-crossed-product}(\ac\rtimes_r\gsi)\rtimes F\cong \ac\rtimes_r\gsif.\end{equation}
                                                                                                                       
On the other hand,
$F$ also acts for each integer $i$ on $\K(\l^2(\Si))\ts A_i$ and hence  pointwisely on $\A_\Si$.
 The isomorphism 
of proposition \ref{prop-iso-cross-product}
is   then clearly $F$-equivariant and hence gives rise under then identification of equation
 (\ref{equ-identification-crossed-product}) to an isomorphism
\begin{equation}
 \label{equ-iso-cross-product}
\II_{F,\Si,\A}: \ac\rtimes_r\gsif\stackrel{\cong}{\longrightarrow} \A_{\Si}\rtimes F\end{equation}

Since $C(\beta_{\N\times\Si})\cong\ell^\infty(\N\times\Si)$, then $\ac$ is  for any family $\A$   a $C(\beta_{\N\times\Si})$-algebra. Let us show that 
$\beta_{\N\times\Si}$ is actually provided with an  action of $\gsi$ on the right that makes $\ac$ into a $\bns\rt\gsi$-algebra. 

Let $p:\bns\to \bs$ be the (only) map  extending  the projection $\N\times\Si\to \Si$ by continuity. Let $x$ be an element of $\bs$, let $\ga$ be an
element of $\gsi$ such that $r(\ga)=x$ and let $E\subset \Si\times\Si$ be an entourage such that 
\begin{itemize}
 \item $\ga$ belongs to $\bar{E}$.
\item  the restrictions
of $s$ and $r$ to $E$ are one-to-one.
\end{itemize}
Let $(n_k,\si_k)_{k\in\N}$ be a sequence in $\N\times\Si$ converging to $z$ in $\bns$ and such that $\si_k$ is in $r(E)$ for every integer $k$.
For any  integer $k$, let $\si'_k$ be the unique element of $s(E)$  such that $(\si_k,\si'_k)$ is in $E$. 
Then the sequence $(n_k,\si'_k)_{k\in\N}$ converge in $\bns$ to
an element $z'$ such that $p(z')=s(\ga)$. This limit does not depend on the choice of $E$ and $(n_k,\si_k)_{k\in\N}$  that satisfy the 
conditions above
and if we set $z\cdot\ga=z'$, we obtain an action of $\gsi$ on $\bns$ on the left. Obviously, the restriction of $\bns\rt\gsi$ to the 
saturated open subset $\N\times\Si$ of $\bns$ is the union of groupoid of pair on $\{n\}\times\Si$.
If $\A$ is a family of $C^*$-algebras, the multiplier action of  $C(\beta_{\N\times\Si})$ is $\gsi$-equivariant and hence we end up with an action of
$\bns\rt\gsi$ on $\ac$.

If $\Si$ is endowed with an action of a finite group $F$ by isometries, then the diagonal action of $F$ on $\N\times\Si$ (trivial on $\N$) gives rise
to an action of $F$ on $\bns$ by homeomorphisms which makes the action of $\gsi$ covariant. Hence  $\bns$ is  provided with an action of 
$\gsif=\gsi\rt F$.
Moreover, if $\A$ is a family of $F$-$C^*$-algebras,  then $\ac$ is a $\bns\rtimes\gsif$-algebra.

Consider now the spectrum $\bnsz$ of the ideal $\l^\infty(\N,C_0(\Si))$ of  $C(\beta_{\N\times\Si})\cong\ell^\infty(\N\times\Si)$.
Then $\bnsz$ is a saturated open subset of $\bns$, the pointwise multiplication of $\ell^\infty(\N,C_0(\Si))$ on $\ac=\prod_{i\in\N}C_0(\Si,A_i)$ provides
$\ac$ with a structure of $C(\bnsz)$-algebra  and thus we see that $\ac$ is indeed a $\bnsz\rt\gsi$-algebra.
The three crossed products $\ac\rtr\gsi$,  $\ac\rtr(\bns\rt\gsi)$ and $\ac\rtr(\bnsz\rt\gsi)$ coincide. 
If  $\Si$ is equipped  with an action of a finite group $F$ by isometries, then $\bnsz$ is $F$-invariant and hence endowed with an action of 
$\gsif$. Moreover, for any
 family $\A$ of $F$-$C^*$-algebras,  then $\ac$ is $\bnsz\rt\gsif$-algebra.
 Let us set then $\gsin$
(resp. $\gsinz$) 
for the groupoid $\bns\rt\gsi$.
 (resp.  $\bnsz\rt\gsi$), and if 
 $\Si$ is provided with an action of a finite group $F$ by
isometries, set then $\gsinf=\gsin\rt F$.
\begin{lemma}\label{lemma-partition-Nentourage}
Let $E$ be a subset of $\N\times\Si\times\Si$ and assume that there exists $r>0$ such that 
for all integer $i$ and all $\si$ and $\si'$ in $\Si$, then $(i,\si,\si')$ in $E$ implies that $d(\si,\si')<r$. Then there exists
\begin{itemize}
 \item $f_1,\ldots,f_k$ in $\ell^\infty(\N\times\Si)$;
 \item $E_1,\ldots,E_k$ entourages of $\Si$  included in $\bigcup_{i\in\N}\{(\si,\si')\in\Si^2;(i,\si,\si')\in E\}$,
\end{itemize}
such that
$\chi_E(i,\si,\si')=\sum_{j=1}^k f_j(i,\si)\chi_{E_j}(\si,\si')$ for all integer $i$ and all $\si$ and $\si'$ in $\Si$.
 \end{lemma}
\begin{proof}
 Let us set $E_1=\cup_{i\in\N}\{(\si,\si')\in\Si^2;(i,\si,\si')\in E\}$. 
 Using \cite[Lemma 2.7]{sty}, we can assume without loss of generality that 
 the restrictions
of $s$ and $r$ to $E_1$ are one-to-one. Define then $f_1:\N\times\Si\to\C$ by
 \begin{itemize}
  \item $f_1(i,\si)=1$ if there exists $\si'$ in $\Si$ such that $(i,\si,\si')$ in in $E$;
  \item $f_1(i,\si)=0$ otherwise.
 \end{itemize}
Then $\chi_E(i,\si,\si')= f_1(i,\si)\chi_{E_1}(\si,\si')$ for all integer $i$ and all $\si$ and $\si'$ in $\Si$.
\end{proof}

\section{The Baum-Connes assembly map for $(\gsif,\,\ac)$}\label{section-BC-for-GSI}
Recall that the definition of the  Baum-Connes assembly map has been extended to the setting of groupoids in \cite{tumoy}.
Let  $G$ be  a locally compact  groupoid equipped with a Haar system  and let $B$ be a $C^*$-algebra acted upon by $G$. Then there is an 
assembly map $$\mu_{G,B,*}:K^{top}_*(G,B)\to K_*(B\rtimes_r G),$$ where $K^{top}_*(G,B)$ is the topological $K$-theory for the
groupoid $G$ with coefficients in $B$. Our aim in this section is to describe the left hand side of this assembly map
 for the action of $\gsi$ on $\aci$ and then  to show that  the Baum-Connes conjecture is equivalent to the bijectivity of 
the geometric assembly map 
 $$\nu_{F,\Si,\A,*}^{\infty}:K_*^{{top},\infty}(F,\Si,\A){\longrightarrow}
K_*(\A^{\infty}_{\Si}\rtimes F)$$ defined in \ref{subsection-another-assembly-map}. Using   result of \cite{sty} on the Baum-Connes conjecture for groupoid affiliated to coarse structures, 
we get examples of coarse spaces that satistifies the permanence approximation property.
Notice that  $\gsinf$ is clearly a $\si$-compact and \'etale groupoid and that  according to  \cite[Lemma 4.1]{sty}, 
the Baum-Connes conjectures for the action of $\gsif$ on $\aci$ and for the action of 
of $\gsinf$ on $\aci$ are indeed equivalent.
\subsection{The classifying space for proper actions of the  groupoid  $\gsin$}
 For  a $\si$-compact and \'etale groupoid $G$, the following description for the left hand side of
the assembly map was given in  \cite[Section 3]{tucoarsegroupoid}.  Let $K$ be a compact subset of $G$ and let us  consider
   the space $P_K(G)$ of probability measures $\eta$  on
$G$ such that for all $\ga$ and $\ga'$  in the support of $\eta$,
\begin{itemize}
 \item   $\ga$ and $\ga'$ have same range;
\item   $\ga^{-1}\cdot\ga'$ is in $K$.
\end{itemize}
We endowed
 $P_K(G)$ with the
 weak-$*$ topology, and equip  it with the natural left action of $G$. 
Then according to \cite[Proposition 3.1]{tucoarsegroupoid}, the  action of $G$ on $P_K(G)$ is proper and cocompact. If $K\subseteq K'$ is an inclusion 
of  
compact subsets of $G$, then for any $G$-algebra $B$,  the inclusion $P_K(G)\hookrightarrow P_{K'}(G)$ induces a  morphism 
$KK^G_*(C_0(P_K(G)),B)\longrightarrow KK^G_*(C_0(P_{K'}(G)),B)$ and we have
$$K^{top}_*(G,B)= \lim_K KK^G_*(C_0(P_K(G)),B),$$ where in the inductive  limit,
 $K$ runs though compact subsets of $G$. 
If the groupoid $G$ is provided with an action of a finite group $F$ by automorphisms, then for any 
$F$-invariant subset of $G$, the space $P_K(G)$ is $F$-invariant and for any $G\rt F$-algebra $B$, we get
$$K^{top}_*(G\rt F,B)= \lim_K KK^{G\rt F}_*(C_0(P_K(G)),B),$$ where in the inductive limit,
 $K$ runs though compact and $F$-invariant subsets of $G$. 
%
%
%
%
If $\Si$ is a proper discrete metric space  and if $r$ is a non positive negative  number, let us set $$E_r=\{(\si,\si')\in \Si\times\Si\text{ such that }d(\si,\si')\lq r\},$$ and then
 consider the element   $\chi_r=1\ts_{C(\beta_\Si)}\chi_{E_r}$ of $C_c(\gsin)$. Then we have $\chi_r^2=\chi_r$ and hence
$$\supp\chi_r=\{\ga\in \gsin \text{ such that }\chi_r(\ga)=1\}$$ is a compact subset of $\gsin$.
Let us set then 
 $P_r(\gsin)=P_{\supp\chi_r}(\gsin)$.
If $\Si$ is provided with an action of a finite group $F$ by isometries, $\chi_r$ being $F$-invariant, we see that  $P_r(\gsin)$
is for any $r>0$ provided with a action of $F$ by homeomorphisms.

For any $\w$ in  $\bns$ and any subset $Y$ of  some $P_r(\gsin)$, let us set $Y_\w$ for the fiber of $Y$ at $\w$, i.e the set of probability measures
of $Y$ supported in the set of elements of $\gsin$ with range $\w$. If $W$ is a subset of   $\bns$ then define $Y_{/W}=\cup_{\w\in W}Y_\w$. 
Let us define $\displaystyle P_r(\gsinz)=P_r(\gsin)_{/\bnsz}$.
For a fix $r>0$, every  element $(n,\si,x)$ of $\N\times\Si\times P_r(\Si)$ can be viewed as a element in $P_r(\gsinz)$.
For any family $\X=(X_i)_{i\in\N}$ of compact subsets of
$P_r(\Si)$, let us set $Z_{\X}$ for the closure of $$\{(n,\si,x)\in\N\times\Si\times P_r(\Si);\,n\in\N\,, \si\in X_n,\, x\in X_n\}$$ 
in $P_r(\gsin)$(we view an element  $\si$ of $\Si$ 
 as an element
of $P_r(\Si)$, the Dirac measure at $\si$).
\begin{lemma}\label{lemma-compact}
 Let $r$ be a positive number and let  $\X=(X_i)_{i\in\N}$  be family of compact subsets of $P_r(\Si)$. Then 
${Z}_{\X}$ is a compact subset of  $P_r(\gsinz)$.
\end{lemma}
\begin{proof}
 Since $X_i$ is  a compact subset of the locally finite simplicial complex $P_r(\Si)$, there exists a finite
set $Y_i$ of $\Si$ such that every element of $X_i$ is supported in $Y_i$.  Applying  lemma \ref{lemma-partition-Nentourage} to
$$E=\{(n,\si,\si')\in\N\times\Si\times\Si;\,\si\in Y_n,\,\si'\in Y_n\,,d(\si,\si')\lq r\},$$
we see  that
there exist $f_1,\ldots,f_k$ in $\ell^\infty(\N\times\Si)\cong C(\beta_{\N\times\Si})$ and $E_1,\ldots,E_k$ entourages of $\Si$ of diameter less than $r$ 
such  that $\chi_E(i,\si,\si')=\sum_{j=1}^k f_j(i,\si)\chi_{E_j}(\si,\si')$ for all integer $j$ and all $\si$ and $\si'$ in $\Si$.
Set then 
$\tilde{\chi}_E=\sum_{j=1}^k f_j\otimes_{C(\beta_{\Si})}\chi_{E_j}\in C_c(\gsinz)$. Then $\tilde{\chi}_E$ is valued in $\{0,1\}$.
The set  of probabity measures $\eta$ such that $\eta(\tilde{\chi}_E)=1$ is closed in the unit ball of
the dual of  $C_c(\gsin)$ equipped with the weak 
topology and hence is compact. Since $\eta(\tilde{\chi}_E)=1$ for any $\eta$ in $Z_{\X}$, we get that $Z_{\X}$ is compact in $P_r(\gsin)$.
But since we also have $\eta(\tilde{\chi}_Ef)= \eta(f)$ for any $\eta$ in $Z_{\X}$ and any $f$ in  $C_c(\gsin)$ and since $\tilde{\chi}_E$ is 
in $C_c(\gsinz)$, we deduce that $Z_{\X}$ is included in 
 $P_r(\gsinz)$.
\end{proof}
\begin{corollary}
 Let $r$ be a positive number and let  $\X=(X_i)_{i\in\N}$  be family of compact subsets of $P_r(\Si)$. Then the closure of
$\bigcup_{i\in\N}\{n\}\times\Si\times X_i$ in  $P_r(\gsin)$ is a $\gsin$-invariant and $\gsin$-compact subset of $P_r(\gsinz)$.
\end{corollary}
\begin{proof}
The closure of $\displaystyle\bigcup_{i\in\N}\{n\}\times\Si\times X_i$ in  
$P_r(\gsin)$ is the $\gsin$-orbit of 
$Z_{\X}$  and hence is  $\gsin$-invariant and $\gsin$-compact in  $P_r(\gsinz)$.
\end{proof}

\subsection{Topological $K$-theory for the groupoid $\gsif$ with coefficients in $\ac$}
The aim of this subsection
is to show that for any free action of a finite group
$F$ by isometries on $\Si$ and any family $\A=(A_i)_{i\in\N}$ of $F$-$C^*$-algebras,
we have a  natural identification between $K^{top}_*(\gsinf,\aci)$ and 
$K^{top,\infty}_*(F,\Si,\A)$.

\smallskip

Let $\Si$ be a discrete metric space with bounded geometry. 
For any $\gsin$-invariant and $\gsin$-compact subset $Y$ of  $P_r(\gsinz)$, then  $Y_{/\{i\}\times\Si\times P_r(\Si)}$
is an invariant and cocompact subset of $\{i\}\times\Si\times P_r(\Si)$ for the action of the groupoid of pairs $\Si\times\Si$. 
Hence, there exists a family $\X^Y=(X_i^Y)_{i\in\N}$ of compact subset of $P_r(\Si)$ such that  
\begin{equation}\label{equ-fiber}Y_{/\{i\}\times\Si\times P_r(\Si)}=\{i\}\times\Si\times X_i^Y
 \end{equation}
for every integer $i$. Notice that if $\Si$ is provided with an action of a finite 
group $F$ by isometries and if $Y$ as above is moreover 
$F$-invariant, then $X_i^Y$ is $F$-invariant for every integer $i$.
For any $\gsinf$-invariant and
$\gsin$-compact subset $Y$ of $P_r(\gsinz)$ and any family $\A=(A_i)_{i\in\N}$ of $F$-$C^*$-algebras,  consider the following composition (recall that 
$\aci=\prod_{i\in\N}C_0(\Si,A_i\ts \K(\H))$)
\begin{equation}\label{equation-iY}
 \begin{split} i_Y:KK^\gsinf_*(&C_0(Y),\aci)\longrightarrow \\ &\prod_{i\in\N}KK^F_*(C(X_i^Y), \K(\H)\ts A_i)
  \longrightarrow\prod_{i\in\N}KK^F_*(C(X_i^Y), A_i),\end{split}\end{equation}
where
\begin{itemize}
 \item $X_i^Y$ is  for any   integer $i$ defined by equation (\ref{equ-fiber});
\item the first map  is induced by groupoid functoriality with respect to the bunch of groupoid morphisms
$$F \hookrightarrow (\N\times\Si\times\Si)\rt F;\, x\mapsto (i,\si,x(\si),x),$$ where $i$ runs through integers and $\si$ is a fixed 
element of $\Si$ (recall that $\N\times\Si\times\Si$ is a $F$-invariant subgroupoid of $\gsin$.)
\item the second map is given for every integer $i$ by the Morita equivalence between $\K(\H)\ts A_i$ and $A_i$.
\end{itemize}

 Let $\A=(A_i)_{i\in\N}$ and $\B=(B_i)_{i\in\N}$ be families  of $F$-algebras and let $z=(z_i)_{i\in\N}$ be 	a family in  
$\prod_{i\in\N}KK^F_*(A_i,  B_i)$.
We can assume without loss of generality that for every integer $i$, then $z_i$  is
 represented by a $F$-equivariant $K$-cycle
$(\pi_i,T_i,\ell^2(F)\otimes \H\otimes B_i)$ where
\begin{itemize}
\item $\H$ is a separable Hilbert
space;
\item $F$ acts diagonally on $\ell^2(F)\otimes \H\otimes B_i$ by the right regular 
representation on $\ell^2(F)$ and  trivially on $\H$.
\item  $\pi_i$ is a $F$-equivariant
representation of $A_i$ in the algebra $\L_{B_i}(\ell^2(F)\otimes \H\otimes B_i)$ of adjointable operators of
$\ell^2(F)\otimes \H\otimes B_i$;
\item $T_i$ is a $F$-equivariant self-adjoint operator of $\L_{B_i}(\ell^2(F)\otimes\H\otimes B_i)$ 
satisfying  the $K$-cycle conditions, i.e. $[T_i,\pi_i(a)]$ and 
  $\pi_i(a)(T_i^2-\Id_{\ell^2(F)\otimes\H\otimes B_i})$  belongs to
  $\K(\ell^2(F))\otimes \H\otimes B_i$, for every $a$ in $A_i$.
\end{itemize}
Noticing that we have an identification between the algebras $\L_{B_i}(\ell^2(F)\otimes \H\otimes B_i)$
and $\L_{\K(\H)\otimes B_i}(\ell^2(F)\ts \K(\H)\otimes B_i)$. Indeed  these two $C^*$-algebras can be viewed as the mutiplier algebra of
$\K(\ell^2(F)\otimes\H)\otimes B_i$.  We see that the pointwise diagonal multiplication by 
$$\N\to \L_{\K(\H)\otimes B_i}(\ell^2(F)\otimes\K(\H)\ts  B_i);\,i\mapsto T_i$$ gives rise to a $F$-equivariant adjointable
operator $T_{C_0(\Si)}^\infty$ of the right $\bci$-Hilbert module
$\prod_{i\in\N}C_0(\Si,\ell^2(F)\ts\K(\H)\ts B_i )\cong\ell^2(F)\ts \bci$. 
The family of representation $(\pi_i)_{i\in\N}$ gives rise to a representation $\pi^\infty_{C_0(\Si)}$ of $\ac$ on the algebra 
of adjointable operators of
$\prod_{i\in\N}C_0(\Si,\ell^2(F)\ts \K(\H)\ts B_i)$. 
It is then straightforward to check that $\pi^\infty_{C_0(\Si)}$  and   $T_{C_0(\Si)}^\infty$ are indeed $\gsinf$-equivariant and that
$T_{C_0(\Si)}^\infty$  satisfies the $K$-cycle conditions. Therefore, we obtain
in this way a $K$-cycle for $KK_* ^\gsinf(\ac,\bci)$ and we can define in this way a morphism
$$\tau_{C_0(\Si)}^\infty:\prod_{i\in\N}KK^F_*(A_i,  B_i)\to KK_* ^\gsinf(\ac,\bci)$$
which  is moreover bifunctorial, i.e if $\A=(A_i)$  and $\B=(B_i)$ are families  of $F$-algebras, then 
\begin{itemize}
 \item for any family   $\A'=(A'_i)$  of $F$-algebras and any family  $f=(f_i)_{i\in\N}$  of $F$-equivariant homomorphisms
$f_i:A_i\to A'_i$, then $\tau_{C_0(\Si)}^\infty(f^*(z))=f_\Si^*(\tau_{C_0(\Si)}^\infty(z))$ for any $z$ in
$ \prod_{i\in\N}KK^F_*(A'_i,  B_i)$;
 \item for any family  $\B'=(B'_i)$ of $F$-algebras and any family $g=(g_i)_{i\in\N}$  of  $F$-equivariant homomorphisms
$g_i:B_i\to B'_i$, then $\tau_{C_0(\Si)}^\infty(g_*(z))=g^{\infty}_{\Si,*}(\tau_{C_0(\Si)}^\infty(z))$ for any $z$ in
$\prod_{i\in\N}KK^F_*(A_i,  B_i)$.
\end{itemize}

For a family  $\X=(X_i)_{i\in\N}$ of compact subsets in some $P_r(\Si)$, we set $\CC_\X=(C(X_i))_{i\in\N}$. If  $\X'=(X'_i)_{i\in\N}$ is another
such a family such that $X_i\subset X'_i$ for any integer $i$ (we say that $(\X,\X')$ is a relative pair of families), let us set
$\CC_{\X,\X'}=(C_0(X'_i\setminus X_i))_{i\in\N}$. Let $Z$ be a $\gsin$-compact subset of some $P_K(\gsin)$ for $K$ a given
 compact subset of $\gsin$. Let us fix $r>0$ such that  $P_K(\gsin)\subset P_r(\gsin)$ and let  $\X=(X_i)_{i\in\N}$ be a family of compact subsets in $P_r(\Si)$ such that
$Z_\X\subset Z$. Define then the $\gsin$-equivariant homomorphism
$$\Lambda_\X^Z:C_0(Z)\to \CC_{\X,C_0(\Si)};\,f\mapsto (f_i)_{i\in\N},$$ with $f_i$ in $C_0(\Si\times X_i)$ defined  by $f_i(\si,x)=f(i,\si,x)$
for any integer $i$, any $\si$ in $\Si$ and any $x$ in $X_i$.
In the same way, if $(Z,Z')$ is a relative pair of $\gsin$-compact subsets of   $P_K(\gsin)$ and if 
 $(\X,\X')$ is a relative pair of families of compact subsets in  $P_r(\Si)$ such that
$Z_\X\subset Z$ and $Z_{\X'}\subset Z'$, the restriction of $\Lambda_{\X'}^{Z'}$ to  $C_0(Z'\setminus Z)$ gives rise to 
a   $\gsin$-equivariant homomorphism
$$\Lambda_{\X,\X'}^{Z,Z'}:C_0(Z'\setminus Z)\to \CC_{\X,\X',C_0(\Si)}.$$
If $(Z,Z')$ is a relative pair of $\gsin$-compact subsets of  $P_K(\gsin)$ and if 
$\X'$ is a family of compact subsets in  $P_r(\Si)$ such that 
$Z_{\X'}\subset Z'$, then there exists a unique family $\X'_{/Z}=(X'_{i,/Z})_{i\in\N}$ of compact subsets in  $P_r(\Si)$ such that 
$(\X'_{/Z},\X')$ a relative pair of families and $Z_{\X'_{/Z}}= Z_{\X'} \cap Z$.
If the relative paire $(Z,Z')$ is moreover $F$-invariant, then  $(\X'_{/Z},\X')$ is a relative pair of families of $F$-invariant compact spaces and 
the map
\begin{eqnarray*}
 \prod_{i\in\N} KK^F_*(C_0(X'_i\setminus {X'_{i,/Z}}),A_i)&\longrightarrow &
KK^\gsinf_*(C_0(Z'\setminus Z),\aci)\\
z=(z_i)_{i\in\N}&\mapsto& \Lambda_{\X,\X'}^{Z,Z',*}(\tau^\infty_{C_0(\Si)}(z)))
\end{eqnarray*}
is 
compatible with family of inclusions
$(X_i\hookrightarrow X'_i)_{i\in\N}$ of $F$-invariant compact subsets. Hence, taking the inductive limite and setting
$$K^{F,\infty}_*(Z,Z',\A)=\lim_{\X'}\prod_{i\in\N} KK^F_*(C_0(X'_i\setminus {X'_{i,/Z}}),A_i),$$ where
$\X'= (X'_i)_{i\in\N}$ runs through family of  compact $F$-invariant subsets in $P_r(\Si)$ such that
$Z_{\X'}\subset Z'$, we end up with a morphism
\begin{equation}\label{equation-upsilon}\upsilon^{Z,Z'}_{F,\Si,\A,*}: K^{F,\infty}_*(Z,Z',\A)\to KK^\gsinf_*(C_0(Z'\setminus Z),\aci).\end{equation}  We set 
$K^{F,\infty}_*(Z,\A)$ for $K^{F,\infty}_*(\emptyset,Z,\A)$ and $\upsilon^Z_{F,\Si,\A,*}$  for $\upsilon^{\emptyset,Z}_{F,\Si,,\A,*}$.
\begin{lemma}\label{lemma-loc-inj}
 Let $Z$ be a $\gsinf$-invariant closed subset of some $P_K(\gsin)$ for $K$ a compact subset of $\gsin$. Assume that the restriction to $Z$ of
the anchor map for the action of $\gsin$ on $P_K(\gsin)$ is locally injective, i.e there exists a covering of $Z$ by open subsets  for which  the restriction 
of the anchor map is one-to-one. Then for any family $\A=(A_i)_{i\in\N}$ of $F$-algebras, 
$$\upsilon^{Z}_{F,\Si,\A,*}: K^{F,\infty}_*(Z,\A)\to KK^\gsinf_*(C_0(Z),\aci)$$ is an isomorphism.
\end{lemma}
\begin{proof}
According to \cite{tucoarsegroupoid}, since $\aci$ is indeed a $C(\beta_{\N\times\Si}^{0})$-algebra, there is an isomorphism
\begin{equation}\label{iso-open}
\lim_{Z'}:KK^\gsinf(C_0(Z'),\aci)\longrightarrow KK^\gsinf(C_0(Z),\aci)\end{equation}
where
\begin{itemize}
 \item in the inductive limit of the left hand side, $Z'$ runs through $\gsin$-compact and $F$-invariant subsets of $Z_{/\beta_{\N\times\Si}^{0}}$.
\item the map is then induced by the inclusion $Z'\hookrightarrow Z$.
\end{itemize}
Under the identification of equation (\ref{iso-open}), the bunch of maps defined by equation (\ref{equation-iY})
 $$i_{Z'}:KK^\gsinf_*(C_0(Z'),\aci)\longrightarrow \prod_{i\in\N}KK^F_*(C(X_i^{Z'}), A_i),$$ where
$Z'$ runs through $\gsin$-compact and $F$-invariant subsets of $Z_{/\beta_{\N\times\Si}^{0}}$ provides an inverse for
$\upsilon^{Z}_{F,\Si,\A,*}$. 
\end{proof}

Since for any compact subset $K$ of $\gsin$, there exists $r>0$ such that  $P_K(\gsin)\subset P_r(\gsin)$, we get that
\begin{equation}\label{equation-lhs}
 K^{top,\infty}_*(F,\Si,\A)=\lim_KK^{F,\infty}_*(P_K(\gsin),\A),\end{equation}
 where in the right hand side, $K$ runs through compact $F$-invariant subsets of $\gsi$, and 
the inductive limit is taken under the maps  induced by  inclusions  $P_K(\gsin)\hookrightarrow  P_{K'}(\gsin)$  corresponding to 
relative pairs  $(K,K')$ of
$F$-invariant compact subset of $\gsin$.
The maps  $$\upsilon^{P_K(\gsin)}_{F,\Si,\A,*}: K^{F,\infty}_*(P_K(\gsin),\A)\to K^\gsinf_*(C_0(P_K(\gsin)),\aci)$$ are then obviously 
compatible with the inductive limit of
equation (\ref{equation-lhs}) and hence gives rise to a morphism
 \begin{equation*}\upsilon_{F,\Si,\A,*}:  K^{top,\infty}_*(F,\Si,\A) \to K^{top}_*(\gsinf,\aci).\end{equation*}

We end this subsection by proving that $\upsilon_{F,\Si,\A,*}$ is an isomorphism. The idea is to use the simplicial structure  of $P_K(\gsin)$
to carry out a Mayer-Vietoris argument. In order to do that, we need first to reduce to the case of a 
second-countable and \'etale groupoid. 
Recall from \cite[Lemma 4.1]{sty} that there exists 	a second countable {\'e}tale groupoid $G'_\Si$ with compact base space $\bsp$ and an action 
of $G'_\Si$ on $\beta_\Si$ such that $\gsi=\beta_\Si\rt G'_\Si$. The groupoid $G'_\Si$ then acts on  $\beta_{\N\times\Si}$ through the action of 
$\gsi$ and $\beta_{\N\times\Si}\rt \gsi=\beta_{\N\times\Si}\rt G'_\Si$.  For any subset $X$ of a $\gsip$-space, 
let us set $X^\N_\Si=\bns\times_{\bsp} X$. 
If $\Si$ is provided with an action by isometries of  a finite group $F$, then
$G'_\Si$ can be choosen provided with an action of $F$ by automorphisms that make the action on $\beta_{\Si}$ and hence on 
 $\beta_{\N\times\Si}$  equivariant. If we set then $\gsifp=G'_\Si\rt F$,   then 
for any family $\A=(A_i)_{i\in\N}$ of $F$-algebra, 
 $\ac$ is a $\gsifp$-algebra. Let $Y$ be a locally compact space equipped
with a proper and cocompact action of $\gsifp$. Then the map $$KK_*^\gsinf(C_0(\bnsy),\ac)\to KK^\gsifp_*(C_0(Y),\ac)$$ obtained by forgetting
the $C(\bns)$-action is an isomorphism.  Moreover, up to the identification 
$\ac\rtr\gsinf\cong \ac\rtr\gsifp$, then
the Baum-Connes conjecture for $\gsinf$ and for $\gsifp$  are  for the coefficient $\ac$ equivalent \cite{sty}. 
For any compact subset $K$ of $\gsip$, we have a natural identification 
$$P_{K^\N_\Si}(\gsi)\cong  P_{ K}(\gsip)_\Si^\N.$$
Fix for a compact subset $K$ of $\gsip$ a positive number $r$ such that 
$P_{K^\N_\Si}(\gsi)\subset P_r(\gsin)$. Let $Y$ be a $\gsifp$-invariant closed subset of $P_{ K}(\gsip)$ and let
$\X=(X_i)_{i\in\N}$ be a family of $F$-invariant compact subset of $P_r(\gsin)$ such that
$Z_\X\subset \bnsy$ and let us consider the composition
$${\Lambda'}_Y^\X:C_0(Y)\longrightarrow C_0(\bnsy)\stackrel{\Lambda_\bnsy^\X}{\longrightarrow}\CC_{\X,C_0(\Si)},$$ where the first 
map of the composition is induced by the projection
$\bnsy\to Y$. Let us also consider the relative version: let  $(Y,Y')$ be   a relative pair of $\gsifp$-invariant closed subsets
of  $P_{ K}(\gsip)$  and let $(\X,\X')$ be a relative family of compact $F$-invariant subsets of 
$P_r(\gsin)$ such that
$Z_\X\subset \bnsy$ and $Z_{\X}'\subset \bnsyp$, define
$${\Lambda'}_{Y,Y'}^\X:C_0(Y'\setminus Y)\longrightarrow \CC_{\X,\X',C_0(\Si)}$$ as the restiction of
${\Lambda'}_Y^\X$ to $C_0(Y'\setminus Y')$. Let us then proceed as we did to define $\upsilon^{Y,Y'}_{F,\Si,\A,*}$ in equation (\ref{equation-iY}).

If $\tau'^\infty_{C_0(\Si)}(\bullet)$ stands for the restriction of  
$\tau^\infty_{C_0(\Si)}(\bullet)$ to $KK^\gsifp_*(\bullet,\bullet)$, then the map
\begin{eqnarray*}
 \prod_{i\in\N} KK^F_*(C_0(X'_i\setminus {X'_{i,/\bnsy}}),A_i)&\longrightarrow &
KK^\gsifp_*(C_0(Y'\setminus Y),\aci)\\
z=(z_i)_{i\in\N}&\mapsto&{ \Lambda'}_{\X,\X'}^{Y,Y',*}(\tau'^\infty_{C_0(\Si)}(z)))
\end{eqnarray*}
is 
compatible with family of inclusions
$(X_i\hookrightarrow X'_i)_{i\in\N}$ of $F$-invariant compact subset. Hence, taking the inductive limit under family of $F$-invariant compact subset of $Y'$,
we end up as in equation (\ref{equation-iY}) with a morphism
$$\upsilon'^{Y,Y'}_{F,\Si,\A,*}: K^{F,\infty}_*(\bnsy,\bnsyp,\A)\longrightarrow  KK^\gsifp_*(C_0(Y'\setminus Y),\aci).$$  
 which is indeed the composition
\begin{equation*}\begin{split}K^{F,\infty}_*(\bnsy,\bnsyp,\A)\stackrel{\upsilon^{Y,Y'}_{F,\Si,\A,*}}{\longrightarrow}
 KK^\gsinf_*(C_0({Y'_\Si}^\N\setminus Y^\N_\Si),&\aci)\\&\longrightarrow  KK^\gsifp_*(C_0(Y'\setminus Y),\aci),\end{split}\end{equation*}
where the second map is induced 
by the projection
$\bnsyp\to Y'$.                                   We set 
also  $\upsilon'^Y_{F,\Si,\A,*}$  for $\upsilon'^{\emptyset,Y}_{F,\Si,,\A,*}$
\begin{lemma}\label{lemma-upsilon-iso}
 Let $Y$ be a $\gsifp$-simplicial complexe in sense of   \cite[Definition 3.7]{tucoarsegroupoid} lying in some 
$P_K(\gsifp)$ for $K$ a compact subset of $\gsifp$. Then for any family $\A=(A_i)_{i\in\N}$ of $F$-algebras, 
$$\upsilon^{\bnsy}_{F,\Si,\A,*}: K^{F,\infty}_*(\bnsy,\A)\longrightarrow KK^\gsinf_*(C_0(\bnsy),\aci)$$ is an isomorphism.
\end{lemma}
\begin{proof}Notice first that as we have already mentionned, this is equivalent to prove that
 $\upsilon'^{Y}_{F,\Si,\A,*}: K^{F,\infty}_*(\bnsy,\A)\longrightarrow  KK^\gsifp_*(C_0( Y),\aci)$ is an isomorphism.
Let  us prove the result by induction on the dimension of the $\gsifp$-simplicial complexe $Y$.
If $Y$ has dimension $0$,  the anchor map for the action of $\gsinf$ is locally injective and hence, the result is consequence of 
lemma \ref{lemma-loc-inj}.
We can assume without loss of generality that $Y$ is typed and that the action of $\gsifp$ is typed preserving.
Let $Y_0\subset Y_1\subset \ldots Y_n=Y$ be the skeleton of $Y$, and assume that we have proved that
$$\upsilon^{\bnsynm}_{F,\Si,\A,*}: K^{F,\infty}_*(\bnsynm,\A)\longrightarrow K^\gsinf_*(C_0(\bnsynm),\aci)$$ is an isomorphism.
Since $Y$ is second countable, the inclusion $Y_{n-1}\hookrightarrow Y_n$ gives rise to a long exact sequence
\begin{equation*}\begin{split}
  \ldots \longrightarrow KK_i^\gsifp(&C_0(Y_{n-1}),\aci)\longrightarrow KK_i^\gsifp(C_0(Y_{n}),\aci)\longrightarrow \\
 & KK_i^\gsifp(C_0(Y_{n}\setminus Y_{n-1}),\aci)\longrightarrow KK_{i-1}^\gsifp(C_0(Y_{n-1}),\aci)
\longrightarrow\ldots\end{split}\end{equation*}
In the same way, we have a long exact sequence
\begin{equation*}\begin{split}
  \ldots \longrightarrow K_i^{F,\infty}(&\bnsynm,\A)\longrightarrow K_i^{F,\infty}(\bnsyn,\A)\longrightarrow \\
 & K_i^{F,\infty}(\bnsynm,\bnsyn,\A)\longrightarrow K_{i-1}^{F,\infty}(\bnsynm,\A)
\longrightarrow\ldots\end{split}\end{equation*}
  By naturality of the  morphisms ${\Lambda'}^{\bullet,\bullet}_ {\bullet,\bullet}$ and $\tau'^\infty_{C_0(\Si)}(\bullet)$, these 
 two long exact sequences 
are intertwinned by the maps $\upsilon^{\bullet,\bullet}_{F,\Si,\A,*}$. Using a five lemma argument, the proof of the result
amouts to show that 
$$ {\upsilon'}^{\bnsynm,\bnsyn}_{F,\Si,\A,*}:    
K_*^{F,\infty}(\bnsynm,\bnsyn,\A) \longrightarrow KK_*^\gsifp(C_0(Y_{n}\setminus Y_{n-1}),\aci)$$
is an isomorphism. Let $Y'$ be the the set of centers of $n$ simplices of $Y$. Since the action of $\gsifp$ is type preserving,
we have a $\gsifp$-equivariant identification \begin{equation}\label{equ-identification-interior}
  Y_n\setminus Y_{n-1}\cong Y'\times \intdt,   \end{equation} where $\intdt$ is the interior of the standard simplex, and
 where the action of $\gsifp$ on $Y'\times \intdt$ is diagonal through $Y'$.  Let then $[\partial_{Y_{n-1},Y_n}]$ be the element
of $KK^\gsifp_*(C_0(Y'), C_0(Y_{n}\setminus Y_{n-1}))$ that implements up to the identification of equation (\ref{equ-identification-interior}) the Bott 
periodicity isomorphism. We can assume without loss of generality that in the definition of
 $KK_*^{F,\infty}(\bnsynm,\bnsyn,\A)$, the inductive limit is taken over families $\X=(X_i)_{i\in\N}$ of $F$-invariant compact subsets of some  $P_r(\Si)$
such that
\begin{itemize}
  \item $X_i$ is  for every integer $i$ a finite union of $n$-simplices with respect to the simplicial structure inherited from $Y$.
 \item  $Z_\X\subset \bnsyn$.
\end{itemize}
Let $\X$ be such a family and let $X'_i$ be  for every integer $i$ the set of centers of $n$-simplices of $X_i$.
Let us set then $\X'=(X'_i)_{i\in\N}$. Since the action of  $F$ is type preserving, we have a $F$-equivariant identification
\begin{equation}\label{equ-identifiaction-simplex-X}
 X_i\setminus {X_i}_{/\bnsynm}\cong X'_i\times\intdt, \end{equation}the action of $F$ on $\intdt$ being trivial.
Let $[\partial_i]$ be the element of $KK^F_*(C(X'_i), C_0(X_i\setminus {X_i}_{/\bnsynm}))$ that implements up to this identification the Bott
periodicity isomorphism and set then 
$$[\tilde{\partial}]=([\partial_i])_{i\in\N}\in \prod_{i\in\N}KK^F_*(C(X'_i), C_0(X_i\setminus {X_i}_{/\bnsynm})).$$
The Bott generator of $KK_*(\C,C_0(\intdt)))$ can be represented by a $K$-cycle $(C_0(\intdt,\C^l),\phi,T)$ for some integer $l$,
 where $m$ is the obvious
representation of $\C$ on $C_0(\intdt,\C^l)$ by scalar multiplication, and $T$ is an adjointable operator on $C_0(\intdt,\C^l)$ that satisfies the
$K$-cycle conditions. Then for any integer $i$, the element $[\partial_i]$ of $KK^F_*(C(X'_i), C_0(X_i\setminus {X_i}_{/\bnsynm}))$ can be
represented by the $K$-cycle $$(C_0(X'_i\times \intdt,\C^l), \phi_i,Id_{C(X'_i)}\ts T),$$ where $C_0(X'_i\times \intdt,\C^l)$ is viewed 
as a right $C_0(X_i\setminus {X_i}_{/\bnsynm})$-Hilbert module by using the identification of equation (\ref{equ-identifiaction-simplex-X}) 
and $\phi_i$ is the obvious diagonal representation of $C_0(X'_i)$ on $C_0(X'_i\times \intdt,\C^l)$.
Let $[\partial_{\X,C_0(\Si)}]$ be the class of the $K$-cycle 
$(\prod_{i\in\N}C_0(\Si\times X'_i\times\intdt,\C^l), \prod_{i\in\N}   Id_{C_0(\Si)}\ts \phi_i,\prod_{i\in\N} Id_{C_0(\Si\times X'_i)}\ts T)$ in 
$KK^\gsifp_*(\CC_{\X',C_0(\Si)}, \CC_{\X_{/\bnsynm},\X,C_0(\Si)})$. Then we have
\begin{equation}\label{equ-boundaries-commutes}
{\Lambda'}_{Y'}^{\X',*}([\partial_{\X,C_0(\Si)}])={\Lambda'}_{Y_{n-1},Y_n,*}^{\X_{/\bnsynm},\X}[\partial_{Y_{n-1},Y_n}].
\end{equation}
Let $z=(z_i)_{i\in\N}$ be a family in $\prod_{i\in\N}KK^F(C_0(X_i\setminus {X_i}_{/\bnsynm}),A_i)$.
Then using the characterisation of the Kasparov product (see \cite{kas} and \cite{legall} for the groupoid case), we get that
$${\Lambda'}_{Y'}^{\X',*}([\partial_{\X,C_0(\Si)}])\ts \tau^\infty_{C_0(\Si)}(z)={\Lambda'}_{Y'}^{\X',*}( \tau^\infty_{C_0(\Si)}( [\tilde{\partial}]\ts z)).$$
This in turn implies that
\begin{eqnarray*}
 [\partial_{Y_{n-1},Y_n}]\ts {\Lambda'}_{Y_{n-1},Y_n}^{\X_{/\bnsynm},\X,*}( \tau^\infty_{C_0(\Si)}(z))
&=& {\Lambda'}_{Y_{n-1},Y_n,*}^{\X_{/\bnsynm},\X}([\partial_{Y_{n-1},Y_n}])\ts  \tau^\infty_{C_0(\Si)}(z)\\
&=& {\Lambda'}_{Y'}^{\X',*}([\partial_{\X,C_0(\Si)}])\ts  \tau^\infty_{C_0(\Si)}(z)\\
&=&{\Lambda'}_{Y'}^{\X',*}( \tau^\infty_{C_0(\Si)}( [\tilde{\partial}]\ts z)),
\end{eqnarray*}
where the first equality is a consequence of bifunctoriality of Kasparov product and the second equality holds by equation
 (\ref{equ-boundaries-commutes}).
From this, we get the existence of   a commutative diagram
$$\begin{CD}
K_*^{F,\infty}(\bnsynm,\bnsyn,\A)@> \cong >> K_*^{F,\infty}(\bnsyp,\A) \\
         @V   {\upsilon'}^{Y_{n-1},Y_n}_{F,\Si,\A,*} VV
         @VV {\upsilon'}^{Y'}_{F,\Si,\A,*} V\\
 KK^{\gsifp}_*(C_0(Y_n\setminus Y_{n-1}),\aci)@>[\partial_{Y_{n-1},Y_n}]\ts>> KK^{\gsifp}_*(C_0(Y'),\aci) 
\end{CD},$$
where the top row is obtained  by taking inductive limit over morphisms
 $$([\partial_i]\ts: KK^F_*(C_0(X_i\setminus {X_i}_{/\bnsynm}),A_i)\stackrel{\cong}{\longrightarrow} KK^F_*(C(X'_i), A_i))_{i\in\N}$$
relative to families  $\X=(X_i)$ of $F$-invariant  compact subset of some $P_r(\Si)$ such that
\begin{itemize}
  \item $X_i$ is  for every integer $i$ a finite union of $n$-simplices.
 \item  $Z_\X\subset \bnsyn$.
\end{itemize}
Since  $Y'$ is a $\gsifp$-simplicial complex of degree $0$ and as we have already seen,   $\upsilon'^{Y'}_{F,\Si,\A,*}$ is an isomorphism, and hence
$\upsilon'^{Y_{n-1},Y_n}_{F,\Si,\A,*}$ is an isomorphism. From this we deduce that  $\upsilon'^{Y_n}_{F,\Si,\A,*}$ 
is an isomorphism and hence that  $\upsilon^{\bnsyn}_{F,\Si,\A,*}$ is an isomorphism.


\end{proof}
\begin{corollary}\label{cor-upsilon-iso}
 Let $\Si$ be a discrete metric space provided with an  action of a finite group $F$ by isometries and let $\A=(A_i)_{i\in\N}$ be a family 
of $F$-algebras. Then,
$$ \upsilon_{F,\Si,\A,*}: K_*^{{top},\infty}(F,\Si,A)\longrightarrow K^{top}_*(\gsinf,\aci)$$ is an isomorphism.
\end{corollary}

\subsection{The assembly map for the action of $\gsif$ on $\aci$}
The aim of this subsection is to show that up to the identifications provided on the left hand side by corollary \ref{cor-upsilon-iso} and
on the right hand side by equation (\ref{equ-identification-crossed-product}), then the maps 
 $$\mu_{\gsinf,\aci,*}:K^{top}_*(\gsinf,\aci)\to K_*(\aci \rtimes_r \gsinf)$$ and 
 $$\nu_{F,\Si,\A,*}^{\infty}:K_*^{{top},\infty}(F,\Si,\A){\longrightarrow}
K_*(\A^{\infty}_{\Si}\rtimes F)$$  coincide. 

\smallskip

Fix a rank one projection $e$ in $\K(\H)$ and let us define $\jmath:\C\to\K(H);\lambda\mapsto \lambda e$. For any family of $C^*$-algebras
$\A=(A_i)_{i\in\N}$, let us consider the family of homomorphisms $(\jmath_\A=\jmath\ts Id_{A_i}: A_i\to A_i\ts\K(\H))_{i\in\N}$.
\begin{proposition}\label{prop-funct-tau}
For any families of $F$-algebras  $\A=(A_i)_{i\in\N}$ and $\B=(B_i)_{i\in\N}$ and any element $z=(z_i)_{i\in\N}$ in
$\prod_{i\in\N}KK^F_*(A_i,  B_i)$, we have a commutative diagram
$$\begin{CD}
 K_*(\ac\rtr\gsinf)@>\ts J_{\gsinf}(\tau_{C_0(\Si)}^\infty(z))>>K_*(\bci\rtr\gsinf)\\
         @V \jmath_{\A,F,\Si,*}\circ  \II_{F,\Si,\A,*} VV
         @VV  \II_{F,\Si,\B^\infty,*}V\\
 K_*(\A^\infty_\Si\rtr F) @>\tau^\infty_{F,\Si}(z)>> K_*(\B^\infty_\Si\rtr F)
\end{CD}$$
where  up to the identifications $K_*(\A^\infty_\Si\rtr F)\cong K_*(\A^\infty_{F,\Si})$ and 
$K_*(\B^\infty_\Si\rtr F)\cong K_*(\B^\infty_{F,\Si})$,  the morphism $\tau^\infty_{F,\Si}(z)$ is induced in $K$-theory by the controlled morphism
 $\TT^\infty_{F,\Si}(z):\K(\A^\infty_{F,\Si})\to \K(\B^\infty_{F,\Si})$.
 \end{proposition}
\begin{proof} Assume first that the family $z=(z_i)_{i\in\N}$ is of even degree.
 According to \cite[Lemma 1.6.9]{laff-inv}, there exists for any integer $i$
\begin{itemize}
 \item a $F$-algebra $A'_i$;
\item two $F$-equivariant homomorphisms $\alpha_i:A'_i\to B_i$ and $\beta_i:A'_i\to A_i$  such that the induced element $[\beta_i]\in KK^F_*(A'_i,A_i)$ 
is invertible and such that $z_i=\alpha_{i,*}([\beta_i]^{-1})$.
\end{itemize}
By naturality of $\jmath_{\bullet,F,\Si}$ and  $\II_{F,\Si,\bullet,*}$ and by left functoriality  of $\tau_{C_0(\Si)}^\infty$, 
$J_{\gsin}$ and $\tau^\infty_{F,\Si}$, we can actually assume that for any integer $i$, then $z_i=[\beta_i]^{-1}$ for a homomorphism 
$\beta_i:B_i\to A_i$ such that the induced element $[\beta_i]\in KK^F_*(B_i,A_i)$  is $KK$-invertible. Let us consider the family of
 homomorphisms $\beta=(\beta_i)_{i\in\N}$. Using the bifunctoriality of 
$\tau^\infty_{F,\Si}$, we see that $\tau^\infty_{F,\Si}(z)$ is an isomorphism  with inverse $\beta^\infty_{F,\Si,*}$.
But then, if we set $[Id_\A]=([Id_{A_i}])_{i\in\N}$, using once again the naturality of $\jmath_{\bullet,F,\Si}$ and  $\II_{F,\Si,\bullet,*}$ and 
right functoriality  of 
$J_{\gsinf}$ and $\tau^\infty_{F,\Si}$, we have 
\begin{eqnarray*}
 \beta^\infty_{F,\Si,*}\circ  \II_{F,\Si,\B^\infty,*}( J_{\gsinf}(\tau_{C_0(\Si)}^\infty(z)))&=
&\II_{F,\Si,\A^\infty,*}( J_{\gsinf}(\tau_{C_0(\Si)}^\infty(\beta_*(z))))\\
&=&\II_{F,\Si,\A^\infty,*}( J_{\gsinf}(\tau_{C_0(\Si)}^\infty([Id_\A]))).
\end{eqnarray*}But up to the identifications provided by $\II_{F,\Si,\bullet}$, then $ J_{\gsinf}(\tau_{C_0(\Si)}^\infty([Id_\A]))$ coincides with
$ \jmath_{\A,F,\Si,*}$ and hence we get the result in the even case. 

\medskip

If $z=(z_i)_{i\in\N}$ is a family of odd  degree.
Then,  for every integer $i$, the element $z_i$ of $KK_1^F(A_i,B_i)$ can be viewed up to Morita equivalence as implementing the boundary element
of a semi-split extension of $F$-algebras
$$0\lto \K(\H)\ts B_i\lto E_i\lto A_i\lto 0.$$If we set $\E=(E_i)_{i\in\N}$, then the induced extension 
$$0\lto  \B^\infty_{C_0(\Si)}\lto \E_{C_0(\Si)}\lto \A_{C_0(\Si)}\lto 0$$ is a semisplit extension of $\gsinf$-algebra and hence gives rise to an extension of $C^*$-algebras
$$0\lto  \B^\infty_{C_0(\Si)}\rtr\gsinf \lto \E_{C_0(\Si)}\rtr\gsinf\lto \A_{C_0(\Si)}\rtr\gsinf \lto 0$$Moreover, by naturality of 
$\II_{F,\Si,\bullet}$, have a commutative diagram with exact rows 
$$\begin{CD}
 0@>>>  \B^\infty_{C_0(\Si)}\rtr\gsinf@>>>\E_{C_0(\Si)}\rtr\gsinf@>>>\A_{C_0(\Si)}\rtr\gsinf@>>>0\\
 @.        @V\II_{F,\Si,\B^\infty}VV @V\II_{F,\Si,\E}VV          @V\II_{F,\Si,\A}VV      @.   \\
  0@>>>  \B^\infty_{\Si}\rtr F@>>>\E_{\Si}\rtr F@>>>\A_{\Si}\rtr F@>>>0\end{CD}.$$ By using  naturally of the boundary map in $K$-theory, the result in the odd case is a consequence of the two following observations:
  \begin{itemize}
  \item $J_{\gsinf}(\tau_{C_0(\Si)}^\infty(z))$ implements the boundary map of the top extension;
  \item if $\partial_{\B^\infty_{\Si}\rtr F,\E_{\Si}\rtr F}$ stands for the boundary map of the bottom extension, then 
  $$\partial_{\B^\infty_{\Si}\rtr F,\E_{\Si}\rtr F}=  \tau^\infty_{F,\Si}(z)\circ \jmath_{\A,F,\Si,*}.$$    \end{itemize}
  
    \end{proof}

\begin{proposition}\label{proposition-intertwinned-assembly-map}
 Let $\Si$ be a discrete metric space provided with a free action of a finite group $F$ by isometries and let $\A=(A_i)_{i\in\N}$ be a family 
of $F$-algebras. Then we have a commutative diagram
$$\begin{CD}
 K_*^{{top},\infty}(F,\Si,A)@> \upsilon_{F,\Si,\A,*}>>K^{top}_*(\gsinf,\aci) \\
         @V   \nu^\infty_{F,\Si,\A,*} VV
         @VV\mu_{\gsinf,\A_{C_0(\Si)},*}V \\
K_*(\A^{\infty}_{\Si}\rtimes F) @>\II^{-1}_{F,\Si,\A^\infty,*}>> K_*(\aci\rtimes_r \gsif)
\end{CD}$$
\end{proposition}
\begin{proof}
 Let $Z=P_K(\gsin)$, for $K$ a $F$-invariant subset in $\gsin$ an let us fix $r>0$ such that $Z\subset P_r(\gsin)$. 
Let us define $\phi_{Z}:Z\to\C;\eta\mapsto\eta(\chi_0)$, where $\chi_0$ is the characteristic function of the diagonal of $\Si\times\Si$. Then $\phi_{Z}$ is a cut-off function for the proper action of $\gsin$ on $Z$.
Let $$P_{Z,\gsin}:Z\times_{\beta_{\N\times\Si}}\gsin\to \C;(\eta,\ga)\mapsto \phi_Z(\eta)^{1/2} \phi_Z(\eta\cdot \ga)^{1/2}$$ be
 the Mishchenko projection of
$C_0(Z)\rtr\gsi$ associated to $\phi_{Z}$. For a family  $\X=(X_i)_{i\in\N}$  of $F$-invariant compact subsets of $P_r(\gsin)$ such that
$Z_{\X}\subset Z$, let us consider $P_\X=(P_{X_i})_{i\in\N}$ in $\CC_{\X,F,\Si}$, where $P_{X_i}$ is for each integer $i$ the projection  defined by 
equation (\ref{equ-def-PX}).
Recall that then, $P_\X^\infty=(P_{X_i}\ts e)_{i\in\N}$ in $\CC^\infty_{\X,\Si}$ for $e$ a fixed rank one projection in $\K(H)$.
Noticing that $[P_\X^\infty]=\jmath^\infty_{\CC_{\X,F,\Si,*}}[P_\X]$ in 
$K_0(\CC^\infty_{\X,F,\Si}))\cong K_0(\CC^\infty_{\X,\Si}\rt F)$, then the commutativity of the diagram amounts to show that
$$\II_{F,\Si,\A^\infty,*}([P_{Z,\gsin}]\ts J_{\gsin}(\Lambda_Z^{\X,*}(\tau^\infty_{C_0(\Si)})(z)))=\tau_{F,\Si}^\infty(z)([\jmath_{\CC_{\X,F,\Si}}(P_\X)])$$ up to the identification
$K_*(A_\Si^\infty\rt F)\cong K_*(A_{F,\Si}^\infty)$.
But it is straitforward to check that
$$\Lambda_\X^Z(\phi_Z)=(\phi_{\Si,i})_{i\in\N}$$ with
$\phi_{\Si,i}:\Si\times X_i\to \C:(\si,x)\mapsto \lambda_{\si}(x)$.
Hence, if $$\Lambda_{\X,\gsinf}^Z:C_0(Z)\rt \gsinf\to \CC_{\X,C_0(\Si)}\rtr\gsinf$$ stands for the map induced by $\Lambda_\X^Z$ on the reduced 
crossed-products, we have
\begin{equation}\label{equ-misch-proj}\II_{F,\Si,\CC_\X}\circ \Lambda_{\X,\gsinf}^Z(P_{Z,\gsin})=P_\X.\end{equation}From this, we deduce
\begin{eqnarray*}
  \II_{F,\Si,\A^\infty,*}([P_{Z,\gsin}]\ts J_{\gsinf}(\Lambda_{Z,*}^{\X}(\tau^\infty_{C_0(\Si)})(z)))&=&
 \II_{F,\Si,\A^\infty,*}([\Lambda_{Z,\gsinf}^{\X,}(P_{Z,\gsinf})]\ts J_{\gsinf}(\tau^\infty_{C_0(\Si)})(z)))\\
&=&
 [\Lambda_{Z,\gsinf}^{\X}(P_{Z,\gsin})]\ts \II_{F,\Si,\A^\infty,*}( J_{\gsinf}(\tau^\infty_{C_0(\Si)})(z))\\
&=&
 \tau_{F,\Si}^\infty(z)\circ \jmath^\infty_{\CC_\X,F,\Si,*}\circ \II_{F,\Si,\CC_\X,*}\circ \Lambda_{Z,\gsinf,*}^{\X}([P_{Z,\gsin}])\\
&=&
 \tau_{F,\Si}^\infty(z)\circ  \jmath^\infty_{\CC_\X,F,\Si,*}([P_{\X}]),
\end{eqnarray*}
where there first equality holds by  naturality of $J_\gsinf$ and  left functoriality of Kasparov product, the second equality holds 
by right functoriality of Kasparov product, the third equality is a consequence of proposition \ref{prop-funct-tau} and the fourth equality 
holds by equation (\ref{equ-misch-proj}).
\end{proof}

As a consequence of corollary  \ref{cor-upsilon-iso} and  proposition \ref{proposition-intertwinned-assembly-map}, we obtain 
\begin{theorem}\label{thm-equ-assembly}
Let $F$ be a finite group acting freely on a discrete metric space $\Si$ with bounded geometry and let $\A=(A_i)_{i\in\N}$ be a family of 
$C^*$-algebras. 
Then the two following assertions are equivalent:
\begin{enumerate}
\item   $\nu_{F,\Si,A,*}^{\infty}:K^{top,\infty}_*(F,\Si,\A) \longrightarrow K_*(\A^{\infty}_{\Si}\rtimes F)$ is an isomorphism.
\item the groupoid $\gsif$ satisfies the Baum-Connes conjecture with coefficients in $\aci$.
\item the groupoid $\gsinf$ satisfies the Baum-Connes conjecture with coefficients in $\aci$.
\end{enumerate}
\end{theorem}
\
\subsection{Quantitative statements}\label{subsection-quantitative-statements}
We are now in position to state the analoque of the quantitative statements of
 \cite[Section 6.2]{oy2} in the setting of discrete metric spaces  with  bounded geometry.

Let  $F$ be a finite group, let $\Si$ a be discrete metric space with bounded geometry
 provided with an action of $F$ by isometries and let $A$ be  a $F$-algebra. Let us consider
for 
$d,d',r,r',\eps$ and $\eps'$ positive numbers with $d\lq d'$, $\eps'\lq\eps< 1/4$,  
$r_{d,\eps}\lq
r$ and  $ r'\lq
r$ the following statements:
\begin{description}
\item[$QI_{F,\Si,A,*}(d,d',r,\eps)$]  for any element $x$ in
  $K_*^F(P_d(\Si),A)$, then 
  $\nu_{F,\Si,A,*}^{\eps,r,d}(x)=0$
  in $K_*^{\eps,r}(A_{F,\Si})$ implies that
  $q_{d,d'}^*(x)=0$ in    $K_*^F(P_{d'}(\Si),A)$.
\item[$QS_{F,\Si,A,*}(d,r,r',\eps,\eps')$] for every $y$
  in  $K_*^{\eps',r'}(A_{F,\Si})$, there exists an element $x$ in $K_*^F(P_d(\Si),A)$
   such that
$$\nu_{F,\Si,A,*}^{\eps,r,d}(x)=\iota_*^{\eps',\eps,r',r}(y).$$
\end{description}

The following results provide  numerous
examples  that satisfy these quantitative statements.

\begin{theorem}\label{thm-quantBC-injectivity} Let  $F$ be a finite group, let $\Si$ a
 be discrete metric space with bounded geometry
 provided with a free action of $F$ by isometries and let $A$ be  a $F$-algebra
Then the following assertions are equivalent:
\begin{enumerate}
\item  For any positive numbers $d$,
  $\eps$ and $r$ with  $\eps<1/4$ and  $r\gq r_{d,\eps}$, there
exists   a positive number $d'$ with $d'\gq d$ for which
$QI_{F,\Si,A,*}(d,d',r,\eps)$ is satisfied.
\item   $\nu_{F,\Si,A,*}^{\infty}:K^{top,\infty}_*(F,\Si,A^\N) \longrightarrow K_*(A^{\N,\infty}_{\Si}\rtimes F)$ is 
one-to-one.
\item   $\mu_{\gsif,A^{\N,\infty}_{C_0(\Si)},*}^{\infty}:K^{top}_*(\gsif,A^{\N,\infty}_{C_0(\Si)}) 
\longrightarrow K_*(A^{\N,\infty}_{C_0(\Si)}\rtr\gsif)$ is 
one-to-one.
\end{enumerate}
\end{theorem}
\begin{proof}
The equivalence between points (ii) and (iii) is a consequence of theorem \ref{thm-equ-assembly}. Let us prove that points 
(i) and (ii) are equivalent.
 Assume that condition (i) holds.
Let $x=(x_i)_{i\in\N}$ be a family of elements in some  $K_*^F(P_{d'}(\Si),A)$ such that
$\nu_{F,\Si,A,*}^{\infty,d}(x)=0$.  By definition of 
 $\nu_{F,\Si,A,*}^{\infty,d}(x)$, we have that
$\iota^{\eps',r'}_*(\nu_{F,\Si,A,*}^{\infty,\eps',r',d}(x))=0$ for any $\eps'$ in $(0,1/4)$ and 
$r'\gq r_{d,\eps'}$ and hence, by proposition \ref{proposition-approximation},
 we can find $\eps$ in $(0,1/4)$  and
 $r\gq r_{d,\eps}$ such that $\mu_{F,\Si,A,*}^{\infty,\eps,r,d}(x)=0$. But up to the controlled morphisms
of proposition \ref{prop-product-tensor-infty}    and of lemma \ref{lem-morita-inf}, 
  $\mu_{F,\Si,A,*}^{\infty,\eps,r,d}(x)$ coincides with $\prod_{i\in\N}\mu_{F,\Si,A,*}^{\eps,r,d}(x_i)$, 
so up to rescale 
$\eps$ and $r$ by a (universal) control pair, we can assume that   $\mu_{F,\Si,A,*}^{\eps,r,d}(x_i)=0$
for all integer $i$. Let $d'\gq d$ be a number such that  $QI_{F,\Si,A,*}(d,d',r,\eps)$ is satisfied. Then we get that
$q_{d,d',*}(x_i)=0$ for all integer $i$ such that $d_i\gq d$ and hence $q_{d,d',*}(x)=0$.
  
\medskip

 Let us prove the converse.
Assume first that there exists  positive numbers $d,\,\eps$ and $r$
  with
  $\eps<1/4$  and  $r\gq
  r_{d,\eps}$ and such that for all $d'\gq d$, the condition
$QI_{\Si,F,A}(d,d',r,\eps)$ does not hold. Let us prove that
$\nu_{F,\Si,A,*}^{\infty,d}$ is not one-to-one.
Let $(d_i)_{i\in\N}$ be an increasing and  unbounded sequence of
positive numbers such that $d_i\gq d$ for all integer $i$.  For all
integer $i$, let $x_i$ be an element in  $K_*^F(P_{d}(\Si),A)$ such  that $\nu_{F,\Si,A,*}^{\eps,r,d}(x_i)=0$
in $K^{\eps,r}_*(A_{F,\Si})$  and
$q_{d,d_i,*}(x_i)\neq 0$ in $K_*^F(P_{d_i}(\Si),A)$ and set $x=(x_i)_{i\in\N}$. Then we have 
$\nu_{F,\Si,A,*}^{\infty,d}(x)=0$ and $q_{d,d_i,*}(x)\neq 0$ for all $i$. Since the sequence $(d_i)_{i\in\N}$ is unbounded,
we deduce that the kernel of $\nu_{F,\Si,A,*}^{\infty}$ is non trivial.
\end{proof}
\begin{theorem}\label{thm-quantBC-surjectivity} There exists $\lambda>1$ such that for any   finite group $F$, any  
 discrete metric space $\Si$  with bounded geometry,
 provided with a free  action of $F$ by isometries and any  $F$-algebra $A$, then the following
 assertions are equivalent:
\begin{enumerate}
\item For any  positive numbers $\eps$ and $r'$  with $\eps<\frac{1}{4\lambda}$, there exist
positive numbers $d$ and  $r$ with  $
r_{d,\eps}\lq r$ and $r'\lq r$ for which  $QS_{F,\Si,A,*}(d,r,r',\lambda\eps,\eps)$
is satisfied.
\item    $\nu_{F,\Si,A,*}^{\infty}:K^{top,\infty}_*(F,\Si,A^\N) \longrightarrow K_*(A^{\N,\infty}_{\Si}\rtimes F)$ is 
onto.
\item  $\mu_{\gsif,A^{\N,\infty}_{C_0(\Si)},*}:K^{top}_*(\gsif,A^{\N,\infty}_{C_0(\Si)}) 
\longrightarrow K_*(A^{\N,\infty}_{C_0(\Si)}\rtr\gsif)$  is 
onto.
\end{enumerate}
\end{theorem}

\begin{proof} The equivalence between points (ii) and (iii) is a consequence of theorem \ref{thm-equ-assembly}.
Choose $\lambda$ as in  proposition \ref{proposition-approximation} and 
 assume that condition (i) holds. Let $z$ be an element in $K_*(A^{\N,\infty}_{\Si}\rtimes F)$
 and let $y$ be an element in
 $K^{\eps,r'}_*(A^{\N,\infty}_{F,\Si}\rtimes F)$ such that $\iota^{\eps,r'}_*(y)$ corresponds to $z$ up to the identification
$K_*(A^{\N,\infty}_{\Si}\rtimes F)\cong K_*(A^{\N,\infty}_{F,\Si})$.
Let $y_i$ be the image of $y$  under the composition
\begin{equation}\label{equ-composition-surj}
 K^{\eps,r'}_*(A^{\N,\infty}_{F,\Si})\to
K^{\eps,r'}_*(\K(\H)\ts A_{F,\Si})\stackrel{\cong}{\to}
K^{\eps,r'}_*(A_{F,\Si}), \end{equation} where the first map is induced by the
evaluation $A^{\N,\infty}_{F,\Si}\longrightarrow
A_{F,\Si}\ts \K(\H)$ at the $i$ th coordinate  and the second map is the Morita equivalence. Let $d$ and $r$ be numbers with
$r\gq r'$ and  $r\gq r_{d,\eps}$ and such that $QS_{F,\Si,A}(d,r,r',\lambda\eps,\eps)$ holds. Then  for any integer $i$, 
there exists
a $x_i$ in
$K^F_*(P_{d}(\Si),A)$ such that $\nu_{F,\Si,A,*}^{\lambda\eps,r,d}(x_i)=\iota_*^{\eps,\lambda\eps,r',r}(y_i)$ 
in $K_*^{\lambda\eps,r}(A_{F,\Si})$.
Consider  then  $x=(x_i)_{i\in\N}$ in $K^{top,\infty}_*(F,\Si,A^\N)$.
By construction of the map  $\nu_{F,\Si,A,*}^{\infty}$, we clearly have  $\nu_{F,\Si,A,*}^{\infty}(x)=z$.
\medskip

Conversely,
 assume that there exist positive numbers  $\eps$ and $r'$ with  $\eps<\frac{1}{4\lambda}$ such that for 
all 
  positive numbers $r$ and $d$ with $r\gq r'$ and  $r\gq r_{d,\eps}$,
  then $QS_{F,\Si,,A,*}(d,r,r',\lambda\eps,\eps)$ does not hold. Let us prove
  then that $\nu_{F,\Si,A,*}^{\infty}$ is not onto. Assume first for sake of simplicity that $A$ is unital. Let $(d_i)_{i\in\N}$ and $(r_i)_{i\in\N}$
  be increasing and unbounded sequences of positive numbers such that
  $r_i\gq r_{d_i,\lambda\eps}$ and $r_i\gq r'$. Let  $y_i$ be an element in
  $K^{\eps,r'}_*(A_{F,\Si})$ such that $\iota_*^{\eps,\lambda\eps,r',r_i}(y_i)$
  is not in the range of $\nu_{F,\Si,A,*}^{\lambda\eps,r_i,d_i}$. There
  exists an element $y$ in
  $K^{\eps,r'}_*(A^{\N,\infty}_{F,\Si})$ such
  that for every integer $i$, the image of $y$ under the composition of equation (\ref{equ-composition-surj})
is $y_i$.
Assume that for some $d'$, there is an $x$ in  $K^{top,\infty}_*(F,\Si,A^\N)$ such that up to the identification 
$K_*(A^{\N,\infty}_{\Si}\rtimes F)\cong K_*(A^{\N,\infty}_{F,\Si})$, then
$\iota_*^{\eps,r'}(y)=\mu_{F,\Si,A,*}^{\infty,d'}(x)$. Using   proposition \ref{proposition-approximation},
  we see that  there exists a positive number $r$ with  $r'\lq r$ and
 $r_{d',\lambda\eps}\lq r$ and
  such that $$\nu_{F,\Si,A,*}^{\infty,\lambda\eps,r,d'}(x)=\iota_*^{\eps,\lambda\eps,r',r}(y).$$ But then, if
  we choose $i$ such that $r_i\gq r$ and $d_i\gq d'$, we get by using
  the definition of the geometric assembly map  $\nu_{F,\Si,\bullet,*}^{\infty,\cdot,\cdot,\cdot}$ and by equation (\ref{equ-composition-surj}) that $\iota_*^{\eps,\lambda\eps,r',r_i}(y_i)$ belongs to the
  image of $\nu_{F,\Si,A,*}^{\lambda\eps,r_i,d_i}$, which contradicts our assumption. If $A$ is not unital, then we use the control pair of
lemma \ref{lem-prod-filtered} to rescal $\lambda$.

\end{proof}

Replacing in the proof of (ii) implies (i)  of theorems  \ref{thm-quantBC-injectivity}  and   \ref{thm-quantBC-surjectivity} the
constant  family $A^\N$ 
by a family $\A=(A_i)_{i\in\N}$ of $F$-algebras, 
we can prove  indeed the
following result.
\begin{theorem}\label{thm-quant-surj}Let $\Si$ be a discrete metric space with bounded geometry provided
 with a free action of a finite group $F$ by isometries.
\begin{enumerate}
\item Assume that   for any  family $\A=(A_i)_{i\in\N}$ of $F$-algebras,
  then the  assembly map   $$\mu_{\gsif,\aci,*}:K^{top}_*(\gsif,\A^{\infty}_{C_0(\Si)}) 
\longrightarrow K_*(\A^{\infty}_{C_0(\Si)}\rtr\gsif)$$ is one-to-one.
    Then for any positive numbers $d,\,\eps,r$ with $\eps<1/4$
  and $r\gq r_{d,\eps}$, there
exists   a positive number $d'$ with $d'\gq d$ such that
$QI_{\Si,F,,A}(d,d',r,\eps)$ is satisfied  for every $F$-algebra
  $A$;
\item Assume that   for any   family $\A=(A_i)_{i\in\N}$ of $F$-algebras, the assembly map 
 $$\mu_{\gsif,\aci,*}:K^{top}_*(\gsif,\A^\infty_{C_0(\Si)}) 
\longrightarrow K_*(\A^{\infty}_{C_0(\Si)}\rtr\gsif)$$ is onto.  Then for some  $\lambda>1$ and for any  positive numbers $\eps$
and $r'$ with  $\eps<\frac{1}{4\lambda}$, there exist
positive numbers $d$ and  $r$ with $
r_{d,\eps}\lq r$ and $r'\lq r$ such that  $QS_{\Si,F,A}(d,r,r',\lambda\eps,\eps)$
is satisfied for every  $F$-algebra
  $A$.
\end{enumerate}
\end{theorem}
Recall from \cite{sty,y1} that if $\Si$
 coarsely embeds
 in a Hilbert space, then the groupoid $\gsif$ satisfies the Baum-Connes conjecture for any coefficients.

\subsection{Application to the persistence approximation property}
Let $F$ be a finite group, let $\Si$  be a discrete metric space with bounded geometry 
 provided with a free  action of $F$ by isometries and let $A$ be a $F$-algebra. We apply  the results of the previous section to the persistence approximation for $A_{F,\Si}$:  for any $\eps$  small enough
and any $r>0$ there exists  $\eps'$ in $(\eps,1/4)$ and $r'\gq r$  such that $\sta(A_{F,\Si},\eps,\eps',r,r')$ is satisfied.

Notice that the approximation property is coarse invariant.
To apply quantitative statements of last subsection to our persistence approximation property, we have to define the analogue in the setting of discrete proper metric space
of the existence of a cocompact universal example for proper action of a discrete group.
\begin{definition}
 A discrete metric space $\Si$ provided with a free  action of a finite group is coarsely uniformly 
$F$-contractible if for every $d>0$ there exists $d'>d$ such  that any invariant compact substet of $P_d(\Si)$ lies in a $F$-equivariantly contractible
invariant compact subset of  of  $P_{d'}(\Si)$.
\end{definition}
\begin{example}
 Any (discrete) Gromov hyperbolic metric space provided with a free  action of a finite group $F$ by isometries is coarsely uniformly 
$F$-contractible \cite{ms}.
\end{example}
\begin{lemma}\label{lemma-inj-geom-quant} $\Si$ be a proper discrete metric space provided with a free  action of finite group $F$ by isometries. Assume that 
$\Si$ is coarsely uniformly 
$F$-contractible. Then for any positive numbers $\eps,\,d$ and $r$ with $\eps<1/4$ and $r\gq r_{d,\eps}$, there exists a positive number $d'$ with $d'\gq d$ such that $QI_{F,\Si,A,*}(d,d',r,\eps)$ is satisfied for any $F$-algebra $A$.
\end{lemma}
\begin{proof}
Let $A$ be a $F$-algebra and let $x$ be an element of $K_*(\P_d(\Si),A)$ such that, $\nu_{F,\Si,A,*}^{\eps,r,d}(x)=0$ in $K_*^{\eps,r}(A_{F,\Si})$. Let $d'\gq d$ be a positive number such that every  invariant compact subset of $P_d(\Si)$ lies in a $F$-equivariantly contractible
invariant compact subset of    $P_{d'}(\Si)$. Then $q^{d,d'}_*(x)\in K_*(\P_{d'}(\Si),A)$ comes indeed from an element
of $KK^F_*(C(\{p\}),A)\cong KK^F_*(\C,A)$ for $p$ a $F$-invariant element in  $P_{d'}(\Si)$. But under the identification between $K_*(A_{F,\Si})$ and $K_*(A\rtimes F)$ given by Morita equivalence (see section \ref{sec-indexmap}),
the map $$KK_*^F(\C,A)\longrightarrow K_*(A_{F,\Si});\, 
x\mapsto [P_{\{p\}}]\otimes_{C(\{p\})_{F,\Si}}\tau_{F,\Si}(x)$$ is the Green-Julg duality isomorphism for finite groups \cite{julg}. Since
\begin{eqnarray*}
\nu_{F,\Si,A,*}^{d'}\circ q_*^{d,d'}(x)&=&\nu_{F,\Si,A,*}^{d}(x)\\
&=& \iota_*^{\eps,r}\circ \nu_{F,\Si,A,*}^{\eps,r,d}(x)\\
&=&0.
\end{eqnarray*}
We deduce that $q_*^{d,d'}(x)=0$
\end{proof}

\begin{theorem}
  There exists $\lambda>1$ such that for any   finite group $F$  and any $F$-algebra $A$ the following holds:

 \medskip

Let $\Si$  be a discrete metric space with bounded geometry,
 provided with a free  action of $F$ by isometries. Assume that
\begin{itemize}
\item $\mu_{\gsif,A^{\N,\infty}_{C_0(\Si)},*}^{\infty}:K^{top}_*(\gsif,A^{\N,\infty}_{C_0(\Si)}) 
\longrightarrow K_*(A^{\N,\infty}_{C_0(\Si)}\rtr\gsif)$  is 
onto.
\item $\Si$  is uniformly 
$F$-contractible.
\end{itemize}
Then for any $\eps$ in $(0,\frac{1}{4\lambda})$ and any $r>0$, there exists $r'>0$ such that $\sta_{F,\Si,A,*}(\eps,\lambda\eps,r,r')$ holds.
\end{theorem}
\begin{proof}
In view of  corollary  \ref{cor-upsilon-iso} and  proposition \ref{proposition-intertwinned-assembly-map}, we get that  under he assumptions of the theorem, 
$$\nu_{F,\Si,A,*}^{\infty}:K^{top,\infty}_*(F,\Si,\A) \longrightarrow K_*(\A^{\infty}_{\Si}\rtimes F)$$ is onto for any $F$-algebra $A$. Consider $\lambda$ as in theorem 
\ref{thm-quantBC-surjectivity}. Let $\eps$ be a positive number with $\eps<\frac{1}{4\lambda}$ and let $d$ and $r'$ be positive number with $r'\gq r_{d,\eps}$ such that   $QS_{\Si,F,A}(d,r,r',\lambda\eps,\eps)$
is satisfied for every  $F$-algebra
  $A$.  Choose $d'$ as in lemma \ref{lemma-inj-geom-quant} such that   $QI_{F,\Si,A,*}(d,d',r',\lambda\eps)$ is satisfies. We can assume without loss of generality that $r'\gq r_{d',\lambda\eps}$.  Let $y$ be an element of $K_{\eps,r}^*(A_{F,\Si})$ such that $\iota_*^{\eps,r}(y)=0$ in $K_*(A_{F,\Si})$. Then there exists $x$ in 
$K_*(\P_d(\Si),A)$ such that $\nu_{F,\Si,A,*}^{\lambda\eps,r',d}(x)=\iota_*^{\eps,\lambda\eps,r,r'}(y)$.
Then we have
\begin{eqnarray*}
\iota_*{\eps,\lambda\eps,r,r'}(x)&=&\nu_{F,\Si,A,*}^d(x)\\
&=& \nu_{F,\Si,A,*}^{d'}\circ q_*^{d,d'}(x)\\
&=&0
\end{eqnarray*}
\end{proof}

Similarly, using theorem \ref{thm-quant-surj}, we get:

\begin{theorem}\label{thm-persistence-family}
  There exists $\lambda>1$ such that for any   finite group $F$  the following holds:

 \medskip

Let $\Si$  be a discrete metric space with bounded geometry,
 provided with a free  action of $F$ by isometries. Assume that
\begin{itemize}
\item for any family $\A=(A_i)_{i\in\N}$ of $F$-algebras,
  then the  assembly map   $$\mu_{\gsif,\aci,*}:K^{top}_*(\gsif,\A^{\infty}_{C_0(\Si)}) 
\longrightarrow K_*(\A^{\infty}_{C_0(\Si)}\rtr\gsif)$$  is 
onto.
\item $\Si$  is coarsely  uniformly 
$F$-contractible.
\end{itemize}
Then  for  any $\eps$ in $(0,\frac{1}{4\lambda})$ and any $r>0$, there exists 
$r'>0$ such that $\sta_{*}(A_{F,\Si},\eps,\lambda\eps,r,r')$ holds for any $F$-algebra $A$.
\end{theorem}

\begin{corollary}
  There exists $\lambda>1$ such that for any   finite group $F$ and  any   discrete Gromov hyperbolic metric space $\Si$ 
provided with a free  action of  $F$ by isometries, then the following holds:
  for  any $\eps$ in $(0,\frac{1}{4\lambda})$ and any $r>0$, there exists $r'>0$ such that $\sta_{*}(A_{F,\Si},\eps,\lambda\eps,r,r')$ holds for 
any $F$-algebra $A$.
\end{corollary}

\section{Applications to Novikov conjecture}\label{section-Novikov}
In this section, we  investigate the connection between the quantitative statements of  \ref{subsection-quantitative-statements} and the Novikov conjecture. 
Indeed, we show that when this statements are satisfied  uniformly for the family of finite subsets of a discrete metric space $\Si$ with bounded geometry, then
$\Si$ satisfies the coarse Baum-Connes conjecture.
\subsection{The coarse Baum-Connes conjecture}
Let us first  briefly recall the statement of the coarse Baum-Connes conjecture. Let $\Si$ be a discrete metric space with bounded geometry and let $H$ be  a separable
Hilbert space. Set  $C[\Sigma]_r$ for  the  space of locally compact operators  on
  $\ell^2(\Sigma)\ts H$ with
 propagation  less than $r$,  i.e  operators  that can be written as blocks  {$T=(T_{x,y})_{(x,y)\in \Sigma^2}$}  of compact operators of $H$ such that   $T_{x,y}=0$  if $d(x,y)>r$.
 The { Roe algebra} of $\Si$ is  then 
$$C^*(\Si)=\overline{\cup_{r>0} C[\Sigma]_r}\subseteq \mathcal{L}(\ell^2(\Sigma)\otimes H)$$  and is by definition
filtered by $(C[\Sigma]_r)_{r>0}$.
In \cite{hr} was defined a bunch of coarse assembly maps (we shall recall later on the definition of these maps)
$$\mu_{\Si,*}^s:K_*(P_s(\Si))\to K_*(C^*(\Si))$$ compatible with the maps $K_*(P_s(\Si))\to K_*(P_{s'}(\Si))$ induced by the inclusions of Rips complexes $P_s(\Si)\hookrightarrow P_{s'}(\Si)$ for $s\lq s'$. Taking the inductive limit, we end up with the so-called coarse Baum-Connes assembly map
$$\mu_{\Si,*}:\lim_{s>0}K_*(P_s(\Si))\to K_*(C^*(\Si)).$$ We say that $\Si$ satisfies the coarse Baum-Connes assembly map if 
$\mu_{\Si,*}$ is an isomorphism.

\medskip
The coarse Baum-Connes conjecture is then related to the quantitative statements of \ref{subsection-quantitative-statements} in the following way.
From now on, if $\Si$ is a discrete metric space, then $QI_{\Si,*}(d,d',r,\eps)$ and $QS_{\Si,*}(d,r,r',\eps,\eps')$ respectively stand for 
$QI_{\{e\},\Si,\C,*}(d,d',r,\eps)$ and $QS_{\{e\},\Si,\C,*}(d,r,r',\eps,\eps')$.
\begin{theorem}\label{thm-CBC}
Let $\Si$ be  a discrete metric space with bounded geometry. Assume that the following assertions hold:
\begin{enumerate}
\item For any positive number $d,\,\eps$ and $r$ with $\eps<1/4$ and $r\geq r_{d,\eps}$, there exists a positive number $d'$ with $d'\geq d$ such that 
$QI_{F,*}(d,d',r,\eps)$  holds for any finite subset $F$ of $\Si$;
\item For any positive number $\eps'$ and $r'$ with  $\eps<1/4$, there exists positive numbers $d,\,\eps$  and  $r$ with 
 $r\geq r_{d,\eps}\,,r\geq r'$ and $\eps$ in $[\eps',1/4)$ such that $QS_{F,*}(d,r,r',\eps,\eps')$ holds for any finite subset $F$ of $\Si$.
\end{enumerate}Then $\Si$ satisfies the coarse Baum-Connes conjecture.
\end{theorem} 
Let us recall now the definition of the Coarse Baum-Connes assembly maps given in \cite[Section 2.3]{sty}. Indeed, the definition of the coarse Baum-Connes assembly map was extended to Roe algebras with coefficients in a $C^*$-algebra.  
Let $\H$ be a separable Hilbert space,  let $\Si$ be a proper discrete metric space with bounded geometry and let $B$ be a $C^*$-algebra. Define $C^*(\Si,B)$ the Roe algebra of $\Si$ with coefficient in $B$ as the closure of operators of the Hilbertian right $B$-module $\ell^2(\Si)\ts \H\ts B$ which are locally compact with finite  propagation. Then $C^*(\Si,B)$ is sub-$C^*$-algebra of $\L_B(\ell^2(\Si)\ts\H\ts B)$. This construction is moreover functorial. Any morphism $f:A\to B$ induces in the obvious way a $C^*$-algebra morphism $f_\Si:C^*(\Si,A)\to C^*(\Si,B)$. In  \cite[Section 2.3]{sty} was defined in this setting a natural transformation
$$\si_\Si:KK_*(A,B)\lto KK_*(C^*(\Si,A),C^*(\Si,B)),$$ for any $C^*$-algebras $A$ and $B$, which can be viewed as the geometrical analogue of the Kasparov transformation for crossed products. Let us give now the description of this map.

\medskip

Let $(\H\ts B,\pi, F)$ be a non-degenerated $K$-cycle for $KK_*(A,B)$. Define then $\widetilde{F}=Id_{\ell^2(\Si)\ts \H}\ts F$ acting on 
the Hilbertian right $B$-module $\ell^2(\Si)\ts \H\ts \H\ts B$.
The map $$\L_A(\ell^2(\Si)\ts \H\ts A)\lto \L_B(\ell^2(\Si)\ts \H\ts  \H \ts B);\,T \mapsto T\ts_\pi Id_{\H\ts B}$$ induces by restriction and under the identification between $\L_B(\ell^2(\Si)\ts \H\ts  \H \ts B)$ and $\L_{\K(\H)\ts B}(\ell^2(\Si)\ts \H\ts  \K(\H)\ts B)$ a morphism
$$\widetilde{\pi}:C^*(\Si,A)\to \M(C^*(\Si,B\ts \K(\H))),$$ where $\M(C^*(\Si,B\ts \K(\H)))$ stands for the multiplier algebra of 
$C^*(\Si,B\ts \K(\H))$. Then $(\M(C^*(\Si,B\ts \K(\H))),\widetilde{\pi},\widetilde{F})$ is a $K$-cycle for $KK_*(C^*(\Si,A),C^*(\Si,B\ts \K(\H)))$ and hence, under the identification between $C^*(\Si,B\ts \K(\H))$ and $C^*(\Si,B)$ we end up with an element in $KK_*(C^*(\Si,A),C^*(\Si,B))$. We obtain in this way a
natural transformation 
$$\si_\Si: KK_*(A,B)\lto KK_*(C^*(\Si,A),C^*(\Si,B)).$$  This transformation is also bifunctorial, i.e for any $C^*$-algebra morphisms 
$f:A_1\to A_2$ and $g:B_1\to B_2$ and any element $z$ in $KK_*(A_2,B_1)$, then we have $\si_\Si(f^*(z))=f_\Si^*(\si_\Si(z))$ and 
$\si_\Si(g_*(z))=g_{\Si,*}(\si_\Si(z))$. If $z$ is an element of $KK_*(A,B)$, we denote by
$$\SS_\Si(z):K_*(C^*(\Si,A))\lto K_*(C^*(\Si,B));\, x\mapsto x\ts_{C^*(\Si,A)} \si_\Si(z)$$ induced by right multiplication by   $\si_\Si(z)$.  

Notice that if 
\begin{equation}\label{semi-split-extension}
0\to J\to A\to A/J\to 0
\end{equation} is  a semi-split extension of $C^*$-algebra, then $C^*(\Si,J)$ can be viewed as an ideal of $C^*(\Si,A)$ and we get then a semi-split extension of $C^*$-algebras 
\begin{equation}\label{semi-split-extension-Roe}
0\lto C^*(\Si,J)\lto C^*(\Si,A)\lto C^*(\Si,A/J)\lto 0.\end{equation} If $z$ is the element of $KK_1(A,B)$ corresponding to the boundary element of the extension
 (\ref{semi-split-extension}), then $\SS_\Si(z):K_*( C^*(\Si,A/J))\lto K_{*+1}( C^*(\Si,J))$ is the boundary morphism  associated to the extension 
 (\ref{semi-split-extension-Roe}). 
 
 For a $C^*$-algebra $A$, let us denote by $SA$ its suspension, i.e $SA=C_0((0,1),A)$, by $CA$ its cone, i.e $CA=\{f\in C_0([0,1],A)\text{ such that} f(1)=0\}$ and by 
 $ev_0:CA\to A$   the evaluation map at zero.
 Let us consider for any $C^*$-algebra $A$ the Bott extension
 $$0\lto SA\lto CA\stackrel{ev_0}{\lto}A\lto 0,$$ with associated  boundary map 
 $\partial_A:K_*(A)\to K_{*+1}(SA)$. It is well known that the  corresponding  element $[\partial_A]$ of $KK_1(A,SA)$  is invertible with inverse up to Morita equivalence the element of $KK_1(\K\ts A,SA)$ corresponding to  the Toeplitz extension
 $$0\lto \K\ts A\lto \T_0\ts A\lto SA\lto 0.$$ 
 \begin{lemma}\label{lem-bott-coarse}
 For any $C^*$-algebra $A$, then $\SS_\Si([\partial_A]^{-1})$ is a left inverse for   $\SS_\Si([\partial_A])$.
 \end{lemma}
\begin{proof}
Consider the following commutative diagram with exact rows
$$\begin{CD}
0@>>> SC^*(\Si,A)@>>> CC^*(\Si,A)@>>>C^*(\Si,A)@>>>0\\
    @.     @V\jmath VV @VVV @VV=V \\
0@>>> C^*(\Si,SA)@>>> C^*(\Si,CA)@>>>C^*(\Si,A)@>>>0
\end{CD},$$
where
\begin{itemize}
\item the top row is the Bott extension for $C^*(\Si,A)$ with boundary map $$\partial_{C^*(\Si,A)}:K_*(C^*(\Si,A))\to K_{*+1}(SC^*(\Si,A));$$
\item the bottom row is the extension induced for Roe algebras by the Bott extension for $A$ with boundary map $$\SS_\Si([\partial_{A}]):K_*(C^*(\Si,A))\to K_{*+1}(C^*(\Si,SA));$$
\item the left and the middle vertical arrows are the obvious inclusions.
\end{itemize}
Consider similarly the commutative diagram
$$\begin{CD}
0@>>> C^*(\Si,A\ts\K)@>>> CC^*(\Si,A\ts \T_0)@>>>C^*(\Si,SA)@>>>0\\
    @.     @AAA @AAA @AA\jmath A \\
0@>>> C^*(\Si,A)\ts \K@>>> C^*(\Si,A)\ts \T_0@>>>SC^*(\Si,A)@>>>0
\end{CD},$$
\begin{itemize}
\item the bottom  row is the Toeplitz extension for $C^*(\Si,A)$.
\item the top row is the extension induced for Roe algebras by the Toeplitz  extension for $A$;
\item the left and the middle vertical arrows are the obvious inclusions.
\end{itemize}
By naturally of the boundary map in the first commutative diagram, we see that
$$\SS_\Si([\partial_{A}])=\jmath_*\circ \partial_{C^*(\Si,A)},$$ where $\jmath_*: K_{*}(SC^*(\Si,A) \to K_{*}(C^*(\Si,SA))$ is the map induced in $K$-theory by the inclusion $\jmath: SC^*(\Si,A) \hookrightarrow C^*(\Si,SA)$.
Using now the  naturally of the boundary map in the second  commutative diagram, we see that up to Morita equivalence,  
$\SS_\Si([\partial_{A}]^{-1})\circ\jmath_*$ is the boundary map for the Toeplitz  extension associated to $C^*(\Si,A)$ and hence is an inverse for $\partial_{C^*(\Si,A)}$. Therefore, $\SS_\Si([\partial_{A}]^{-1})$ is a left  inverse for $\SS_\Si([\partial_{A}])$.
\end{proof}
 The transformation $\SS_\Si$ is compatible with the Kasparov product in the following sense.
 \begin{proposition}\label{prop-prod-coarse}
 If $A,\,B$ and $D$ are separable $C^*$-algebras, let $z$ be an element in $KK_*(A,B)$ and let  $z'$ be an element in $KK_*(B,D)$. Then we have
 $$\SS_\Si(z\ts_B z')=\SS_\Si(z')\circ\SS_\Si(z).$$
 \end{proposition}
\begin{proof}
Assume first that $z$ is even.
Then according to  \cite[Theorem 1.6.11]{laff-inv}, there exist
\begin{itemize}
\item a $C^*$-algebra $A_1$;
\item  a morphism $\nu :A_1\to B$;
\item   a morphism $\theta :A_1\to A$ such that the associated element $[\theta]$ in $KK_*(A_1,A)$ is invertible,
\end{itemize}
such that
$z=\nu_*([\theta]^{-1})$ is invertible.
By bifunctoriality of the Kasparov product, we have,
\begin{eqnarray*}
z\ts_B z'&=&\nu_*([\theta]^{-1})\ts_B z'\\
&=& [\theta]^{-1}\ts_{A_1}\nu^*(z')
\end{eqnarray*}
Since $ \si_\Si$ and hence $\SS_\Si$ is natural, we see that $\SS_\Si([\theta]^{-1})$ is invertible, with inverse induced by 
$\theta_\Si:C^*(\Si,A_1)\to C^*(\Si,A)$.
Then using once again the naturally of $\SS_\Si$, we have 
\begin{eqnarray*}
\SS_\Si(z\ts_B z')\circ \theta_{\Si,*} &=& \SS_\Si(\nu^*(z'))\\
&=&\SS_\Si(z')\circ  \nu_{\Si,*}\\
&=&\SS_\Si(z') \circ \nu_{\Si,*}\circ \SS_\Si([\theta]^{-1})\circ \theta_{\Si,*}\\
&=&\SS_\Si(z') \circ  \SS_\Si(\nu_*([\theta]^{-1}))\circ \theta_{\Si,*}\\
&=&\SS_\Si(z')\circ \SS_\Si(z)\circ \theta_{\Si,*}
\end{eqnarray*}
Since  $\theta_{\Si,*}$ is invertible, we deduce that $\SS_\Si(z\ts_B z')=\SS_\Si(z')\circ\SS_\Si(z)$.
If $z'$ is even, we proceed similarly.
\smallskip

If $z$ and $z'$ are both odd. Let $[\partial_B]$ be the element of $KK_1(B,SB)$ corresponding to the boundary morphism $\partial_B:K_*(B)\to K_{*+1}(SB)$ associated to the Bott extension $0\to SB\to CB\to B\to 0$.
Then 
\begin{eqnarray*}
\SS_\Si(z\ts_B z')&=&\SS_\Si(z\ts_B [\partial_B] \ts_{SB} [\partial_B]^{-1}\ts_B z')\\
&=&\SS_\Si ([\partial_B]^{-1}\ts_B z')\circ \SS_\Si(z\ts_B [\partial_B])  \\
&=&\SS_\Si ([\partial_B]^{-1}\ts_B z')\circ \SS_\Si(  [\partial_B])\circ \SS_\Si(  [\partial_B]^{-1})\circ \SS_\Si(z\ts_B [\partial_B])  \\
&=&\SS_\Si(z')    \circ  \SS_\Si (z).
\end{eqnarray*}
where,
\begin{itemize}
\item the second  and the fourth equalities hold  by the even cases;
\item the  third equality is a consequence of lemma \ref{lem-bott-coarse}
\end{itemize}
\end{proof}

Now let $\Si$ be a discrete  metric space with bounded geometry. Let $\H$ be a separable Hilbert space and fix a unit vector $\xi_0$ in $\H$. For any positive number $s$, let
$Q_{s,\Si}$ be the operator of $\L_{C_0(P_s(\Si))}(C_0(P_s(\Si))\ts\ell^2(\Si)\ts\H)$ defined by
$$(Q_{s,\Si}\cdot h)(x,\si)=\lambda_{\si}^{1/2}(x)\sum_{\si'\in\Si}\lambda_{\si'}^{1/2}(x)\langle h(x,\si'),\xi_0\rangle\xi_0,$$where $h$ in 
$C_0(P_s(\Si))\ts\ell^2(\Si)\ts\H$ is viewed as function on $P_s(\Si)\times\Si$ with values in $\H$ (recall that $(\lambda_\si)_{\si\in\Si}$ is the family of coordinate functions in $P_s(\Si)$).Then $Q_{s,\Si}$ is a projection of $C^*(\Si,C_0(P_s(\Si)))$.
Let $B$ be a $C^*$-algebra. Then the bunch of maps   
$$\mu^{s}_{\Si,B,*}:KK_*(P_s(\Si),B)\lto K_*(C^*(\Si,B));\, z\mapsto [Q_{s,\Si}]\ts_{C^*(\Si,C_0(P_s(\Si)))}\si_\Si(z)$$
is  compatible with the maps $K_*(P_s(\Si))\to K_*(P_{s'}(\Si))$ induced by the inclusion of Rips complexes $P_s(\Si)\hookrightarrow P_{s'}(\Si)$. Taking the inductive limit, we end with the coarse Baum-Connes assembly map with coefficients in $B$
 $$\mu_{\Si,B,*}:\lim_s KK_*(P_s(\Si),B)\lto K_*(C^*(\Si,B)).$$
 If $\mu_{\Si,B,*}$ is an isomorphism, we say that $\Si$ satisfies the coarse Baum-Connes conjecture with coefficients in $B$. When  $B=\C$, we set $\mu^s_{\Si,*}$ for $\nu^s_{\Si,\C,*}$, $\nu_{\Si,*}$ for $\mu_{\Si,\C,*}$ and we say  that 
  $\Si$ satisfies the coarse Baum-Connes conjecture if   $$\mu_{\Si,*}:\lim_s K_*(P_s(\Si))\lto K_*(C^*(\Si))$$ is an isomorphism. Recall that if $\Ga$ is a finitely generated group, and if $|\Ga|$ stands for the metric space arising from any word metric, then the coarse Baum-Connes conjecture for $|\Ga|$ implies the Novikov conjecture on higher signatures for the group $\Ga$.

\subsection{A geometric assembly map  for families of finite metric spaces}
To prove  theorem \ref{thm-CBC}, we will need a slight modification  of the map  $\nu_{F,\Si,\A,*}^{\infty}$ defined by  equation (\ref{equ-Ainf}).
Let $\A=(A_i)_{i\in\N}$ be a family of $C^*$-algebras and let $\X=(X_n)_{n\in\N}$ be a family of discrete proper metric spaces.
Define $\A^\infty_\X$ as the closure of the set of 
$x=(x_n)_{n\in\N}\in \prod_{n\in\N} A_n \ts \K(\ell^2(X_n)\ts\H)$  such that  for some $r>0$ then $x_n$ has propagation less than $r$ for all integer $n$. Then $\A^\infty_\X$ is obviously a filtered $C^*$-algebra.  When $\A$ is the constant family $A_i=\C$, then we set $C^*(\X)$ for $\A^\infty_\X$. According to lemma
\ref{lem-prod-filtered}, there exists for a universal control pair $(\alpha,h)$, any family $\A=(A_i)_{i\in\N}$  of $C^*$-algebras and any family  $\X=(X_n)_{n\in\N}$  of discrete proper metric spaces a $(\alpha,h)$-controlled isomorphism
$$\K_*(\A^\infty_\X)\lto\prod_{i\in\N}\K_*( A_n \ts \K(\ell^2(X_n)))$$ induced on the $j$ factor and up to Morita equivalence by the restriction to 
$\A^\infty_\X$ of the evaluation  $$\prod_{n\in\N} A_i \ts \K(\ell^2(X_n)\ts\H)\lto A_n \ts \K(\ell^2(X_j)\ts\H).$$
Proceeding as un corollary \ref{cor-tensor-infty}, we see that there exists  a universal control pair $(\alpha,h)$ such that
\begin{itemize}
\item  For any family $\X=(X_n)_{n\in\N}$ of finite metric space;
\item for any families of $C^*$-algebras $\A=(A_i)_{i\in\N}$ and $\B=(B_i)_{i\in\N}$;
\item  for any $z=(z_i)_{i\in\N}$ in $\prod_{i\in\N}KK_*(A_i,B_i)$, \end{itemize}
there 
exists a $(\alpha,h)$-controlled morphism
 $$\TT^\infty_{\X}(z)=(\tau_{\X}^{\infty,\eps,r})_{0<\eps<\frac{1}{4\alpha},r>0}:\K_*(\A^\infty_{\X})\to 
\K_*(\B^\infty_{\X})$$ that satisfies in this setting the analogous properties as  those listed in corollary \ref{cor-tensor-infty} and proposition \ref{prop-product-tensor-infty} for $\TT^\infty_{F,\Si}(\bullet)$. 
Consider now for $Z$ a finite metric space and $s$ a positive number the projection $Q_{s,Z}$ of $C(P_s(Z))\ts \K(\ell^2(Z))$ defined by
 $$Q_{s,Z}(h)(y,z)=\lambda_z^{1/2}(y)\sum_{z'\in Z}h(y,z')\lambda_{z'}^{1/2}(y)$$ for any $h$ in 
 $C(P_s(Z))\ts \ell^2(Z)\cong C(P_s(Z),\ell^2(Z))$ where $(\lambda_z)_{z\in Z}$ is the family of coordinate functions  of 
$P_s(Z)$, i.e $y=\sum_{z\in Z}\lambda_z(y)$ for any $y$ in $P_s(Z)$. Then $Q_{s,Z}$ has propagation less than $s$. If we fix any rank one projection $e$ in $\K(H)$, for any family $\X=(X_i)_{i\in\N}$ of finite metric spaces, then $Q^\infty_{s,\X}=(Q_{s,X_i}\ts e)_{i\in\N}$ is a projection of propagation less than $s$ in
$\A^\infty_\X$, where $\A$ is the family $(C(P_s(X_i)))_{i\in\N}$.

Now we can proceed as in section \ref{subsection-another-assembly-map} to define a quantitative geometric  assembly map valued in $C^*(\X)$.
For any  
$\eps$ in  $(0,1/4)$, any positive numbers $s$ and $r$  such that  $r\gq  r_{d, \eps}$, define

$$\nu_{\X,*}^{\infty,\eps,r,s}: \prod_{i\in\N}K_*(P_s(\X)){\longrightarrow}
K_*^{\eps,r}(C^*(\X));\,  z\mapsto 
\tau_{\X}^{\infty,\eps/\alpha,r/h_{\eps/\alpha}}(z)([Q^\infty_{s,\X},0]_{\eps/\alpha,r/h_{\eps/\alpha}}).$$ The bunch of maps $(\mu_{\X,*}^{\infty,\eps,r,s})_{0<\eps<1/4, r> r_{s, \eps}}$ is obviously compatible with the structure maps of $\K_*(C^*(\X))$, i.e $\iota_*^{\eps,\eps',r,r'}\circ \nu_{\X,*}^{\infty,\eps,r,s}= \nu_{\X,*}^{\infty,\eps',r',s}$ for $0<\eps\lq\eps'<1/4$ and $r_{s, \eps}<r\lq r'$. This allows to define 
$$\nu_{\X,*}^{\infty,s}: \prod_{i\in\N}K_*(P_s(\X)){\longrightarrow}
K_*(C^*(\X))$$ as  $\nu_{\X,*}^{\infty,s}=\iota_*^{\eps,r}\circ \nu_{\X,*}^{\infty,\eps,r,s}$.
The quantitative assembly maps are also compatible with inclusion of Rips complexes. 
 Let \begin{equation}
  q_{s,s',*}^\infty: \prod_{i\in\N}K_*(P_s(X_i)){\longrightarrow} \prod_{i\in\N}K_*(P_s(X_i))
 \end{equation}
be the map induced by the bunch of inclusions
 $P_s(X_i)\hookrightarrow  P_{s'}(X_i)$, then we have
 $$\nu_{\X,*}^{\infty,\eps,r,s'}\circ q_{s,s',*}^\infty=\nu_{\X,*}^{\infty,\eps,r,s}$$ for any
 positive numbers $\eps,\,s,\,s,'$ and $r$  such that 
 $\eps\in(0,1/4),\,s\lq s',\, r\gq r_{s',\eps}$, and thus
 $$\nu_{\X,*}^{\infty,s'}\circ q_{s,s',*}^\infty=\nu_{\X,*}^{\infty,s}$$ for any
 positive numbers $s$ and $s'$  such that 
$s\lq s'$. 

Let $\Si$ be a graph space in the sense of \cite{wy} i.e $\Si=\coprod_{i\in\N} X_i$, where $(X_i)_{i\in\N}$ is a family of finite metric spaces such that
\begin{itemize}
\item For any $r>0$, there exists an integer $N_r$ such that for any integer $i$, any ball of radius $r$ in $X_i$ has at most  $N_r$ element;
\item The distance between $X_i$ and $X_j$ is at least  $i+j$ for any distinct integers $i$ and $j$;
\end{itemize}
If $\X_\Si$ stands for the family $(X_i)_{i\in\N}$, we obviously have an inclusion of filtered $C^*$-algebras $\jmath_\Si:C^*(\X_\Si)\hookrightarrow C^*(\Si)$.
\begin{proposition}\label{proposition-assembly-map-graph-space}
 Let $\Si$ be a graph space   $\Si=\coprod_{i\in\N} X_i$ as above and let $s$ a positive number such that $d(X_i,X_j)>s$ if $i\neq j$. Then we have a commutative diagram
$$\begin{CD}
 \prod_{i\in\N}K_*(P_s(X_i))@> \nu_{\X,*}^{\infty,s}>>K_*(C^*(\X_\Si))\\
         @VV\simeq V
         @VV\jmath_{\Si,*} V\\
K_*(P_s(\Si))@>\mu^s_{\Si,*}>> K_*(C^*(\Si))
\end{CD},$$ where
 in view of the equality  $P_s(\Si)=\coprod_{i\in\N}P_s(X_i)$, the   left vertical map is  the identification   between 
 $\prod_{i\in\N}K_*(P_s(X_i))$ and $K_*(\coprod_{i\in\N} P_s(X_i))$;
\end{proposition}
The proof of this proposition will require some preliminary steps.
If $\A=(A_i)_{i\in\N}$ is a family of $C^*$-algebras, we set $\Ap=\oplus_{n\in\N}A_i$. The orthogonal family  $(A_i\ts\K(\ell^2(X_i)\ts\H))_{i\in\N}$ of  corners 
in $\Ap\ts\K(\ell^2(\Si)\ts\H)$ gives rise to a  one-to-one morphism $\jmath_{\A,\Si}:\A_{\XS}^\infty\lto C^*(\Si,\Ap)$. Let $z=(z_i)_{i\in\N}$ be a family in
 $\prod_{i\in\N}KK_*(A_i,\C)$. Recall that we have a canonical identification between $\prod_{i\in\N}KK_*(A_i,\C)$ and $KK_*(\Ap,\C)$. 
 Let $\widetilde{z}$ be the element of $KK_*(\Ap,\C)$ corresponding to $z$ under this identification.
 \begin{lemma}\label{lemma-preliminary}
 For any family $\A=(A_i)_{i\in\N}$ of $C^*$-algebras, any graph space $\Si=\coprod_{i\in\N} X_i$ and  any $z$ in $\prod_{i\in\N}KK_*(A_i,\C)$, then we have a commutative 
 diagram 
 $$\begin{CD}
 K_*(\A^\infty_\Si)@>\T_\XS^\infty(z)>>K_*(C^*(\X_\Si))\\
         @V\jmath_{\A,\Si}VV
         @VV\jmath_{\Si,*} V\\
K_*(C^*(\Si,\Ap))@>\SS_\Si>> K_*(C^*(\Si))
\end{CD},$$
\end{lemma}
\begin{proof}
Assume first that $z$ is odd. Let us fix a separable Hilbert space $\H$. For each integer $i$, let $(\H,\pi_i,T_i)$ be the $K$-cycle for  $KK_*(A_i,\C)$ representing $z_i$ with $\pi_i:A_i\to \L(\H)$ a representation and $T_i$ in $\L(\H)$ representing the $K$-cycle conditions. Let us set $P_i=\frac{T_i+\Id_\H}{2}$ and 
$$E_i=\{(x,T)\in A_i\oplus \L(\H)\text{ such that } P_i\pi_i(x)P_i-T\in\K(\H)\}.$$
We have an inclusion $$\K(H)\hookrightarrow E_i;\,T\mapsto (0,T)$$ as an ideal  and a surjection $$E_i\to A_i;\, (x,T)\to x.$$
Up to Morita equivalence, $z_i$ induces by left multiplication the boundary morphism of the semi-split extension
$$0\lto \K(\H)\lto E_i\lto A_i\lto 0,$$
Let $\mathcal{E}$ be the family $(E_i)_{i\in N}$ and set $\CH$ for the constant family $\K(\H)$.
Then the extension

$$0\mapsto \prod_{i\in \N}\K(\H)\ts \K(\ell^2(X_i)\ts \H)\lto \prod_{i\in \N} E_i\ts \K(\ell^2(X_i)\ts \H)\lto   \prod_{i\in \N}A_i\ts \K(\ell^2(X_i)\ts \H)\lto 0$$ restrict to a semi-split  extension of filtered $C^*$-algebras
$$0\mapsto \CH^\infty_\XS\lto \E^\infty_\XS\lto  \A^\infty_\XS\lto 0.$$ Up to the identification between
$K_*( \CH^\infty_\XS)$ and $K_*( C^*(\XS))$ arising from   Morita equivalence between $\C$ and $\K(\H)$, the boundary morphism  associated to this extension is  $\T_\XS^\infty(z):K_*(\A^\infty_\XS)\lto K_{*+1}( C^*(\XS))$.
In the same way, let $$E=\{((x_i)_{i\in\N},T)\in \left(\oplus_{n\in\N}A_i\right)\oplus \L(\ell^2(\N,\H))\text{ such that }\left(\oplus_{i\in\N}p_i\pi_i(x_i)p_i\right)-T\in \K(\ell^2(\N,\H))\}.$$
As before we have a semi-split  extension
\begin{equation}\label{equation-coarse}0\lto  \K(\ell^2(\N)\ts \H)\lto E\lto \Ap\lto 0\end{equation} and 
$\SS_\Si(\widetilde{z}):K_*(C^*(\Si,\Ap))\to K_{*+1}(C^*(\Si))$ is up to the identification between $K_{*}(C^*(\Si))$ and 
$K_*(C^*(\Si,\K(\ell^2(\N)\ts \H)))$ arising from Morita equivalence is the boundary morphism for the extension
$$0\lto  C^*(\Si,\K(\ell^2(\N)\ts \H))\lto C^*(\Si,E)\lto C^*(\Si,\Ap)\lto 0$$ induced by the extension of equation (\ref{equation-coarse}). For  every integer $i$, there is an obvious representation of
$\K(\H\ts\ell^2(X_i))\ts E_i$ on the right $E$-Hilbert module $\H\ts\ell^2(\Si)\ts E$ as a corner which gives rise to a $C^*$-morphism
$\jmath'_{\E,\Si}:\E_{\XS}^\infty\to C^*(\Si,E)$ such that $\jmath'_{\E,\Si}( \CH^\infty_\XS)\subseteq C^*(\Si,\K(\ell^2(\N)\ts \H))$. We have then a commutative diagram
\begin{equation}\label{CD-coarse}\begin{CD}
0@>>> \CH^\infty_\XS@>>> \E^\infty_\XS@>>>\A^\infty_\XS@>>>0\\
    @.     @V\jmath'_{\E,\Si}VV @V\jmath'_{\E,\Si}VV @VV\jmath_{\A,\Si}V \\
0@>>>C^*(\Si,\K(\ell^2(\N)\ts \H))@>>>  C^*(\Si,E)@>>> C^*(\Si,\Ap) @>>>0
\end{CD},\end{equation}
The restriction  morphism $\CH^\infty_\XS \stackrel{\jmath'_{\E,\Si}}{\lto}C^*(\Si,\K(\ell^2(\N)\ts \H))$ is homotopic to the composition
\begin{equation}\label{equ-composition-jmath}
\CC\H^\infty_\XS {\lto}C^*(\Si,\K(\H)){\lto}C^*(\Si,\K(\ell^2(\N)\ts \H))
\end{equation}
where,
\begin{itemize}
\item the first map is induced by the obvious representation of $\CC\H^\infty_\XS=\prod_{i=1}^\infty \K(\H\ts\ell^2(X_i))\ts\K(\H)$ on the $\K(\H)$-right Hilbert 
module $\H\ts\ell^2(\Si)\ts\K(\H)$ (each factor acting as a corner);
\item the second map is induced by the morphism $$\K(\H)\to\K(\ell^2(\N)\ts\H);\, x\mapsto x\ts e,$$ where $e$ is any rang one projection in 
$\K(\ell^2(\N))$.
\end{itemize}
But up to the identification on one hand between $K_*(\CC\H^\infty_\XS)$ and $K_*(C^*(\XS))$, and on the other hand between $K_*(C^*(\Si,\K(\ell^2(\N)\ts \H)))$ and  $K_*(C^*(\Si))$,  the morphism of equation (\ref{equ-composition-jmath})  induces in $K$-theory $\jmath_{\Si,*}:K_*(C^*(\XS))\to K_*(C^*(\Si)$. Since in the commutative diagram
(\ref{CD-coarse}),  $\SS_\Si(\widetilde{z})$  is the boundary morphism associated to the top   row
and  $\T_\XS^\infty(z) $ is the boundary morphism associated to the bottom  row, the lemma in the odd case is then a consequence of the naturally of the boundary morphisms.

\smallskip
If $z$ is even, set $[\partial_\A]=([\partial_{A_i}])_{i\in\N}\in \prod_{i\in\N}KK_*(A_i,SA_i)$ and $[\partial_\A]^{-1}=([\partial_{A_i}]^{-1})_{i\in\N}\in \prod_{i\in\N}KK_*(SA_i,A_i)$.     Let us also define the families $\mathcal{S}\A=\prod_{i\in\N} SA_i$ and $\mathcal{C}\A=\prod_{i\in\N} CA_i$ and set $z'=([\partial_{A_i}]^{-1}\ts_{A_i}z_i)_{i\in\N}\in \prod_{i\in\N}KK_*(SA_i,\C)$.
Using the odd case  and the compatibility of the transformation $\T_\XS^\infty(\bullet)$ with Kasparov products, we get that
\begin{equation}\label{equ-proof-even-case}
\jmath_{\Si,*}\circ\T_\XS^\infty(z)=\jmath_{\Si,*}\circ\T_\XS^\infty(z')\circ \T_\XS^\infty([\partial_\A])=\SS_\Si(\widetilde{z'})\circ \jmath_{\Si,\mathcal{S}\A,*}\circ \T_\XS^\infty([\partial_\A])
\end{equation}

Under the canonical identifications $(\mathcal{S}\A)^\oplus \simeq  S\A^\oplus$ and  $(\mathcal{C}\A)^\oplus \simeq  C\A^\oplus$, we have a commutative diagram
$$\begin{CD}\label{commutative-diagram-coarse}
0@>>>\mathcal{S}\A^\infty_\XS  @>>>\mathcal{C}\A^\infty_\XS@>>>\A^\infty_\XS@>>>0\\
    @.     @V\jmath_{\mathcal{S}\A,\Si}VV @V\jmath_{\mathcal{C}\A,\Si}VV @VV\jmath_{\A,\Si}V \\
0@>>>C^*(\Si,S\Ap )@>>>  C^*(\Si
,C\Ap )@>>> C^*(\Si,\Ap) @>>>0
\end{CD},$$
where the row both arise from the family of Bott extensions
$$(0\to SA_i\to CA_i\to A_i\to 0)_{i\in\N}.$$
Then 
\begin{itemize}
\item $\T_\XS^\infty([\partial_\A]):K_*(\mathcal{S}\A^\infty_\XS)\to K_{*+1}( \A^\infty_\XS)$ is the boundary morphism for the top row;
\item $\SS_\Si([\partial_{\A^\oplus}]):K_*(C^*(\Si, \A^\oplus)\to K_{*+1}( C^*(\Si, S\A^\oplus)$ is the boundary morphism for the bottom  row;
\end{itemize}
By naturally of the boundary extension, we get that
$$\jmath_{\Si,\mathcal{S}\A,*}\circ \T_\XS^\infty([\partial_\A])=\SS_\Si([\partial_{\A^\oplus}])\circ \jmath_{\A,\Si,*}.$$
Hence, using proposition \ref{prop-prod-coarse}, we deduce from equation (\ref{equ-proof-even-case}) that 
$$\jmath_{\Si,*}\circ\T_\XS^\infty(z)=  \SS_\Si([\partial_\Ap]\ts_\Ap\widetilde{z'})\circ \jmath_{\A,\Si,*}.$$
But using the Connes-Skandalis characterization of Kasparov products, we get that 
$[\partial_\Ap]\ts_\Ap\widetilde{z'}=\widetilde{z}$ and hence $\jmath_{\Si,*}\circ\T_\XS^\infty(z)=   \SS_\Si(\widetilde{z})\circ \jmath_{\A,\Si,*}$.

\end{proof}

{\it Proof of proposition \ref{proposition-assembly-map-graph-space}:}
Let $z=(z_i)_{i\in\N}$ be a family in $\prod_{i\in\N}K_*(P_s(X_i))=\prod_{i\in\N}KK_*(C(P_s(X_i)),\C)$. Then under the identification between $\prod_{i\in\N}KK_*(C(P_s(X_i)),\C)$ and $KK_*(C_0(P_s(\Si)),\C)$ given by the equality 
$C_0(P_s(\Si))=\oplus_{i\in\N} C(P_s(X_i))$, we have a correspondence between $z$ and $\widetilde{z}$ and hence the commutativity of the diagram amounts to prove the equality 
\begin{equation}\label{equ-proof-proposition-assembly-map-graph-space}
\SS_\Si(\widetilde{z})([Q_{s,\Si},0])=\jmath_{\Si,*}\circ \T_{\XS}^\infty(z)([Q_{s,\X_\Si}^\infty,0]).\end{equation}
Let us consider the family $\A=(C(P_s(X_i)))_{n\in\N}$. Since $d(X_i,X_j)\geq s$ if $i\neq j$, we see that 
$\jmath_{\A,\Si}(Q_{s,\X_\Si}^\infty)=Q_{s,\Si}$ and hence 
$$\SS_\Si(\widetilde{z})([Q_{s,\Si},0])=\SS_\Si(\widetilde{z})\circ \jmath_{\A,\Si,*}([Q^\infty_{s,\XS},0]).$$
The equality (\ref{equ-proof-proposition-assembly-map-graph-space}) is then a consequence of lemma \ref{lemma-preliminary}.
\qed
\subsection{Proof of theorem \ref{thm-CBC}}
Let $\Si$ be a discrete metric space with bounded geometry that satisfies the assumptions of theorem \ref{lemma-preliminary}. According to \cite{wy}, we can assume by using a coarse Mayer-Vietoris argument  that $\Si$ is a graph space $\Si=\coprod_{i\in\N}X_i$.

\medskip

Let us show that $\mu_{\Si,*}$ is one-to-one. Let $d$ be a positive number and let $x$ be an element in $K_*(P_d(\Si))$ such that 
$\mu_{\Si,*}(x)=0$. Fix $\eps>0$ small enough and choose a positive number $\lambda$ as in \cite[Remark 1.18]{oy2}. We can assume without loss of generality that $d(X_i,X_j)\gq d$ if $i\neq j$. Then $P_d(\Si)=\coprod_{i\in\N} P_d(X_i)$  and  up to the corresponding  identification between  $K_*(P_d(\Si))$ and 
$\prod_{i\in\N}K_*(P_d(X_i))$, we can view $x$ as a family  $(x_i)_{i\in\N}$ in  $\prod_{i\in\N}K_*(P_d(X_i))$. According to proposition  \ref{proposition-assembly-map-graph-space}, we get that $$\jmath_{\Si,*}\circ \nu_{\XS,*}^\infty(x)=0.$$
If we fix $r\gq r_{d,\eps}$, then we have
\begin{eqnarray*}
\jmath_{\Si,*}\circ \nu_{\XS,*}^\infty&=&\jmath_{\Si,*}\circ \iota_*^{\eps,r}\circ \nu_{\XS,*}^{\infty,\eps,r}\\
&=& \iota_*^{\eps,r}\circ \jmath^{\eps,r}_{\Si,*}\circ\nu_{\XS,*}^{\infty,\eps,r},
\end{eqnarray*}
and hence according to proposition \ref{proposition-approximation},
 there exists $r'\gq r$ such that $$\jmath^{\lambda\eps,r'}_{\Si,*}\circ\nu_{\XS,*}^{\infty,\lambda\eps,r'}(x)=0.$$  Therefore, replacing $\lambda\eps$ by $\eps$ and $r'$ by $r$, we see that there exists $\eps$ in $(0,1/4)$ and $r$ a positive number such that $\jmath^{\eps,r}_{\Si,*}\circ\nu_{\XS,*}^{\infty,\eps,r}(x)=0.$
We can also assume without loss of generality that $d(X_i,X_j)\gq r$ if $i\neq j$ and hence $\mu_{\XS,*}^{\infty,\eps,r}(x)=0$
in $K_*^{\eps,r}(C^*(\XS))$. Using the control isomorphism between $\K_*(C^*(\XS))$ and $\prod_{i\in\N} \K_*(\K(\ell^2(X_i))$, we see that up to rescale $\eps$ and $r$, we can assume that $\mu_{X_i,*}^{\eps,r}(x_i)=0$ in $K_*^{\eps,r}(\K(\ell^2(X_i))$ for every integer $i$. Let then $d'\gq d$ such that $QI_{F,*}(d,d',r,\eps)$ is satisfied for every finite subset $F$ of $\Si$. We have then $q_{d,d',*}(x_i)=0$ in $K_*(P_{d'}(\X_i))$ for every integer $i$ and  therefore $q_{s,s',*}(x)=0$ in $K_*(P_{d'}(\Si))$. Hence $\mu_{\Si,*}$ is one-to-one.

\medskip

Let us prove that  $\nu_{\Si,*}$ is onto. Let $z$ be an element in $K_*(C^*(\Si))$ and fix $\eps'$ small enough. Then for
some positive number $r'$, there exists $y'$ in $K_*^{\eps',r'}(C^*(\Si))$ such that $z=\iota_*^{\eps',r'}(y')$. Pick $\eps$ in $[\eps',1/4)$, $d$ a positive number  and $r\gq r'$ such that $QS_{F,*}(d,r,r',\eps,\eps')$ holds for any finite subset $F$ of $\Si$. We can assume without loss of generality that $d(X_i,X_j)>r$ and $d(X_i,X_j)>d$ if $i\neq j$. Then there exist an element $y$ in $K^{\eps',r'}_*(C^*(\XS))$ such that $\jmath_{\Si,*}^{\eps',r'}(y)=y'$. For every integer $i$, let $y_i$ be the image of $y$ under the composition $$K_*^{\eps',r'}(C^*(\XS))\to K_*^{\eps',r'}(\K(\ell^2(X_i)\ts \H))\to K_*^{\eps',r'}(\K(\ell^2(X_i))),$$
where
\begin{itemize}
\item the first morphism is induced by the restriction to $C^*(\XS)$ of the $i$ th projection 
$\prod_{n\in\N}\K(\ell^2(X_n)\ts \H)\to \K(\ell^2(X_i)\ts \H)$;
\item the second morphism is the Morita equivalence.
\end{itemize}
For every integer $i$, there exists $x_i$ in $K_*(P_d(X_i))$ such that $\mu_{X_i,*}^{\eps,r,d}(x_i)=\iota_*^{\eps',\eps,r,r'}(y_i)$. Set then $x=(x_i)_{i\in\N}$ in $\prod_{i\in\N}K_*(P_d(X_i))$. Then 
$\nu_{\XS,*}^{\infty,d}(x)=\iota^{\eps,r}_*(y)$ and hence according to proposition \ref{proposition-assembly-map-graph-space} and under the identification between $K_*(P_d(\Si))$ and 
$\prod_{i\in\N}K_*(P_d(X_i))$, we get that $\mu_{\Si,*}^{d}(x)=\jmath_{\Si,*}(\iota^{\eps,r}_*(y))=\iota^{\eps,r}_*(y')=z$. 
Hence $\mu_{\Si,*}$ is onto.

\section*{Acknowledgements}
The first author would like to thank the Shanghai Center for Mathematical Sciences and Texas A\&M University 
for financial support and hospitality and the second author would like to thank University of Metz for financial support and hospitality. The two authors  would also like to thank  Dan Burghelea for suggesting the name of the persistence approximation property.

{\small
\bibliographystyle{plain}
\bibliography{persistence}
}
\end{document}